\documentclass[11pt]{article}
\usepackage{amsmath} \usepackage{amssymb} \usepackage{latexsym} \usepackage{enumitem} \usepackage{mathrsfs} \usepackage{comment}
\usepackage{color}
\usepackage{accents}
\usepackage{bbm}

\usepackage[colorlinks=true,urlcolor=blue,
citecolor=red,linkcolor=blue,linktocpage,pdfpagelabels,bookmarksnumbered,bookmarksopen]{hyperref}

\newtheorem{assumption}{Assumption}

\def\qed{ \ \vrule width.2cm height.2cm depth0cm\smallskip}
\newenvironment{proof}{\noindent {\bf Proof.\/}}{$\qed$\vskip 0.1in}

\newcommand{\ba}{\begin{array}}
\newcommand{\ea}{\end{array}}
\newcommand{\be}{\begin{equation}}
\newcommand{\ee}{\end{equation}}
\newcommand{\bea}{\begin{eqnarray}}
\newcommand{\eea}{\end{eqnarray}}
\newcommand{\beaa}{\begin{eqnarray*}}
\newcommand{\eeaa}{\end{eqnarray*}}

\def\dbE{\mathbb{E}}
\def\dbF{\mathbb{F}}

\def\dbL{\mathbb{L}}
\def\dbM{\mathbb{M}}
\def\dbN{\mathbb{N}}
\def\dbP{\mathbb{P}}
\def\dbR{\mathbb{R}}

\def\dbX{\mathbb{X}}

\def\a{\alpha}
\def\b{\beta}
\def\g{\gamma}
\def\d{\delta}
\def\e{\varepsilon}

\def\k{\kappa}

\def\si{\sigma}

\def\f{\varphi}
\def\th{\theta}
\def\o{\omega}

\def\G{\Gamma}
\def\D{\Delta}
\def\Th{\Theta}

\def\Si{\Sigma}

\def\O{\Omega}
\def\cA{{\cal A}}
\def\cB{{\cal B}}
\def\cC{{\cal C}}

\def\cF{{\cal F}}
\def\cG{{\cal G}}
\def\cH{{\cal H}}
\def\cI{{\cal I}}

\def\cL{{\cal L}}
\def\cM{{\cal M}}

\def\cP{{\cal P}}

\def\cU{{\cal U}}
\def\cV{{\cal V}}

\def\cX{{\cal X}}
\def\cY{{\cal Y}}
\def\cZ{{\cal Z}}

\def\no{\noindent}

\def\ms{\medskip}

\def\q{\quad}
\def\qq{\qquad}

\def\pa{\partial}
\def\cd{\cdot}
\def\cds{\cdots}

\def\td{\nabla}

\def\tr{\hbox{\rm tr}}

\def\qed{ \hfill \vrule width.25cm height.25cm depth0cm\smallskip}

\newcommand{\basa}{\begin{assumption}}
\newcommand{\easa}{\end{assumption}}

\newcommand{\bas}{\begin{assum}}
\newcommand{\eas}{\end{assum}}

\def\limsup{\mathop{\overline{\rm lim}}}

\def\pa{\partial}

 \def\cd{\cdot}
\def\cds{\cdots}

\def\supp{\hbox{\rm supp$\,$}}

\def\tr{\hbox{\rm tr$\,$}}

\def\dis{\displaystyle}

\def\bx{{\bf x}}

\def\1{{\bf 1}}

\def\:{\!:\!}
\def\reff{\eqref}
\def \proof{{\noindent \bf Proof.\quad}}

\def\bm{\textbf{m}}

at 9pt

\definecolor{alp}{rgb}{0.0, 0.5, 0.0}

\newtheorem{thm}{Theorem}[section]

\newtheorem{rem}[thm]{Remark}

\newtheorem{defn}[thm]{Definition}
\newtheorem{assum}[thm]{Assumption}

\begin{document}

\title{\bf  Local Well-posedness of General Mean Field Game Master Equations}
\author{Chenchen Mou\thanks{\noindent  Department of Mathematics,
City University of Hong Kong. E-mail: \href{mailto:chencmou@cityu.edu.hk}{chencmou@cityu.edu.hk}. This author is supported in part by NSFC grant 12522122, NSFC/RGC JRS N\_CityU165/25, GRF 11311422, GRF 11303223 and GRF 11305726.}, ~ Jianfeng Zhang\thanks{\noindent  Department of Mathematics,
University of Southern California. E-mail:
\href{mailto:jianfenz@usc.edu}{jianfenz@usc.edu}. This author is supported in part by NSF grant DMS-2510403.
} ~ and ~
Jianjun Zhou \thanks{\noindent  College of Science,
Northwest A\&F University. E-mail:
\href{mailto:jianfenz@usc.edu}{zhoujianjun@nwsuaf.edu.cn}. This author is supported in part by  Shaanxi Natural Science Foundation Grant 2025JC-YBMS-021,  
Shaanxi Fundamental
Science Research Project for Mathematics and Physics Grant 25JSY042, and NSFC grant 12671546.
}
}
\date{}
\maketitle

\begin{abstract}
This paper presents a  generic approach for establishing  mean field game master equations, applicable whenever the mean field equilibrium can be characterized by a McKean-Vlasov forward-backward stochastic differential equation system. The core of our approach is a representation formula for the first-order Lions derivative of the  decoupling field of this forward-backward SDE system. We then employ a bootstrap argument to recursively compute its higher-order derivatives. To demonstrate the method's versatility, we establish the local well-posedness for master equations in three distinct models: extended mean field games, mean field games with volatility control, and mean field games with a major player.
\end{abstract}

\no{\bf Keywords.}  Mean field games, master equation, forward backward SDEs, local well-posedness

\ms
\no{\it 2020 AMS Mathematics subject classification:}  35Q89, 49N80, 60H30, 91A16

\vfill\eject


\section{Introduction}
\label{sect-Introduction}
\setcounter{equation}{0}

The field of Mean Field Games (MFGs), launched independently in the seminal works of Caines-Huang-Malhamé \cite{HCM06} and Lasry-Lions \cite{LL07a}, has witnessed profound development over the past two decades. This framework provides a powerful mathematical paradigm for analyzing the collective behavior of large populations of rational and strategic agents. Foundational aspects of the theory were systematized in the pioneering lectures of Lions \cite{Lions} and the notes of Cardaliaguet \cite{Cardaliguet}, with more comprehensive developments detailed in the monographs of Carmona-Delarue \cite{CD1, CD2}.

A cornerstone of this theory is the master equation, a key innovation attributed to Lions \cite{Lions}. This equation serves as the fundamental PDE governing the value function of a representative agent in the mean field limit, generalizing the HJB equation from classical control theory. Its principal feature is the incorporation of the population's probability distribution as a state variable. This formulation elevates the problem to a PDE on the Wasserstein space of probability measures, an infinite-dimensional setting whose solution fully characterizes the closed-loop mean field equilibrium (MFE). It  encodes precisely how an agent's value and optimal strategy respond not only to their private state but also to the evolution of the overall population distribution.

The well-posedness of master equations is overall a very challenging problem. The existing  literature has mainly focused on the standard MFGs.  For local (in time) well-posedness results, we refer to Bensoussan-Yam \cite{BY}, Gangbo-Swiech \cite{GS}, and Mayorga \cite{M}. The global well-posedness relies on certain monotonicity conditions, and very serious efforts have been made in this direction,  see e.g. Bensoussan-Graber-Yam \cite{BGY}, Bertucci \cite{B1}, Cardaliaguet-Delarue-Lasry-Lions \cite{CDLL}, Cardaliaguet-Souganidis \cite{CarSou}, Chassagneux-Crisan-Delarue \cite{CCD},  Gangbo-Meszaros \cite{GM}, Gangbo-Meszaros-Mou-Zhang \cite{GMMZ}, Jakobsen-Rutkowski \cite{JR},\footnote{This work considers fractional and nonlocal diffusions, under the standard Lasry-Lions monotonicity condition.} and Mou-Zhang \cite{MZ2, MZ3}, to mention a few. We would also like to mention Cecchin-Delarue \cite{CecchinDelarue1} and Iseri-Zhang \cite{IZ} for some dynamic approaches on MFGs with multiple MFEs.

The problem remains largely open when one considers more general MFGs, which appear often in applications, especially when global well-posedness is concerned. This manuscript introduces a generic approach for the local well-posedness of very general MFG master equations, and we will address their global well-posedness in our accompanying papers \cite{MZZ1, MZZ2}. Our main message is that the master equation admits a unique classical solution over small time intervals as long as the corresponding MFE can be characterized through certain McKean-Vlasov Forward-Backward SDE (FBSDE) systems. To be precise, letting $B=(B^0, B^1)$ and $\cF^0_t := \cF^{B^0_t}$, we consider the following two FBSDE systems:
\bea
\label{intro-Xxi}
&&\left\{\ba{lll}
                         \dis  X_t^{\xi}=\xi+\int_{0}^t b_1(s,\Xi^\xi_s)ds+\int_0^t\sigma_1(s,X_s^{\xi},Y_s^{\xi},\cL_{X^\xi_s|\cF^{0}_s})dB_s,\ms\\
                         \dis Y_t^{\xi}= g_1(X_T^\xi,\cL_{X_T^{\xi}|\cF_T^{0}})+\int_{t}^T f_1(s,\Xi^\xi_s)ds - \int_t^T Z^{\xi}_s dB_s,
                 \ea\right.\\
\label{intro-X2x}
&&\left\{\ba{lll}
                         \dis  X_t^{2,x}=x+\int_{0}^{t}b_2(s,\Xi^{2,x}_s)ds+\int_0^t \sigma_2(s,X_s^{2,x},Y_s^{2,x},\cL_{X_s^\xi|\cF_s^{0}})dB_s,\ms\\
                         \dis Y_t^{2,x}= g_2(X^{2,x}_T,\cL_{X_T^{\xi}|\cF_T^{0}}) +\int_{t}^T  f_2(s,\Xi^{2,x}_s)ds- \int_t^T Z^{2,x}_s dB_s,
                 \ea\right.
\eea
where   $\Xi^\xi:=(X^\xi,Y^{\xi},Z^\xi, \cL_{X^\xi|\cF^0})$, $\Xi^{2, x} := (X^{2,x},Y^{2,x},Z^{2,x},\cL_{X^\xi|\cF^0})$. These systems satisfy the desired flow property, and in particular we may consider their decoupling field $U$:
\bea
\label{intro-U}
Y^{2, x}_t = U\big(t, X^{2, x}_t, \cL_{X^\xi_t|\cF^0_t}\big).
\eea
For a standard MFG, $B^0$ stands for the common noise, $\Phi_1=\Phi_2$ for $\Phi = b, \si, f, g$, and $U$ is the value of the MFG and thus is the  solution to the master equation. We are considering much more general systems here, in particular both $Y^\xi$ and $Y^{2,x}$ can be multi-dimensional.

 Our methodology is grounded in providing a pointwise representation formula for the Lions derivative $\pa_\mu U(t,x, \mu, \tilde x)$, through certain linear McKean-Vlasov FBSDE systems. This representation generalizes those in our earlier works Mou-Zhang \cite{MZ2} and  Gangbo-Meszaros-Mou-Zhang \cite{GMMZ} for standard MFGs. Then,  the differentiability of  $U$ follows from the local well-posedness of these linear FBSDEs, which is more or less standard.  We remark that the spatial derivative $\pa_x U$ can also be obtained through certain linear variational  FBSDEs, as in the standard literature. 

 One main feature of our representation formula is that the introduced FBSDEs are again in the form of \reff{intro-Xxi}-\reff{intro-X2x} (typically with higher dimensions). This allows us to invoke the bootstrap method to establish  higher order differentiabilities of $U$ for free, provided that the coefficients $(b_i, \si_i, f_i, g_i)$ have sufficient regularity and $T$ is small. Moreover, the time derivative $\pa_t U$ can be obtained by combining the flow property of the system and the spatial differentiability. Consequently, $U$ is a classical solution of the corresponding master equation, derived from applying the It\^{o} formula on \reff{intro-U}. In particular, when the systems \reff{intro-Xxi}-\reff{intro-X2x} are derived from an MFG and $U$ is the value function, we obtain the classical solution of the MFG master equation.
 
The above representation formula is obtained through differentiating both sides of \reff{intro-U} with respect to $\xi$. An alternative representation formula can be obtained by formally differentiating the corresponding master equation, which will lead to certain BSDE system (instead of FBSDE system). However, these BSDE systems are nonlinear and require higher regularities on the coefficients.
 
 There have been various approaches on  the local well-posedness in the literature, especially on continuous differentiability of $U$. For example, in the contexts of standard MFG, \cite[Section 5.2]{CD2} established the continuous differentiability of $U$ through certain uniform regularity of the random variable $\pa_\mu U(0, \mu, \xi)$, without constructing a pointwise representation of $\pa_\mu U$. Our previous works \cite{MZ2, GMMZ} provided such a representation, which is further generalized in the present paper. Again for standard MFG, \cite{CDLL} considered the MFG system, a forward backward system of stochastic PDEs. By formally differentiating the MFG system, they considered the new system satisfied by the linear functional derivative ${\d\over \d\mu}U$, and established the existence and regularity of ${\d\over \d\mu}U$ by certain uniform estimates of this system.\footnote{ \cite{CDLL} considered the global well-posedness under Lasry-Lions monotonicity. The same arguments remain valid over small time horizon without requiring the monotonicity.} In the contexts of extended MFG, the work Bertucci-Lasry-Lions \cite{BLL} investigated the master equation for the Lions derivative $\pa_\mu U$ and obtained the desired regularity of $\pa_\mu U$ and thus establish the local well-posedness. Our alternative representation formula is the natural representation derived from the master equation for $\pa_\mu U$. 
 
Our contributions are mainly threefold. First, our approach is very general and easy to apply. It is possible to extend the approaches in \cite{CD2, CDLL, BLL} to more general settings. However, this requires re-establishing the uniform estimates in each new setting, which can be nontrivial and often entails substantial technical work. We instead transform the issue to finding an FBSDE representation for the value of the MFG, which is typically quite easy.  Note that the required uniform estimates for these FBSDE systems are established  in this paper, so we don't need to redo it when encountering a new setting. This is exactly the main motivation of our paper.  Second, we established the $\cC^1$-regularity of $U$ whenever the coefficients are in $\cC^1$, and the $\cC^2$-regularity of $U$ and hence the existence of classical solutions when the coefficients are in $\cC^2$.  So our technical conditions are essentially sharp, which particularly could  be crucial when one considers weak solutions. The approaches in \cite{CDLL, BLL} require higher regularities on the coefficients in order to use the corresponding equations. Finally, we believe that the pointwise representation is interesting in its own right.

 To illustrate our approach and to show its power, we investigate three non-standard examples: extended MFGs, MFGs with volatility controls, and MFGs with a major player,  all viewed quite challenging problems in the literature. In each case, we proceed in three steps: 1) introduction of the game; 2) heuristic derivation of the  corresponding FBSDE systems; and 3) rigorous establishment of the local well-posedness. We emphasize that the third step follows directly from our main result and thus requires minimum extra efforts.

Extended MFGs, introduced by Lions and Souganidis \cite{LionsSouganidis},\footnote{It should be noted that  \cite{LionsSouganidis} considers MFGs with local coupling, see also Munoz \cite{Munoz}.} generalize standard MFGs by allowing the vector field governing the population flow to differ from that for an individual player at MFE. That is, the coefficients in \reff{intro-Xxi} and \reff{intro-X2x} are different. So, among other reasons, our FBSDE systems are exactly tailored for extended MFGs, and thus the local well-posedness of these master equations follows immediately from our main result. 
We remark that, under slightly stronger technical conditions,\footnote{If restricting to Lipschitz continuous solutions, without requiring  the continuous differentiability of $U$, \cite{BLL} established the local wellposedness under appropriate  Lipschitz conditions.} this local well-posedness for extended MFG has already been established in  \cite{BLL}, by using the master equation. In the global setting (with large $T$),  \cite{BLL} proved uniqueness for globally Lipschitz continuous solutions, without establishing the existence. Conversely, Mou-Zhang \cite{MZ-EMFG} established the existence of the maximal and minimal solutions for extended submodular MFGs, as introduced in Dianetti-Ferrari-Fischer-Nendel \cite{DFFN1}, confirming its global ill-posedness in general. We shall investigate systematically the global well-posedness of general extended MFG master equations in our accompanying paper \cite{MZZ2}.

The volatility control, as in the standard stochastic control theory, is much harder to analyze. To the best of our knowledge, the well-posedness of master equations (both local and global) for MFGs with volatility controls is completely open. One major difficulty is that the $Z$-components of the FBSDEs correspond to the first order derivatives $\pa_x V$ of the value function $V$, while in the volatility control case the MFE involves the second order derivatives $\pa_{xx} V$. Thus in the standard approach one cannot express its MFE through a self-contained system of FBSDEs. We manage to overcome this difficulty by increasing the dimension of $Y$, roughly speaking, we consider the FBSDE system for $Y_t = (V, \pa_x V, \pa_{xx} V)(t, X_t, \cL_{X^\xi_t|\cF^0_t})$. This enables us to characterize the MFE, as well as the value function $V$, through FBSDEs and hence obtain the local well-posedness of the master equation following our general result. Building on this result, we shall address the global well-posedness  in another accompanying paper \cite{MZZ1} by introducing a new notion called second order monotonicity conditions.\footnote{The first version of \cite{MZZ1} included the local well-posedness result. Since it has been covered by the general approach here, we will delete that part in the revision.}  We would also like to mention the very interesting works  Chowdhury-Jakobsen-Krupski \cite{CJK1, CJK2},  which study the global well-posedness of the mean field game system (the associated fully nonlinear forward backward PDEs) for certain MFG with volatility control under the standard Lasry-Lions monotonicity. We shall note though these works do not study the master equation, and thus the differentiability of $U$ is irrelevant.

Huang \cite{Huang} introduced MFGs with a major player to model strategic interactions between a single influential principal and a continuum of minor agents. This leads to a coupled system of two master equations: one for the major player and the other for the representative minor player. The resulting FBSDE systems have a major difference in this case though. Note that the minor population relies on the major player's state, denoted as $X^0$, then it is natural to consider the conditional law $\cL_{X^\xi|\cF^{X^0}}$, rather than $\cL_{X^\xi|\cF^{B^0}_t}$ in \reff{intro-Xxi}-\reff{intro-X2x}, where $X^0$ is one component of the $X^{2,x}$ in \reff{intro-X2x}. To overcome this difficulty, we introduce an auxiliary problem by formally considering common noise for the minor players, namely by assuming the minor players rely on the major player's noise $B^0$ rather than his state $X^0$. It turns out that this provides an alternative representation for the same master equation system, and thus enables us to apply our general result to establish its local well-posedness. We next use the classical solution of this master equation system to construct the MFE of the original game, and hence verify that this solution is indeed the value function of the MFG.  We remark that, when the volatility coefficients are constants, the local well-posedness of this master system was first established by Cardaliaguet-Cirant-Porretta \cite{CCP}, while its global well-posedness under the Lasry-Lions monotonicity condition was later obtained by Delarue-Mou \cite{DM}. Both of these results rely primarily on PDE arguments and require uniformly non-degenerate idiosyncratic noise, while our results are derived through probabilistic methods and permit general and possibly degenerate noise.

The remainder of this paper is organized as follows. In Section \ref{sect-setting} we introduce  the basic setting and some preliminary results. In Section \ref{sect-derivativemu}, we derive the key representation formula for the Lions derivative $\pa_\mu U$ of the decoupling field $U$ for the FBSDE systems. In Section \ref{sect-bootstrap} we employ the bootstrap method to establish the higher order differentiability of $U$. In Section \ref{sect-alterrep} we present the alternative representation formula by using the master equation. Sections \ref{sect-EMFG},  \ref{sect-MFGVC}, and \ref{sect-MFGM} are devoted to extended MFGs, MFGs with volatility controls, and MFGs with a major player, respectively. Finally, in Section \ref{sect-appendix} we complete a postponed technical proof.

\ms
\no{\bf Some notational conventions.} This paper considers multiple dimensional functions and their high order derivatives. The notations can be quite complicated and require very careful attention, in particular their precise descriptions will involve tensors. To ease the presentation and improve the readability, throughout the paper we take the following conventions, including possible notation abuse when there is no confusion.

\begin{itemize}
\item All vectors $x\in \dbR^d$ are viewed as column vectors, with $|x|:= \max_{1\le i\le d}|x_i|$. Given two vectors $x\in \dbR^{d}, y \in \dbR^{m}$, we abuse the notation and simply use $(x, y)$ to denote the column vector $(x^\top, y^\top)^\top \in \dbR^{d+m}$.

\item For a function $f: x\in\dbR^{d}\to \dbR^m$,  its gradient $\pa_x f\in \dbR^{d\times m}$ with $(i,j)$-th component $\pa_{x_j} f^i$. In particular, $\pa_x f^i$ is the $i$-th row of $\pa_x f$ and thus is viewed as a row vector in this case.

\item For $x, y\in \dbR^d$, $x\cd y:= \sum_{i=1}^d x_i y_i$; and for $A, B\in \dbR^{d\times m}$, $A: B := \sum_{i=1}^d \sum_{j=1}^m A_{ij} B_{ij}$.

\item For a random field $\Phi$, we typically write it as $\Phi_t(\cd)$ and we omit $\o$. However, for deterministic functions $\Phi$ involving time variable $t$, we write it as $\Phi(t,\cd)$, and when there is no confusion,  we omit $t$.

\end{itemize}
\section{Preliminaries}
\label{sect-setting}
\setcounter{equation}{0}

Fix a finite time horizon $[0,T]$. For $i=0,1$, let $(\Omega^i,\dbF^i,\dbP^i)$ be a filtered probability space, on which is defined a $d_i$-dimensional Brownian motion $B^i$. For later contexts, $B^0$ and $B^1$ stand for the common noise and the idiosyncratic noise, respectively,\footnote{The interpretation will be different in Section \ref{sect-MFGM}.} and when there is no common noise, we may consider $(\Omega^1,\dbF^1,\dbP^1, B^1)$ only. We shall assume $\dbF^0 = \dbF^{B^0}$ and $\dbF^1 = \cF^1_0\vee \dbF^{B^1}$ with $\cF^1_0$ atomless. We consider the product space: $\O:= \O_0\times \O_1$, $\cF_t:= \cF^0_t \otimes \cF^1_t$, $\dbP:= \dbP^0 \times \dbP^1$, $B := (B^0, B^1)$, $d_{01}:= d_0+d_1$, and denote $\dbE := \dbE^\dbP$. As usual in the MFG literature, given an $\cF_t$-measurable random variable $\xi$, let $\tilde \xi, \bar \xi$ etc. denote the conditionally independent copies of $\xi$, conditional on $\cF^0_t$. This can be easily achieved by expanding the probability space. For example, let $(\tilde \Omega^1,\tilde \dbF^1,\tilde \dbP^1)$ be a copy of $(\Omega^1,\dbF^1,\dbP^1)$, and extend the probability space further to $(\O_0\times \O_1\times \tilde \O_1$,  $\cF^0_t \otimes \cF^1_t \otimes \tilde\cF^1_t$, $\dbP^0 \times \dbP^1\times \tilde \dbP^1)$. Then, for $\xi = \f(\o^0, \o^1)$, $(\o^0, \o^1)\in \O$, the random variable $\tilde\xi := \f(\o^0, \tilde \o^1)$,  $(\o^0, \tilde \o^1)\in \O_0\times \tilde \O_1$, satisfies the desired property. Throughout the paper, whenever needed, we shall extend the probability space in this way without mentioning and we shall still denote it as $(\O, \dbF, \dbP)$, and $\dbE_t:= \dbE_{\cF_t}$.

For $d\ge 1$ and $p\ge 1$, let $\cP_p(\dbR^d)$ denote the set of probability measures $\mu$ on $\dbR^d$ with finite $p$-th moment: $\|\mu\|_p^p= \int_{\dbR^d} |x|^p \mu(dx)<\infty$. For any $\si$-algebra $\cG\subset \cF_T$, let  $\dbL^p(\cG)$ denote the set of $\cG$-measurable random variables $\xi$ with finite $p$-th moment: $\|\xi\|_p^p = \dbE[|\xi|^p]<\infty$, and $\dbL^p(\cG;\mu):=\big\{\xi\in \dbL^p(\cG): \cL_\xi = \mu\big\}$, where $\cL_\xi$ denotes the law of $\xi$ under $\dbP$. The set $\dbL^p(\cG)$ is equipped with the $p$-th norm $\|\cd\|_p$, and $\cP_p(\dbR^d)$ is equipped with the $p$-Wasserstein distance $W_p$:
\beaa
  W_p^p(\mu_1,\mu_2):=\inf\big\{\big(\|\xi_1-\xi_2\|_p: \text{ for all  $\xi_i\in \dbL^p(\cF_T;\mu_i)$, $i=1,2$}\big\},\q \mu_1, \mu_2\in \cP_p(\dbR^d).
\eeaa
Both spaces are complete metric spaces.

For a $W_2$-continuous function $U: \mathcal{P}_2(\dbR^d) \mapsto \mathbb R$, its Lions derivative, also called Wasserstein derivative, is a function $\pa_\mu U: (\mu, \tilde x)\in \mathcal{P}_2(\dbR^d)\times \mathbb R^d\mapsto \mathbb R^d$ satisfying:
\begin{eqnarray}
\label{pamu}
\left.\ba{lll}
\dis U(\mathcal{L}_{\xi +  \eta}) - U(\cL_\xi) = \mathbb E\big[ \partial_\mu U(\mu, \xi) \cdot \eta  \big] + o(\|\eta\|_2), \ \forall\xi, \eta\in\mathbb L^2(\mathcal{F}_T).
\ea\right.
\end{eqnarray}
Denote $\Theta:= [0, T]\times \mathbb R^n \times \mathcal{P}_2(\dbR^d)$ for some generic dimensions $n, d$. For a function $U: \zeta=(t,\bx,\mu)\in \Th\mapsto \dbR$,  we call it  locally uniformly continuous if $U$ is uniformly continuous on $D_R:= \{\zeta\in \Th: |\bx|\le R, \|\mu\|_2\le R\}$ for any $R>0$.\footnote{In the finite dimensional case, clearly locally uniform continuity is equivalent to pointwise continuity.}  We define its first order derivatives $\pa_\bx U: \Th\mapsto\dbR^n$ in the standard sense and $\pa_\mu U: (\zeta, \tilde x)\in\Th \times \dbR^d \mapsto \dbR^d$ as in \reff{pamu}. We may define the second order derivatives, including $\pa_{t}U: \Th\mapsto \dbR$, $\pa_{\bx\bx}U: \Th\mapsto \dbR^{n\times n}$, $(\pa_{ \bx\mu} U,\pa_{\tilde x\mu} U): (\zeta, \tilde x)\in\Th \times \dbR^d \mapsto (\dbR^{n\times d}, \dbR^{d\times d})$, and $\pa_{\mu\mu} U: (\zeta, \tilde x,\bar x)\in\Th \times \dbR^d\times\dbR^d \mapsto \dbR^{d\times d}$. Here we are using the parabolic order, namely the derivative in $t$ counts for two orders, and the derivative in $\bx$ or $\mu$ counts for one order. We follow this to define the $k$th-order (in parabolic sense) derivatives of $U$ for any $k\in\dbN$. Let $\cC^k$ denote the space of   continuous functions $U: \Th\to \dbR^m$ for some generic dimension $m$ such that it has up to $k$th-order  locally uniformly  continuous derivatives, and $\cC^k_b$ the subspace such that all the involved derivatives (but not $U$ itself) are bounded.

One important property  is the following It\^o's formula, see \cite{BLPR, CD2, CCD}. For $i=1,2$, let $d X^i_t := b^i_t dt + \sigma^i_t dB_t,$ where $b^i:[0,T]\times\Omega\to\mathbb R^{d^i_x}$ and $\sigma^i = (\sigma^{i,0}, \si^{i,1}):[0,T]\times\Omega\to\mathbb R^{d^i_x\times d_{01}}$ are $\mathbb F$-progressively measurable, and $U: [0, T]\times \dbR^{d^1_x}\times \cP_2(\dbR^{d^2_x})\to \dbR$ is in $\cC^2_b$  (the boundedness of the derivatives  can be weakened) .
Then, denoting $\th_t:= (t, X^1_t, \mathcal L_{X^2_t|\mathcal F^0_t})$,  we have
\begin{equation}
\left.\begin{array}{ll}
d U(\th_t) =  \Big[\partial_t U(\th_t) + \partial_x U(\th_t)\cdot b^1_t + \frac{1}{2} \partial_{xx} U(\th_t) : \sigma_t^1 (\sigma_t^1)^\top \Big]dt +\partial_xU(\th_t)\cdot\sigma_t^1dB_t  \\
\displaystyle \quad+ \dbE_t \Big[\big[\partial_\mu U(\th_t,\tilde X^2_t) \cd \tilde b^{2}_t + \frac{1}{2} \partial_{\tilde x}\partial_\mu U(\th_t, \tilde X^2_t): \tilde \sigma_t^2 (\tilde \sigma_t^2)^\top\big]dt + \partial_\mu U(\th_t,\tilde X^2_t) \cd \tilde \sigma^{2,0}_tdB_t^0 \label{Ito} \\
\displaystyle \quad+\big[\partial_x\partial_\mu U(\th_t,\tilde X^2_t): \sigma^{1,0}_t (\tilde \sigma_t^{2,0})^\top +\frac{1}{2}\partial_{\mu\mu}U(\th_t,\tilde X^2_t,\bar X^2_t) : \tilde\sigma_t^{2,0}(\bar \sigma_t^{2,0})^\top\big] dt\Big].
\end{array}\right.
\end{equation}
Here $\cL_{X_t^2|\cF_t^0}$ stands for the conditional law of $X_t^2$ given $\cF_t^0$, $\tilde X^2, \bar X^2$ are conditionally independent copies $X^2$, conditional on $\dbF^0$, and similarly for $\tilde b^2$ etc.

Our analysis of MFGs and master equations will rely heavily on coupled McKean-Vlasov FBSDEs, and the focus will be their dependence on the measure variable $\mu$. As a preparation, we present here some  results concerning the $x$-variable,  which are more or less standard and we postpone the proof to Appendix.

Consider the following standard FBSDE  with random coefficients on $[0, T]$:
\begin{eqnarray}\label{FBSDEx}
\begin{cases}
                         \dis  X_t^{x}=x+\int_{0}^t b_s(\Pi^x_s)ds+\int_0^t\sigma_s(X_s^x,Y_s^x)dB_s,\ms\\
                         \dis Y_t^{x}= g(X_T^x)+\int_{t}^T f_s(\Pi^x_s)ds - \int_t^T Z^{x}_s dB_s.
                 \end{cases}
\end{eqnarray}
Here $\Pi^x:=(X^x,Y^x,Z^x)$ is the solution triple; $X, Y$ have dimensions $d_x, d_y$, respectively, and the other terms have dimensions accordingly; $b, \si, f$ are $\dbF^0$-progressively measurable (due to the common noise in the MFGs later); $g$ is $\cF^0_T\times \cB(\dbR^{d_x})$-measurable. In particular, $\si = (\si^0, \si^1)\in \dbR^{d_x\times d_{01}}$,  $Z=(Z^0, Z^1)\in \dbR^{d_y\times d_{01}}$, and quite often $d_y >1$.  We note that $\si$ does not depend on $Z$. For notational convenience, we extend $\si$ and $g$ also to $\pi = (x, y, z)$. Given $x, \D x\in \dbR^{d_x}$, introduce further the following variational FBSDE, with solution $\td_x \Pi = \td_x \Pi^{x, \D x}$ corresponding to the derivatives of $\Pi^x$ with respect to $x$ along the direction $\D x$:
\begin{eqnarray}\label{tdFBSDEx}
\begin{cases}
                         \dis  \td_x X_t= \D x+\int_{0}^t \pa_\pi b_s(\Pi^x_s) \td_x \Pi_s ds +\int_0^t \pa_\pi \si_s(\Pi^x_s) \td_x \Pi_s dB_s,\ms\\
                         \dis \td_x Y_t= \pa_x g(X_T^x)\td_x X_T +\int_{t}^T \pa_\pi f_s(\Pi^x_s)\td_x \Pi_s ds - \int_t^T\td_x Z_s dB_s,
                 \end{cases}
\end{eqnarray}
where, letting $\si^{\k,i}$, $Z^{x,\k,i}$ denote the $i$-th column of $\si^\k, Z^{x,\k}$, for $\k=0,1$, $i=1,\cds, d_\k$,
\bea
\label{papi}
\left.\ba{c}
\dis   \pa_\pi \Phi \td_x \Pi :=\sum_{j=1}^{d_x} \pa_{x_j} \Phi \td_x X^j + \sum_{j=1}^{d_y}\pa_{y_j} \Phi \td_y Y^j +\sum_{l=1}^{d_y}\sum_{\k=0}^1 \sum_{j=1}^{d_\k} \pa_{z^{\k,l, j}} \Phi \td_x Z^{\k, l, j},~ \Phi = b, \si^{\k,i}, f;\\
\dis \pa_\pi \si \td_x \Pi dB_s :=  \sum_{\k=0}^1 \sum_{i=1}^{d_\k} \big[\pa_x \si^{\k, i}\td_x Y + \pa_y \si^{\k, i} \td_x X\big] dB^{\k, i}_s;\\
\dis \td_x Z_s dB_s :=\sum_{\k=0}^1 \sum_{i=1}^{d_\k} \td_x Z^{\k,i}_s dB^{\k, i}_s.
 \ea\right.
 \eea

\begin{assum}
\label{assum-0}
(i) $b, \si, f$ are $\dbF^0$-progressively measurable and $g$ is $\cF^0_T$-measurable, with $I_0^2<\infty$, where $I_t^2 := \dbE_{\cF^0_t}\big[|g(0)|^2 + \sum_{\Phi = b, \si, f}\int_t^T |\Phi_s(0,0,0)|^2ds\big]$.

\ms
\no(ii) For  $\Phi=b, \si, f, g$, $\Phi$ is differentiable in $\pi$ such that $|\pa_\pi\Phi|\le L_\Phi$ and $\pa_\pi \Phi$ is locally uniformly continuous in $\pi$ in the sense that, for any $R>0$, there exists a modulus of continuity function $\rho_R$ such that
\bea
\label{rhoR1}
|\pa_\pi\Phi_t(\pi_1) - \pa_\pi\Phi_t(\pi_2)|\le \rho_R(|\pi_1-\pi_2|),~\mbox{a.s.},\q\mbox{whenever}~ |\pi_1|, |\pi_2|\le R.
\eea
\end{assum}

\begin{thm}\label{thm-pax}
Under Assumption \ref{assum-0},\footnote{ The locally uniform continuity \reff{rhoR1} is needed only for  the time regularity in \reff{rhoR2} and the representation of $Z^{x,1}$ in \reff{randomFK}. In particular, it is not needed for the representation formula \reff{paxrep} of $\pa_x u_0$.}  there exists a constant $\e_0>0$, depending only on the dimensions $d_{01}, d_x, d_y$, and the bounds $L_\Phi$, $\Phi = b, \si, f, g$,  such that, whenever $0<T\leq \e_0$:

\no(i) The FBSDEs \reff{FBSDEx}-\reff{tdFBSDEx} are well-posed; and for any $p\ge 2$, there exists $C_p>0$ such that
\bea
\label{Lpest1}
\|\td_x \Pi\|_p^p:= \dbE\Big[ \sup_{0\le t\le T} [|\td_x X_t|^p + |\td_x Y_t|^p] + \big(\int_0^T |\td_x Z_t|^2dt\big)^{p\over 2}\Big] \le C_p |\D x|^p<\infty.
\eea

\no (ii) There exists an $\dbF^0$-progressively measurable random field $u: [0, T]\times \dbR^{d_x}\times \O\to \dbR^{d_y}$ such that $u$ is  differentiable in $x$, $|\pa_x u_t|\le R_0:= L_g+1$, and  $\pa_x u$ is  locally uniformly continuous in the sense that, for any $R>0$, there exists a modulus of continuity function $\rho_{R}$ such that
\bea
\label{rhoR2}
\left.\ba{c}
|\pa_x u_t(x^1) - \pa_x u_t(x^2)| \le \rho_{R}(|x_1-x_2|), ~ \mbox{a.s. on $\{I_t \le R\}$, whenever}~  |x_1|, |x_2| \le R;\\
|\pa_x u_{t_1}(x)- \dbE_{\cF^0_{t_1}}[\pa_x u_{t_2}(x)]|\le \rho_{R}(\sqrt{t_2-t_1}), ~\mbox{a.s. on $\{I_{t_1} \le R\}$, whenever}~  |x| \le R, t_1\le t_2.
\ea\right.
\eea
Moreover, it holds that
\bea
\label{randomFK}
Y^x_t = u_t(X^x_t),\q \td_x Y_t = \pa_x u_t(X^x_t) \td_x X_t,\q Z^{x,1}_t = \pa_x u_t(X^x_t) \si^1_t(X^x_t, Y^x_t).
\eea
In particular, $u_0(x) = Y^x_0$ is deterministic and we have a representation formula for $\pa_x u_0(x)$:\footnote{\label{matrix}In \reff{tdFBSDEx} we consider the directional derivatives of $\Pi$, and thus $\td_x \Pi$ has the same dimension as $\Pi^x$ itself.  Alternatively we may consider $\hat \td_x \Pi^x = (\td_x \Pi^{x, e_1}, \cds, \td_x \Pi^{x, e_{d_x}})$, which does not depend on $\D x$ anymore. In fact, $\td_x \Pi^{x, \D x} = \hat\td_x \Pi^x \D x$. Then $\hat \td_x X^x\in \dbR^{d_x\times d_x}$, with $\hat \td_x X^x_0$ the $d_x\times d_x$ identity matrix,  and $\pa_x u_0(x) = \hat\td _x Y^x_0\in \dbR^{d_y\times d_x}$. The results are actually more clean in this way, however, the notations will become a lot heavier, especially when we discuss the higher order derivatives. }
\bea
\label{paxrep}
\pa_{x_i} u_0(x) = \td_x Y^{x, e_i}_0,\q i=1,\cds, d_x,
\eea
where $e_i$ denotes the basis vector in $\dbR^{d_x}$ with the $i$-th coordinate equal to $1$.
\end{thm}

While the local well-posedness of the FBSDEs is standard and the representation \reff{paxrep} is natural, the regularity of $u$ is not trivial. In particular, the representation of $Z^{x,1}$ under such weak conditions seems new. However, since our main focus is the Lions derivative with respect to measure variable $\mu$, we postpone this proof to Section \ref{sect-appendix} below.

\section{Representation for Lions derivatives of decoupling fields}
\label{sect-derivativemu}
\setcounter{equation}{0}

Note that Theorem \ref{thm-pax}  provides a representation formula for $\pa_x u_0$, where $u$ is the decoupling field of the (standard) FBSDE \reff{FBSDEx}.  In this section we present a representation formula for the Lions derivatives of the decoupling field of general McKean-Vlasov FBSDE system, which will be the key for the regularity of the solutions to MFG master equations.

Given $\xi\in\dbL^2(\cF_0,\dbR^{d^1_x})$ and $x\in\dbR^{d^2_x}$, recall the McKean-Vlasov FBSDE systems \reff{intro-Xxi}-\reff{intro-X2x}:
\bea
\label{Xxi}
&&\left\{\ba{lll}
                         \dis  X_t^{\xi}=\xi+\int_{0}^t b_1(s,\Xi^\xi_s)ds+\int_0^t\sigma_1(s,X_s^{\xi},Y_s^{\xi},\cL_{X^\xi_s|\cF^{0}_s})dB_s,\ms\\
                         \dis Y_t^{\xi}= g_1(X_T^\xi,\cL_{X_T^{\xi}|\cF_T^{0}})+\int_{t}^T f_1(s,\Xi^\xi_s)ds - \int_t^T Z^{\xi}_s dB_s,
                 \ea\right.\\
\label{X2x}
&&\left\{\ba{lll}
                         \dis  X_t^{2,x}=x+\int_{0}^{t}b_2(s,\Xi^{2,x}_s)ds+\int_0^t \sigma_2(s,X_s^{2,x},Y_s^{2,x},\cL_{X_s^\xi|\cF_s^{0}})dB_s,\ms\\
                         \dis Y_t^{2,x}= g_2(X^{2,x}_T,\cL_{X_T^{\xi}|\cF_T^{0}}) +\int_{t}^T  f_2(s,\Xi^{2,x}_s)ds- \int_t^T Z^{2,x}_s dB_s,
                 \ea\right.
\eea
where $(X^\xi, Y^\xi)$ is $d_x^1\times d^1_y$-dimensional,  $(X^{2,x}, Y^{2,x})$ is $d_x^2\times d^2_y$-dimensional, and the other terms have dimensions accordingly. Here and in the sequel, $\Pi^\xi:=(X^\xi,Y^{\xi},Z^\xi)$, $\Xi^\xi:=(\Pi^\xi, \cL_{X^\xi|\dbF^0})$, and $ \Pi^{2, x}=(X^{2,x},Y^{2,x},Z^{2,x})$, $ \Xi^{2, x}:=(\Pi^{2,x},\cL_{X^\xi|\dbF^0})$.  It is clear that $\Pi^{2,x}$, $\Xi^{2,x}$ depend on $\xi$ as well. When there is a need to emphasize this dependence, we use the notation $\Pi^{2,\xi, x} = \Pi^{2,x}$ and $\Xi^{2,\xi,x}=\Xi^{2,x}$. We shall use this convention for other similar notations.
We remark that, for the purpose of studying higher order derivatives in the next section and for studying extended mean field games,  we consider different coefficients in \reff{Xxi} and \reff{X2x}. From now on, for $i=1,2$, the coefficients $b_i, \si_i, f_i, g_i$ are deterministic functions, and as our notational convention we often omit  $t$. Moreover, for notational convenience we  write $\si_i, g_i$  as $\si_i(\pi, \mu)$ and $g_i(\pi,\mu)$, for $\pi= (x, y, z)$. Provided the well-posedness which we will establish later, it is obvious that $Y^{2, \xi,x}_0 = Y^{2,x}_0$ is law invariant in terms of $\xi$, thus we can introduce:
\bea
\label{Uxmu}
U(0,x,\mu):= Y^{2, \xi, x}_0,\q \cL_\xi = \mu.
\eea
Our goal of this section is to provide a representation formula for $\pa_\mu U(0,x, \mu, \tilde x)$.

 Considering the $\tilde x\in \dbR^{d^1_x}$ inside $\pa_\mu U(0,x, \mu, \tilde x)$, we shall also introduce the following FBSDE with coefficients $(b_1, \si_1, f_1, g_1)$:
 \begin{eqnarray}\label{X1x}
\begin{cases}
                         \dis  X_t^{1, \tilde x}=\tilde x+\int_{0}^{t}b_1(s,\Xi^{1,\tilde x}_s)ds+\int_0^t \sigma_1(s,X_s^{1,\tilde x},Y_s^{1, \tilde x},\cL_{X_s^\xi|\cF_s^0})dB_s,\ms\\
                         \dis Y_t^{1, \tilde x}= g_1(X^{1, \tilde x}_T,\cL_{X_T^{\xi}|\cF_T^0}) +\int_{t}^T  f_1(s,\Xi^{1, \tilde x}_s)ds- \int_t^T Z^{1, \tilde x}_s dB_s.
                 \end{cases}
\end{eqnarray}
As in \reff{X2x}, here $\Pi^{1, \xi, \tilde x} = \Pi^{1,x} = (X^{1,x}, Y^{1,x}, Z^{1,x})$, $\Xi^{1,\xi, \tilde x}= \Xi^{1,\tilde x}=(\Pi^{1,\tilde x}, \cL_{X^\xi|\dbF^0})$. It is clear that $\Pi^{1, \xi, \xi} = \Pi^\xi$, and when $\Phi_1=\Phi_2$ for $\Phi=b,\si, f, g$, we have $\Pi^{1,x}=\Pi^{2,x}$.

We now specify the technical conditions.
\begin{assum}\label{assum-1} (i)  All the coefficients are deterministic and progressively measurable in all variables, with $ \int_0^T |\Phi_i(t,0,0,0, \d_0)|^2dt<\infty$ for $\Phi = b, \si, f$, $i=1,2$.

\ms
\no(ii) For $\Phi = b, \si, f, g$ and $i=1, 2$, $\pa_\pi \Phi_i$ and $\pa_\mu \Phi_i$ exist and are continuous in $(\pi, \mu, \tilde x)$, with $|\pa_\pi\Phi|, |\pa_\mu \Phi| \leq L_1$. In particular, $|\pa_x g_i|, |\pa_\mu g_i|\le L_g$.
\end{assum}

As in \reff{tdFBSDEx}, we shall introduce certain variational FBSDEs. For this purpose, we first introduce some differential operators. Let $\Phi$ be a function with variables $\zeta=(\pi,\mu) = (x, y, z, \mu)\in \dbR^{d^i_x}\times \dbR^{d^i_y} \times \dbR^{d^i_y\times d_{01}} \times \cP_2(\dbR^{d^1_x})$. Here the measure variable $\mu$ is always $d^1_x$-dimensional, but the dimension of $\pi$ could vary, which will be made clear in the contexts, and for simplicity we shall use the same notation $\pi$.   Given $\d \pi = (\d x, \d y, \d z)$, and $\vec x = (x^1,\cds, x^m)$, $\d\vec x=( \delta x^1,\cds, \delta x^m) \in (\dbR^{d^i_x})^m$, for some $m\ge 1$, denote
 \bea
 \label{tdPhi1}
\dis  \nabla_{\mu}\Phi(\zeta; \vec x, \delta\vec x):= \sum_{i=1}^m(\delta x^i)^\top \pa_{\mu}\Phi(\zeta,x^i);\q
\nabla \Phi(\zeta; \d\pi; \vec x, \delta\vec x):= \pa_\pi\Phi(\zeta)\delta \pi + \nabla_{\mu}\Phi(\zeta; \vec x, \delta\vec x),
\eea
where  $\pa_\pi\Phi(\zeta)\delta \pi$ is defined by \reff{papi}, with $\td_x\Pi$ replaced with $\d\pi$. Note that the above operators are linear in $\d\pi$ and $\d \vec x$, consequently  the FBSDEs  driven by them at below are linear.

Given $\tilde x, \D \tilde x\in \dbR^{d^1_x}$, as in \reff{FBSDEx}  let $\td_{\tilde x} \Pi^{1, \xi, \tilde x, \D\tilde x} = \td_{\tilde x} \Pi^{\tilde x} = (\nabla_{\tilde x} X^{\tilde x}, \nabla_{\tilde x} Y^{\tilde x}, \nabla_{\tilde x} Z^{\tilde x})$ correspond to the derivatives of $\Pi^{1,\tilde x}$ with respect to $\tilde x$ along the direction $\D\tilde x$: recalling the notations \reff{papi},
\begin{eqnarray}\label{tdxFBSDE}
\begin{cases}
                         \dis  \nabla_{\tilde x} X_t^{\tilde x}=\Delta \tilde x+\int_{0}^t \pa_\pi b_1(\Xi^{1,\tilde x}_s)\nabla_{\tilde x}\Pi^{\tilde x}_sds+\int_0^t\pa_\pi \sigma_1(\Xi_s^{1,\tilde x}) \nabla_{\tilde x}\Pi_s^{\tilde x}dB_s,\\
                         \dis \nabla_{\tilde x}Y_t^{\tilde x}= \pa_x g_1(\Xi^{1,\tilde x}_T)\nabla_{\tilde x}\Pi^{\tilde x}_T+\int_{t}^T \pa_\pi f_1(\Xi^{1,\tilde x}_s)\nabla_{\tilde x}\Pi^{\tilde x}_sds- \int_t^T \nabla_{\tilde x} Z^{\tilde x}_s dB_s.
                 \end{cases}
\end{eqnarray}
The next one is motivated from differentiating \eqref{Xxi} in $\xi$, see \eqref{tdXeta} below. However, to obtain the pointwise representation, we need to make some modifications, with the solution denoted as $\nabla_{\mu}\Pi^{1, \xi, \tilde x}=\nabla_{\mu}\Pi^{\tilde x}:=(\nabla_{\mu}X^{\tilde x},\nabla_{\mu}Y^{\tilde x},\nabla_{\mu}Z^{\tilde x})$, and denoting $\Upsilon^{\tilde x}:=\big((X^{1, \tilde x},X^\xi),(\nabla_{\tilde x}X^{\tilde x},\nabla_{\mu}X^{\tilde x})\big)$:
\begin{eqnarray}\label{tdmuFBSDE}
\begin{cases}
                         \dis  \nabla_{\mu} X_t^{\tilde x}=\int_{0}^t  \dbE_s  \big[\nabla b_1(\Xi^{\xi}_s;\nabla_{\mu}\Pi^{\tilde x}_s;\tilde\Upsilon_s^{\tilde x})\big]ds+\int_0^t\dbE_s  \big[\nabla \sigma_1(\Xi^{\xi}_s;\nabla_{\mu}\Pi^{\tilde x}_s;\tilde\Upsilon_s^{\tilde x})\big]dB_s,\\
                         \dis \nabla_{\mu}Y_t^{\tilde x}=\dbE_T \big[\nabla g_1(\Xi^{\xi}_T;\nabla_{\mu}\Pi^{\tilde x}_T;\tilde\Upsilon_T^{\tilde x})\big]+\int_{t}^T \dbE_s  \big[\nabla f_1(\Xi^{\xi}_s;\nabla_{\mu}\Pi^{\tilde x}_s;\tilde\Upsilon_s^{\tilde x})\big]ds - \int_t^T \nabla_{\mu} Z^{\tilde x}_s dB_s.
                 \end{cases}
\end{eqnarray}
We note that
\beaa
\mbox{$\Upsilon^{\tilde x}$ corresponds to $(\vec x, \d \vec x)$ in (\ref{tdPhi1}) with $m=2$, $\vec x = (X^{1,\tilde x},X^\xi)$, $\d \vec x = (\nabla_{\tilde x}X^{\tilde x},\nabla_{\mu}X^{\tilde x})$.}
\eeaa
Recall that $\tilde \Upsilon^{\tilde x}$ is a conditionally independent copy of $\Upsilon^{\tilde x}$,  so \reff{tdmuFBSDE} involves the law of the unknown $\nabla_{\mu}X^{\tilde x}$ and thus it is a linear McKean-Vlasov FBSDE system. Our representation formula requires another FBSDE system, by replacing the $\Pi^\xi$ above with $\Pi^{2,x}$ and using the coefficients $\Phi_2$:
\begin{eqnarray}\label{tdmuFBSDE2}
\begin{cases}
                         \dis  \nabla_{\mu} X_t^{x,\tilde x}=\int_{0}^t  \dbE_s  \big[\nabla b_2(\Xi^{2,x}_s;\nabla_{\mu}\Pi^{x,\tilde x}_s;\tilde\Upsilon_s^{\tilde x})\big]ds+\int_0^t\dbE_s  \big[\nabla \sigma_2(\Xi^{2,x}_s;\nabla_{\mu}\Pi^{x,\tilde x}_s;\tilde\Upsilon_s^{\tilde x})\big]dB_s,\ms\\
                         \dis \nabla_{\mu}Y_t^{x,\tilde x}=\dbE_T \big[\nabla g_2(\Xi^{2,x}_T;\nabla_{\mu}\Pi^{x,\tilde x}_T;\tilde\Upsilon_T^{\tilde x})\big] - \int_t^T\nabla_{\mu} Z^{x,\tilde x}_s dB_s\\
      \dis  \qq\qq\qq                 + \int_{t}^T \dbE_s  \big[\nabla f_2(\Xi^{2,x}_s;\nabla_{\mu}\Pi^{x,\tilde x}_s;\tilde\Upsilon_s^{\tilde x})\big]ds.
                 \end{cases}
\end{eqnarray}
Here the solution is denoted as $\nabla_{\mu}\Pi^{2,\xi, x,\tilde x, \D \tilde x} = \nabla_{\mu}\Pi^{x,\tilde x}=(\nabla_{\mu}X^{x,\tilde x},\nabla_{\mu}Y^{x,\tilde x},\nabla_{\mu}Z^{x,\tilde x})$. We note that this FBSDE is a standard one without involving the law of the unknown.

\begin{thm}\label{thm-pamuUrep}
Let Assumption \ref{assum-1} hold.

\no (i) There exists $\e_0>0$, depending only on $L_1$, and the dimensions $d^1_x, d^1_y, d^2_x, d^2_y, d_{01}$,  such that for any $x\in \dbR^{d^2_x}$, $\tilde x, \D \tilde x\in \dbR^{d^1_x}$,  $\xi\in\dbL^2(\cF_0^1,\dbR^{d^1_x})$, and $T\le \e_0$,  the FBSDE systems \eqref{Xxi}, \eqref{X2x}, \eqref{X1x},\eqref{tdxFBSDE}, \eqref{tdmuFBSDE}, \eqref{tdmuFBSDE2} are well-posed on $[0,T]$, and for any $p\ge 2$, there exists $C_p>0$ such that
\bea
\label{Lpest2}
 \|\td_{\tilde x} \Pi^{\tilde x}\|_p + \|\td_{\mu} \Pi^{\tilde x}\|_p + \|\td_{\mu} \Pi^{x,\tilde x}\|_p \le C_p|\D \tilde x|.
\eea

\no(ii) For the $U$ defined by \reff{Uxmu}, $\pa_{\mu}U(0,\cdot,\cd,\cd)$ exists and is continuous, and
\bea
\label{pamuUbound}
|\pa_{\mu}U(0,x,\cL_{\xi},\tilde x)|\le R_0:=L_g+1.
\eea
 Moreover, the following representation formula holds
\begin{equation}\label{pamuUrep}
\pa_{\mu_i}U(0,x,\cL_{\xi},\tilde x) = \nabla_{\mu}Y_0^{2,\xi,x,\tilde x, \tilde e_i},
\end{equation}
where $\{\tilde e_i\}_{i=1,\cds, d^1_x}$ is the basis of $\dbR^{d^1_x}$, and $\nabla_{\mu}Y_0^{2,\xi,x,\tilde x, \tilde e_i}=\nabla_{\mu}Y_0^{x,\tilde x}$ as in \reff{tdmuFBSDE2} with $\D \tilde x = \tilde e_i$.
\end{thm}
\proof  The existence of  $\e_0>0$ such that the FBSDE systems are well-posed on $[0,T]$ whenever $T\leq \e_0$ follows from the standard FBSDE theory and its extension to McKean-Vlasov equations. Next, for $T\le \e_0$ and for a generic constant $C>0$ depending on the same parameters, applying standard FBSDE estimates on \eqref{tdxFBSDE}, \eqref{tdmuFBSDE}, \eqref{tdmuFBSDE2} we have: assuming $|\D \tilde x|= 1$,
\beaa
& \dis\dbE\Big[\sup_{0\le t\le T} |\td_{\tilde x} X^{\tilde x}_t|^2\Big] \le 1+ CT,\q \dbE\Big[\sup_{0\le t\le T} |\td_{\tilde x} Y^{\tilde x}_t|^2 + \int_0^T |\td_{\tilde x} Z^{\tilde x}_t|^2dt\Big] \le L_g^2 + CT;\\
& \dis\dbE\Big[\sup_{0\le t\le T} |\td_{\mu} X^{\tilde x}_t|^2\Big] \le CT, \q \dbE\Big[\sup_{0\le t\le T} |\td_{\mu} Y^{\tilde x}_t|^2 + \int_0^T |\td_{\mu} Z^{\tilde x}_t|^2dt\Big] \le L_g^2 + CT;\\
& \dis \dbE\Big[\sup_{0\le t\le T} |\td_{\mu} X^{x,\tilde x}_t|^2\Big] \le CT, \q \dbE\Big[\sup_{0\le t\le T} |\td_{\mu} Y^{x,\tilde x}_t|^2 + \int_0^T |\td_{\mu} Z^{x,\tilde x}_t|^2dt\Big] \le L_g^2 + CT.
\eeaa
In particular, $|\td_{\mu} Y^{x,\tilde x}_0|\le R_0$ when $T\le \e_0$ is small enough. Thus, given  \reff{pamuUrep}, we have $|\pa_{\mu}U(0,x,\cL_{\xi},\tilde x)|\le R_0$. The estimate \reff{Lpest2} follows from similar arguments as those for \reff{Lpest1}, and  given  \reff{pamuUrep}, the continuity of $\pa_\mu U(0, \cd, \cd, \cd)$ follows from the stability of the FBSDEs. We now prove \reff{pamuUrep} in three steps.

\ms
{\bf Step 1.} For any $\xi,\eta\in \mathbb{L}^2(\mathcal{F}^1_0,\dbR^{d})$, following standard arguments and by the stability  of the involved FBSDE systems we have
\begin{eqnarray}
\label{Xetaconv}
\left.\ba{c}
\dis\lim_{\varepsilon\rightarrow0}\mathbb{E}\Big[\sup_{0\leq t\leq T}\Big|\frac{1}{\varepsilon}[X^{\xi+\varepsilon\eta}_t-X^{\xi}_t]
-\delta X^{\eta}_t\Big|^2+ \sup_{0\leq t\leq T}\Big|\frac{1}{\varepsilon}[X^{\xi+\varepsilon\eta,x}_t-X^{\xi,x}_t]
-\delta X^{\eta,x}_t\Big|^2\Big]=0,\ms\\
\dis\lim_{\varepsilon\rightarrow0}\mathbb{E}\Big[\sup_{0\leq t\leq T}\Big|\frac{1}{\varepsilon}[Y^{\xi+\varepsilon\eta}_t-Y^{\xi}_t]
-\delta Y^{\eta}_t\Big|^2+ \sup_{0\leq t\leq T}\Big|\frac{1}{\varepsilon}[Y^{\xi+\varepsilon\eta,x}_t-Y^{\xi,x}_t]
-\delta Y^{\eta,x}_t\Big|^2\Big]=0,
\ea\right.
\end{eqnarray}
where $\d \Pi^{\eta}:=(\delta X^\eta,\delta Y^{\eta},\delta Z^{\eta})$  satisfies the following linear McKean-Vlasov FBSDEs, obtained by formally differentiating \reff{Xxi} in $\xi$ along the direction $\eta$:
\begin{eqnarray}\label{tdXeta}
\begin{cases}
\dis \delta X^{\eta}_t=\eta+\int^t_{0}\dbE_s \big[\nabla b_1(\Xi^\xi_s; \d \Pi^{\eta}_s; \tilde X^\xi_s, \widetilde{\d X_s^{\eta}})\big]ds +\int^t_{0} \dbE_s \big[\nabla \sigma_1(\Xi^\xi_s;\d \Pi^{\eta}_s; \tilde X^\xi_s, \widetilde{\d X_s^{\eta}})\big]dB_s,\ms\\
\dis \delta {Y}^{\eta}_t=\dbE_T\big[\td  g_1(\Xi^\xi_T; \d \Pi^{\eta}_T; \tilde X^\xi_T, \widetilde{\d X_T^{\eta}})\big] +\int^{T}_{t}\dbE_s \big[\nabla f_1(\Xi^\xi_s;\d \Pi^{\eta}_s; \tilde X^\xi_s, \widetilde{\d X_s^{\eta}})\big]ds -\int^T_s\delta {Z}^{\eta}_sdB_s,
\end{cases}
\end{eqnarray}
and $\d \Pi^{ \eta,x}:=(\delta X^{\eta,x},\delta Y^{\eta,x},\delta Z^{\eta,x})$ is obtained by formally differentiating \reff{X2x} in $\xi$ along $\eta$:
\begin{eqnarray}\label{tdXxeta}
\begin{cases}
\dis \delta X^{\eta,x}_t=\int^t_{0}\dbE_s \big[\nabla b_2(\Xi^{2, x}_s;\d \Pi^{\eta,x}_s;\tilde X_s^\xi,\widetilde{\d X_s^{\eta}})\big]ds +\int^t_{0} \dbE_s \big[\nabla \sigma_2(\Xi^{2, x}_s;\d \Pi^{\eta,x}_s; \tilde X_s^\xi,\widetilde{\d X_s^{\eta}})\big]dB_s,\ms\\
\dis \delta {Y}^{\eta,x}_t=\dbE_T\big[\nabla g_2(\Xi^{2, x}_T;\d \Pi^{\eta,x}_T; \tilde X_T^\xi,\widetilde{\d X_T^{\eta}})\big]  -\int^T_s\delta {Z}^{\eta,x}_sdB_s\\
\dis\qq\qq +\int^{T}_{t}\dbE_s \big[\nabla f_2(\Xi^{2, x}_s;\d \Pi^{\eta,x}_s; \tilde X_s^\xi,\widetilde{\d X_s^{\eta}})\big]ds.
\end{cases}
\end{eqnarray}
Here we used \reff{tdPhi1} with $m=1$. In particular, \reff{Xetaconv} implies
\bea
\label{pamuUrep1}
\lim_{\varepsilon\rightarrow0}\left|\frac{1}{\varepsilon}[U(0,x,\mathcal{L}_{\xi+\varepsilon\eta})-U(0,x,\mathcal{L}_{\xi})]
-\delta {Y}^{\eta,x}_0\right|^2=0.
\eea
Now fix $x$ and lift $U$  to a function $\cU:\dbL^2(\cF_0, \dbR^{d_x^1})\to \dbR^{d^2_y}$:
\bea
\label{lift}
\cU(\xi) := U(0,x, \cL_\xi).
\eea
 Then \reff{pamuUrep1} implies that $\cU$ has the Gateaux derivative $D \cU(\xi) \in \dbL^2(\cF_0, \dbR^{d_y^2\times d_x^1})$ such that
\begin{eqnarray}\label{pamuUrep2}
 \mathbb{E}\big[D\cU(\xi)\eta\big]=\delta Y^{\eta,x}_0.
\end{eqnarray}

\ms
{\bf Step 2.} In this step we assume $\xi$  is discrete: $p_i=\mathbb{P}(\xi=\tilde x_i)$, $i=1,\cdots, n$, and show that
\bea
\label{pamuUrep-claim}
\d \Pi^{ \eta_i,x} = p_i  \td_\mu \Pi^{x,\tilde x_i}, \q\mbox{where}\q \eta_i :=  {\bf 1}_{\{\xi=\tilde x_i\}}\Delta \tilde x, \q  i=1,\cds, n.
\eea
First, note that \reff{X1x} is a standard FBSDE with parameter $\cL_{X^\xi|\dbF^0}$. By \reff{Xxi} and \reff{X1x}, one can easily see that
\bea
\label{Pixidecompose1}
\Xi^\xi = \sum_{i=1}^n \kappa_i \Xi^{1,\tilde x_i},\q \mbox{where}\q \kappa_i:={\bf 1}_{\{\xi=\tilde x_i\}}.
\eea
 This implies that, for $\Phi = b_1, \si_1$, since $\{\xi=\tilde x_i\}_{1\le i\le n}$ form a partition of $\O$,
\bea
\label{etapartition}
\left.\ba{lll}
\dis \kappa_i \pa_\pi \Phi(\Xi^\xi)\nabla_{\tilde x}\Pi_s^{\tilde x_i} = \sum_{j=1}^n \kappa_i\kappa_j \pa_\pi \Phi(\Xi^{1, \tilde x_j}_s) \td_{\tilde x} \Pi^{\tilde x_i}_s =\kappa_i \pa_\pi \Phi(\Xi^{1, \tilde x_i}_s) \td_{\tilde x} \Pi^{\tilde x_i}_s;\\
\dis  \dbE_s \big[\pa_{\mu} \Phi(\Xi^\xi_s, \tilde X^\xi_s)\widetilde {\kappa_i} \widetilde{\td_{\tilde x} X_s^{\tilde x_i}}\big] =  \sum_{j=1}^n \dbE_s \big[\pa_{\mu} \Phi(\Xi^\xi_s, \tilde X^{1,\tilde x_j}_s)\widetilde {\kappa_i} \widetilde {\kappa_j} \widetilde{\td_{\tilde x} X_s^{\tilde x_i}}\big]\\
\dis = \dbE_s\big[\pa_{\mu} \Phi(\Xi^\xi_s, \tilde X^{1,\tilde x_i}_s)\widetilde {\kappa_i}  \widetilde{\td_{\tilde x} X_s^{\tilde x_i}}\big] = p_i \dbE_s\big[\pa_{\mu} \Phi(\Xi^\xi_s, \tilde X^{1,\tilde x_i}_s) \widetilde{\td_{\tilde x} X_s^{\tilde x_i}}\big],
\ea\right.
\eea
where the last equality used the fact that $(\Xi^\xi, \tilde X^{1,\tilde x_i}, \widetilde{\td_{\tilde x} X^{\tilde x_i}})$ are conditionally independent of $\widetilde\kappa_i$, conditional on $\dbF^0$, and  $\dbE_s[\tilde \k_i]=\dbE[\tilde \kappa_i] =p_i$.

Next, denote $\d \Pi^{'i}:= \kappa_i \td_{\tilde x} \Pi^{\tilde x_i} + p_i \td_\mu \Pi^{\tilde x_i}$. Then, in light of \reff{tdxFBSDE}, \reff{tdmuFBSDE}, and \reff{tdXeta}, by \reff{tdPhi1} and \reff{etapartition}  we have, again for $\Phi = b_1, \si_1$,
\beaa
&&\dis \kappa_i \pa_\pi \Phi (\Xi^{1,\tilde x_i}_s) \td_{\tilde x} \Pi^{\tilde x_i}_s  + p_i\dbE_s \big[ \td \Phi (\Xi^\xi_s; \td_\mu \Pi^{\tilde x_i}_s; \widetilde\Upsilon^{\tilde x_i}_s)\big]\\
&=&\dis  \kappa_i \pa_\pi \Phi (\Xi^{\xi}_s) \td_{\tilde x} \Pi^{\tilde x_i}_s  + p_i\pa_\pi \Phi (\Xi^{\xi}_s)\td_\mu \Pi^{\tilde x_i}_s + p_i\dbE_s \big[ \pa_{\mu} \Phi (\Xi^\xi_s, \tilde X^{1,\tilde x_i}_s)\widetilde{\td_{\tilde x}  X^{\tilde x_i}_s} +\pa_{\mu} \Phi (\Xi^\xi_s, \tilde X^{\xi}_s)\widetilde{\td_\mu X^{\tilde x_i}_s}\big]\\
&=& \dis \pa_\pi \Phi (\Xi^{\xi}_s) \d \Pi^{'i}_s  + \dbE_s\big[\pa_{\mu} \Phi(\Xi^\xi_s, \tilde X^\xi_s)\widetilde {\kappa_i} \widetilde{\td_{\tilde x} X_s^{\tilde x_i}}\big] + p_i\dbE_s\big[\pa_{\mu} \Phi(\Xi^\xi_s, \tilde X^{\xi}_s)\widetilde{\td_\mu X^{\tilde x_i}_s}\big]\\
&=&\dis  \pa_\pi \Phi (\Xi^{\xi}_s; \d \Pi^{'i}_s)  +  \dbE_s\big[\pa_{\mu} \Phi(\Xi^\xi_s, \tilde X^\xi_s)\widetilde{\d X_s^{'i}}\big] = \dis   \dbE_s \big[\td \Phi(\Xi^\xi_s;  \d \Pi^{'i}_s; \tilde X^\xi_s, \widetilde{\d X_s^{'i}})\big].
\eeaa
This implies  that
\beaa
\dis d \d X^{'i}_s &=& \kappa_i d \td_{\tilde x} X^{\tilde x_i}_s + p_i d \td_\mu X^{\tilde x_i}_s \\
\dis &=& \Big[\kappa_i \pa_\pi b_1 (\Xi^{1,\tilde x_i}_s) \td_{\tilde x} \Pi^{\tilde x_i}_s  + p_i\dbE_s \big[ \td b_1 (\Xi^\xi_s; \td_\mu \Pi^{\tilde x_i}_s; \widetilde\Upsilon^{\tilde x_i}_s)\big]\Big]ds\\
&&+ \Big[\kappa_i \pa_\pi \sigma_1 (\Xi^{1,\tilde x_i}_s)\td_{\tilde x} \Pi^{\tilde x_i}_s  + p_i\dbE_s \big[ \td \sigma_1 (\Xi^\xi_s; \td_\mu \Pi^{\tilde x_i}_s; \widetilde\Upsilon^{\tilde x_i}_s)\big]\Big]dB_s\\
\dis &=&  \dbE_s\big[\td b_1(\Xi^\xi_s;  \d \Pi^{'i}_s; \tilde X^\xi_s, \widetilde{\d X_s^{'i}})\big]ds +\dbE_s\big[\td \sigma_1(\Xi^\xi_s;  \d \Pi^{'i}_s; \tilde X^\xi_s, \widetilde{\d X_s^{'i}})\big]dB_s.
\eeaa
Similarly we can show that
\beaa
 d\delta {Y}^{'i}_s=\dbE_s \big[\td f_1(\Xi^\xi_s; \d \Pi^{'i}_s; \tilde X^\xi_s, \widetilde{\d X_s^{'i}})\big]ds+\delta Z_s^{'i}dB_s,\q \delta Y^{'i}_T=\dbE_T\big[\td g_1(\Xi^\xi_T; \d \Pi^{'i}_T; \tilde X^\xi_T, \widetilde{\d X_T^{'i}})\big].
\eeaa
Moreover, $\delta X^{'i}_0=\kappa_i \td_{\tilde x} X^{\tilde x_i}_0 + p_i \td_\mu X^{\tilde x_i}_0 =  \kappa_i\widetilde{\Delta  x}$. So $\d \Pi^{'i}$ satisfies \reff{tdXeta} with $\eta=\eta_i=\kappa_i\widetilde{\Delta x}$. Then, by the uniqueness of the FBSDE we have $\d \Pi^{'i} = \d \Pi^{\eta_i}$. That is,
\bea
\label{Pixidecompose2}
\d \Pi^{\eta_i} = \kappa_i \td_{\tilde x} \Pi^{\tilde x_i} + p_i \td_\mu \Pi^{\tilde x_i}.
\eea

Finally, for $\Phi = b_2, \si_2$, similarly to \reff{etapartition} we have
\beaa
 \dbE_s \big[\pa_{\mu} \Phi(\Xi^{2,x}_s, \tilde X^\xi_s)\widetilde {\kappa_i} \widetilde{\td_{\tilde x} X_s^{\tilde x_i}}\big]  = p_i \dbE_s\big[\pa_{\mu} \Phi(\Xi^{2,x}_s, \tilde X^{1,\tilde x_i}_s) \widetilde{\td_{\tilde x} X_s^{\tilde x_i}}\big].
\eeaa
Then, in light of \eqref{tdmuFBSDE2} and \reff{tdXxeta}, by \reff{tdPhi1} and \eqref{Pixidecompose2} we have
\beaa
&&\dbE_s \big[\nabla \Phi(\Xi^{2,x}_s; \d \Pi^{\eta_i,x}_s; \tilde X^\xi_s,\widetilde{\delta X_s^{\eta_i}})\big] \\
&=&\pa_\pi \Phi(\Xi^{2,x}_s) \d \Pi^{\eta_i, x}_s + \dbE_s \big[\pa_{\mu}\Phi(\Xi^{2,x}_s, \tilde X^\xi_s) \tilde\kappa_i\widetilde{\td_{\tilde x} X_s^{\tilde x_i}}\big]+p_i\dbE_s \big[\pa_{\mu}\Phi(\Xi_s^{2,x},\tilde X_s^{\xi})\widetilde{\nabla_{\mu}X_s^{\tilde x_i}}\big]\\
&=& \pa_\pi \Phi(\Xi^{2,x}_s) \d \Pi^{\eta_i, x}_s+ p_i\dbE_s \Big[\pa_{\mu}\Phi(\Xi^{2,x}_s, \tilde X^{1,\tilde x_i}_s) \widetilde{\td_{\tilde x} X_s^{\tilde x_i}}+\pa_{\mu}\Phi(\Xi_s^{2,x},\tilde X_s^{\xi})\widetilde{\nabla_{\mu}X_s^{\tilde x_i}}\Big].
\eeaa
Thus, denoting $\d\Pi^{'x,i}:= {1\over p_i}\d \Pi^{\eta_i,x}$, again for $\Phi = b_2, \si_2$, we have
\beaa
&&\dbE_s \big[\nabla \Phi(\Xi^{2,x}_s; \d \Pi^{'x,i}_s; \tilde X^\xi_s, \widetilde{\d X_s^{\eta_i}})\big] \\
&=& \pa_\pi \Phi(\Xi^{2,x}_s) \d \Pi^{'x,i}_s+\dbE_s \Big[\pa_{\mu}\Phi(\Xi^{2,x}_s, \tilde X^{\xi,\tilde x_i}_s) \widetilde{\td_{\tilde x} X_s^{\tilde x_i}}+\pa_{\mu}\Phi(\Xi_s^{2,x},\tilde X_s^{\xi})\widetilde{\nabla_{\mu}X_s^{\tilde x_i}}\Big]\\
&=&\dbE_s \big[\nabla \Phi(\Xi^{2,x}_s; \d \Pi^{'x,i}_s; \tilde\Upsilon^{\tilde x}_s)\big].
\eeaa
Similarly one can check the other required equalities, so that $\d\Pi^{'x,i}$ satisfies \reff{tdmuFBSDE2} with $\tilde x = \tilde x_i$. Then by the uniqueness of FBSDE  \reff{tdmuFBSDE2}, we obtain $\d\Pi^{'x,i} = \td_\mu \Pi^{x,\tilde  x_i}$. This verifies \reff{pamuUrep-claim} and hence \reff{pamuUrep-discrete} when $\xi$ is discrete.

\ms
{\bf Step 3.} We now prove \reff{pamuUrep}.  We first consider discrete $\xi$ as in Step 2. By \cite[Proposition 1]{WZ1} or \cite[Theorem 2.2, Step 1]{WZ2}, especially by the arguments there, we see  those results  hold true for Gateaux derivatives and thus there exists a function $\pa_\mu \cU: \cP_2(\dbR^{d_x^1}) \times \dbR^{d_x^1}\to \dbR^{d_y^2\times d_x^1}$ such that
\bea
\label{pamuUrep3}
D\cU(\xi) = \pa_\mu \cU(\cL_\xi, \xi).
\eea
Then, by  \reff{pamuUrep2} and \reff{pamuUrep3} we have, again for fixed $x$,
\beaa
\delta Y^{\eta_i,x}_0 = \mathbb{E}\big[\partial_{\mu} \cU(\cL_{\xi},\xi)\eta_i\big] = p_i \partial_{\mu} \cU(\cL_{\xi},\tilde x_i)\Delta \tilde x.
\eeaa
 Combining with \reff{pamuUrep-claim} and recalling \reff{tdmuFBSDE2} with $\nabla_{\mu}\Pi^{2,\xi, x,\tilde x, \D \tilde x} = \nabla_{\mu}\Pi^{x,\tilde x}$, this implies that
\bea
\label{pamuUrep-discrete1}
\partial_{\mu} \cU(\cL_{\xi},\tilde x_i)\Delta \tilde x = \td_\mu Y^{2, \xi, x,\tilde x_i, \D \tilde x}_0 = \sum_{j=1}^{d_x^1} \td_\mu Y^{2, \xi, x,\tilde x_i, \tilde e_j}_0\D \tilde x_j,\q i=1,\cds, n.
\eea
Here the last equality thanks to the fact that the equations \reff{tdxFBSDE}, \reff{tdmuFBSDE}, \reff{tdmuFBSDE2} are all linear.
Therefore, for fixed $x\in\dbR^{d^2_x}$ and for discrete $\xi$, we have
\begin{equation}\label{pamuUrep-discrete}
\partial_{\mu}\cU(\cL_{\xi},\tilde x)= \psi(x, \cL_{\xi}, \tilde x) := (\td_\mu Y^{2, \xi, x,\tilde x, \tilde e_1}_0,\cds, \td_\mu Y^{2, \xi, x,\tilde x, \tilde e_{d^1_x}}_0),\q \text{for $\cL_{\xi}$-a.e. $\tilde x\in\dbR^{d^1_x}$.}
\end{equation}

We next prove \reff{pamuUrep} in the general case. Fix $\xi,\eta\in \mathbb{L}^2(\mathcal{F}_0^1,\dbR^{d_x^1})$. One can easily construct discrete $\xi_n, \eta_n\in \dbL^2(\cF_0^1,\dbR^{d_x^1})$ such that $\dis\lim_{n\to\infty}\dbE\big[|\xi_n-\xi|^2 + |\eta_n-\eta|^2\big]=0$.
Clearly $\xi_n + \th \eta_n$ is also discrete, for any $\th\in [0,1]$. Then it follows from   \reff{pamuUrep3} and \reff{pamuUrep-discrete} that,
\beaa
&\dis U(0, x,  \cL_{\xi_n + \eta_n}) - U(0, x,  \cL_{\xi_n }) = \int_0^1 {d\over d\th} \cU(\xi_n +\th \eta_n) d\th = \int_0^1 \dbE\big[\pa_{\mu} \cU( \cL_{\xi_n +\th \eta_n},  \xi_n + \th \eta_n) \eta_n \big] d\th\\
&\dis = \int_0^1 \dbE\big[ \psi(x,\cL_{\xi_n +\th \eta_n},  \xi_n + \th \eta_n) \eta_n \big] d\th,
\eeaa
for any fixed $x\in \dbR^{d^2_x}$ and $n\ge 1$.
By the stability of the FBSDEs, it is clear that $\psi$ is continuous in all its variables. Since $U$ is also continuous, then, by sending $n\to\infty$ we obtain
\beaa
 U(0, x,  \cL_{\xi + \eta}) - U(0, x,  \cL_{\xi}) =  \int_0^1 \dbE\big[\psi(x,\cL_{\xi +\th \eta},  \xi + \th \eta )\eta\big] d\th.
\eeaa
Again since $\psi$ is continuous, we have
\beaa
 U(0, x,  \cL_{\xi + \eta}) - U(0, x,  \cL_{\xi}) = \dbE \big[\psi(x,\cL_{\xi},  \xi )\eta\big] + o(\|\eta\|_2).
\eeaa
This means that $U$ is differentiable in $\mu$ with
\beaa
 \pa_{\mu} U(0, x,  \cL_{\xi}, \tilde x) = \psi(x, \cL_{\xi}, \tilde x),  \q\mbox{for $\cL_{\xi}$-a.e. $\tilde x\in\dbR^d$}.
  \eeaa
Finally, for $\tilde x\notin \supp(\cL_{\xi})$, since $ \pa_{\mu} U(0, x,  \cL_{\xi}, \tilde x)$ can be defined arbitrarily, we may naturally use $\psi(x, \cL_{\xi}, \tilde x)$ to define it, then we obtain \reff{pamuUrep} for all $(x, \cL_{\xi}, \tilde x)\in\dbR^{d_x^2}\times\cP_2(\dbR^{d_x^1})\times\dbR^{d_x^1}$. \qed

\begin{rem}
\label{rem-W2Lip}
 We note that the boundedness of $\pa_\mu U$ implies $U$ is Lipschitz continuous in $\mu$ under $W_1$, which is stronger than the Lipschitz continuity under $W_2$. For future purpose, especially when one considers global well-posedness of the MFG master equations, we emphasize that, in the above theorem and in all the results in this paper, the $\e_0$ actually depends only on the $W_2$-Lipschitz constant of $g_i$, rather than their $W_1$-Lipschitz constant.
\end{rem}

\section{Bootstrap method for higher order regularities}
\label{sect-bootstrap}
\setcounter{equation}{0}
We now establish the higher order regularity of $U$ by bootstrap arguments, provided the coefficients are sufficiently smooth, recalling the parabolic order  and the space $\cC^k_b$   in Section \ref{sect-setting}:
\bea
\label{assum-Ckb}
\Phi_i \in \cC^{k}_b~\mbox{for some $k\ge 2$, $\Phi=b, \si, f, g$, $i=1,2$.}
\eea
Note that, unlike in Assumption \ref{assum-1}, here the derivatives of $\Phi_i$ are locally uniformly continuous, in particular are continuous in $t$. We shall always assume $T\le \e_0$ for some $\e_0$ sufficiently small.
We emphasize that $\e_0$ depends only on the bounds of the first order derivatives, not on the bounds of the higher order derivatives in \reff{assum-Ckb}. We proceed in three steps and will summarize the results in the end of this section.

{\bf Step 1.} In this step we study the higher order derivatives of $U$ with respect to $x$. We assume the coefficients are sufficiently smooth in $x$ with uniform bounded derivatives.

First, given $\xi\in \dbL^2(\cF^1_0, \dbR^{d^1_x})$, denote $\bm_t := \cL_{X^\xi_t|\cF^0_t}$, and for $\Phi= b, \si, f$, $\Phi^\bm_t(x, y, z) := \Phi_2(t,x,y,z, \bm_t)$, $g^\bm(x) := g_2(x, \bm_T)$. We note that $\|\bm_t\|_2 \le C[1+\|\mu\|_2]$. It is clear that, under Assumption \ref{assum-1},  $(b^\bm,\si^\bm, f^\bm, g^\bm)$ satisfy all the requirements in Assumption \ref{assum-0}, in particular they are $\dbF^0$-progressively measurable and  $I_0 \le C[1+\|\mu\|_2]$ for the $I_0$ in Assumption \ref{assum-0} (i).  Then, with these coefficients,  FBSDEs \reff{X2x} and \reff{FBSDEx} are the same, and thus the function $u$ in Theorem \ref{thm-pax}  identifies with $U$ defined by \reff{Uxmu}: $u_0(x)  = U(0,x,\mu)$. Applying Theorem \ref{thm-pax}, this implies that $\pa_xU(0,x,\mu) = \pa_x u_0(x)$ exists, is  locally uniformly continuous with representation, and is bounded by $R_0:= L_{g}+1$.

Next, by using the notations in \reff{FBSDEx}-\reff{tdFBSDEx}, denote
\beaa
& x' = (x^{'1}, x^{'2}) := (x, \D x),\q \Pi' = (\Pi^{'1}, \Pi^{'2}) := (\Pi, \td_x \Pi), \q  \pi' = (\pi^{'1}, \pi^{'2}) := (\pi, \td_x \pi),\\ &\Phi'_t(\pi') = (\Phi^{'1}_t(\pi'), \Phi^{'2}_t(\pi')) := (\Phi_2^{\bm}(\pi), \pa_\pi \Phi^{\bm}_{2,t}(\pi) \td_x\pi),\q \Phi = b, \si, f, g.
\eeaa
Clearly $\Pi'$ satisfies  \reff{FBSDEx} with coefficient $\Phi'$. Then formally we may derive the SDE corresponding to \reff{tdFBSDEx} for $\td_{x'} \Pi^{'x', \D x'} = (\td_{x'} \Pi^{'1, x', \D x'} , \td_{x'} \Pi^{'2, x', \D x'} )$. Note that \reff{FBSDEx} and \reff{tdFBSDEx} are decoupled, then one can easily verify that the first component of $\td_{x'}\Pi'$ identifies with the second component of $\Pi'$: $\td_{x'} \Pi^{'1, x', \D x'} =\Pi^{'2, x', \D x'}= \td_x\Pi^{x^{'1}, \D x^{'1}}$, where the last one is the solution to  \reff{tdFBSDEx} with original coefficients $\Phi^\bm_2$. Moreover, for $i, j=1,\cds, d_x^2$, we have
\bea
\label{paxij}
\pa_{x_j x_i} U(0,x,\mu) = \pa_{x_j x_i} u_0(x) = \td_{x'} Y^{'2, (x, e_i), (e_j, 0)},
\eea
which provides a representation formula for the second order derivatives $\pa_{xx} U$.

There is one technical issue in the above argument: the coefficients $\Phi^{'2}_t(\pi') := \pa_\pi \Phi^{\bm}_{2,t}(\pi^{'1}) \pi^{'2}$ are not Lipschitz continuous in $\pi'$, more precisely not in $\pi^{'1}$. This can be easily overcome by the decoupling structure of the SDEs \reff{FBSDEx} and \reff{tdFBSDEx}. Indeed, when considering the equation for $\td_{x'}\Pi^{'2}$, the process $\td_{x'}\Pi^{'1} = \Pi^{'2}=\td_x \Pi $ is already obtained, and thus the Lipschitz continuity in $\pi^{'1}$ is not needed. To be precise, fix $i, j$ and denote $\td_{x'} \Pi' = \td_{x'} \Pi^{'(x, e_i), (e_j, 0)}$. Then we have
\begin{eqnarray}\label{tdxxFBSDE}
\begin{cases}
                         \dis  \td_{x'} X^{'2}_t= \int_{0}^t \big[\langle \pa_{\pi\pi} b_s(\Pi^x_s), ( \td_x \Pi^{x,e_i}_s, \td_x \Pi^{x,e_j}_s)\rangle +\pa_\pi b_s(\Pi^x_s) \td_{x'} \Pi^{'2}_s\big] ds \ms\\
       \dis \q                 +\int_0^t  \big[\langle \pa_{\pi\pi} \si_s(\Pi^x_s), ( \td_x \Pi^{x,e_i}_s, \td_x \Pi^{x,e_j}_s)\rangle +\pa_x\si_s(\Pi^x_s) \td_{x'} X^{'2}_s + \pa_y \si_s(\Pi^x_s) \td_{x'} Y^{'2}_s \big] dB_s,\ms\\
                         \dis  \td_{x'} Y^{'2}_t=  \big[\langle \pa_{x} g(X^x_T), ( \td_x X^{x,e_i}_T, \td_xX^{x,e_j}_T)\rangle +\pa_x g(X^x_T) \td_{x'} X^{'2}_T\big]\ms\\
                         \dis\q +\int_{t}^T \big[\langle \pa_{\pi\pi} f_s(\Pi^x_s), ( \td_x \Pi^{x,e_i}_s, \td_x \Pi^{x,e_j}_s)\rangle +\pa_\pi f_s(\Pi^x_s) \td_{x'} \Pi^{'2}_s\big] ds - \int_t^T \td_{x'} Z^{'2}_s dB_s,
                 \end{cases}
\end{eqnarray}
where the bilinear operators $\langle \pa_{\pi\pi} b, ( \td_x \Pi, \td_x \Pi)\rangle$ etc. are in obvious sense.  Note that the coefficients $\pa_\pi \Phi$ of the unknowns $\td_{x'} \Pi^{'2}$ are bounded by $L_1$. Then by  \reff{Lpest1} we see that, for the same $\e_0$ in Theorem \ref{thm-pax} and for any $T\le \e_0$, the FBSDE \reff{tdxxFBSDE} is well-posed and for any $p\ge 2$,
\bea
\label{Lpest3}
 \|\td_{x'}\Pi'\|_p \le C_p \|\td_x \Pi\|_{2p}^2 \le C_p<\infty.
 \eea
 We emphasize again that, the $\e_0$ depends only on the bound $L_1$  in Theorem \ref{thm-pax}, but not on the bounds of the higher order derivatives. The $C_p$ in \reff{Lpest3}, however, relies on the latter bounds.

Following similar arguments, we can show that $U$ has bounded and  locally uniformly continuous higher order derivatives in $x$, provided that the coefficients are sufficiently smooth in $x$ with bounded derivatives.

{\bf Step 2.} We now study the further derivatives of $\pa_\mu U$. Before that, we note that since we now assume in \reff{assum-Ckb} the local uniform regularity of the coefficients, we can show that $\pa_\mu U(0,\cd,\cd,\cd)$ is also locally uniformly continuous, following similar arguments as in Step 3 in the proof of Theorem \ref{thm-pax}, see Section \ref{sect-appendix} below. 

To study the further derivatives of $\pa_\mu U$, the main idea is to rewrite the representation of $\pa_\mu U$ in the form of FBSDEs \reff{Xxi}-\reff{X2x}, with increased dimensions. First, recall the FBSDEs \reff{Xxi}, \reff{X1x}, \reff{tdxFBSDE}, \reff{tdmuFBSDE}, and note that the representation of $\pa_\mu U$ involves the distribution of $\Upsilon^{\tilde x}:=\big((X^{1, \tilde x},X^\xi),(\nabla_{\tilde x}X^{\tilde x},\nabla_{\mu}X^{\tilde x})\big)$. Then, for $\xi' = (\xi^{1}, \xi^{2}, \xi^{3}, \xi^{4})\in \dbL^2(\cF_0, \dbR^{4d^1_x})$, let $\Pi^{'\xi'}$, $\Xi^{'\xi'}=(\Pi^{'\xi'}, \cL_{X^{'\xi'}|\dbF^0})$ correspond to these four FBSDEs, with the components in the order of  \reff{Xxi}, \reff{X1x}, \reff{tdxFBSDE}, \reff{tdmuFBSDE}, but with initial condition $X^{'\xi'}_0 = \xi'$. That is, they solve \reff{Xxi} with the following coefficients $\Phi'_1$: for $\Phi = b, \si, f, g$, and for $\zeta' = (\pi^1, \pi^2, \pi^3, \pi^4, \mu')$, where $\mu'\in \cP_2(\dbR^{4d^1_x})$,
\bea
\label{Phi1'}
\Phi'_1(\zeta') = \Big(\Phi_1(\pi^1, \mu^1), ~\Phi_1(\pi^2, \mu^1), ~\pa_\pi \Phi_1(\pi^2, \mu^1) \pi^3, ~ \int_{\dbR^{4d^1_x}}\td \Phi_1(\pi^1, \mu^1; \pi^4, \tilde x') \mu'(d\tilde x')\Big).
\eea
Here $\mu^1$ denotes the first marginal of $\mu'$, and $\tilde x' = (\tilde x^1, \tilde x^2, \tilde x^3, \tilde  x^4)$.

 Next, given $\xi'$ and $x'=(x, \tilde x)\in \dbR^{d_x^2}\times \dbR^{d_x^1}$, let  $\Pi^{'2,\xi', x'}$, $\Xi^{'2,\xi',x'}=(\Pi^{'2,\xi',x'}, \cL_{X^{'\xi'}|\dbF^0})$ correspond to \reff{X2x} and  \reff{tdmuFBSDE2}, but with initial condition $X^{'2,\xi',x'}_0 = x'$. That is, they solve \reff{X2x} with the following coefficients $\Phi'_2$: for $\Phi = b, \si, f, g$, and for $\zeta' = (\pi^1, \pi^2, \mu')$,
\bea
\label{Phi2'}
\Phi'_2(\zeta') = \Big(\Phi_2(\pi^1, \mu^1), ~ \int_{\dbR^{4d^1_x}}\td \Phi_2(\pi^1, \mu^1; \pi^2, \tilde x') \mu'(dx')\Big).
\eea
Then as in \reff{Uxmu} we may define
\bea
\label{Uliftrep}
U'(0, \mu', x') := Y^{'2,\xi',x'}_0, \q \mu'\in \cP_2(\dbR^{4d_x^1}), x'\in \dbR^{d_x^2\times d_x^1},~ \cL_{\xi'} = \mu'.
\eea
In particular, for the $U$ in \reff{Uxmu}, we have
\bea
\label{U-Ulift}
\big(U, \pa_{\mu_i} U\big)(0, x, \mu, \tilde x) = U'(0, \cL_{\xi'}, x'),\q \mbox{where}\q \xi' = (\xi, \tilde x, \tilde e_i, 0), ~ x' = (x, \tilde x).
\eea

Following the arguments in Step 1 and by applying Theorem \ref{thm-pamuUrep}, we can easily derive the representations for $\pa_{x'}U'(0, \mu', x')$ and $\pa_{\mu'} U'(0, \mu', x')$. One technical difference is that, in \reff{Phi1'} the third component $\pa_\pi \Phi_1(\pi^2, \mu^1) \pi^3$ is not Lipschitz continuous in $(\pi^2, \mu^1)$ and, by \reff{tdPhi1}, the fourth component is not Lipschitz continuous in $(\pi^1, \mu^1, \tilde x'_1, \tilde x'_2)$; and in \reff{Phi2'}  the second component is not Lipschitz continuous in $(\pi^1, \mu^1, \tilde x'_1, \tilde x'_2)$. However, due to the decoupling structure, following similar arguments as in Step 1, especially  \reff{tdxxFBSDE} and \reff{Lpest3}, we can obtain the desired estimates. Moreover, let $U^{'2}$ denote the second component of $U'$. By \reff{U-Ulift} we have $\pa_{\mu_i} U = U^{'2}$ and thus
\bea
\label{papamuU}
&&\pa_x\pa_{\mu_i} U(0, x, \mu, \tilde x) = \pa_{x'_1} U'\big(0, \cL_{(\xi, \tilde x, \tilde e_i, 0)}, (x, \tilde x)\big);\nonumber\\
&& \pa_\mu\pa_{\mu_i} U(0, x, \mu, \tilde x, \bar x) = \pa_{\mu'_1} U'\big(0, \cL_{(\xi, \tilde x, \tilde e_i, 0)}, (x, \tilde x), \bar x\big);\\
&&\pa_{\tilde x}\pa_{\mu_i} U(0, x, \mu, \tilde x) = \pa_{x'_2} U'\big(0, \cL_{(\xi, \tilde x, \tilde e_i, 0)}, (x, \tilde x)\big) +  \pa_{\mu'_2} U'\big(0, \cL_{(\xi, \tilde x, \tilde e_i, 0)}, (x, \tilde x), \tilde x\big).\nonumber
\eea
Here the last one used the fact that $\pa_x (V(\d_x)) = \pa_\mu V(\d_x, x)$ for a smooth function $V$ on $\cP_2(\dbR^d)$.

Moreover, following the same arguments, we may obtain the higher order derivatives of $\pa_\mu U$ with respect to $x, \tilde x, \mu$, as well as the new state variables $\bar x$ etc arising from the higher order derivatives with respect to $\mu$, provided the coefficients are sufficient smooth.

{\bf Step 3.} We now study $\pa_t U$ and its further derivatives. First, by restricting \reff{Xxi} and \reff{X2x} on $[t, T]$ with initial value $X^{t,\xi}_t = \xi\in \dbL^2(\cF_t, \dbR^{d^1_x})$ and $X^{t, 2,x}_t = x\in \dbR^{d^2_x}$, we may define $U(t, x, \cL_\xi) := Y^{t, 2, x}_t$. Similarly, by considering  \reff{Xxi} and \reff{X1x} on $[t, T]$ we may define $U_1(t, \tilde x, \mu):= Y^{t, 1, \tilde x}_t$ for $\tilde x\in \dbR^{d_x^1}$. Note that we may apply Step 1 and Step 2 on $U_1$ as well. Then, for $i=1, 2$, $U_i$ is sufficiently smooth in terms of $x, \tilde x, \mu$. By standard arguments, see e.g. \cite[Section 5.1]{Zhang}, we have
\bea
\label{YU}
Y^\xi_t = U_1(t, X^\xi_t, \cL_{X^\xi_t|\cF^0_t}),\q Y^{1, x}_t = U_1(t, X^{1, x}_t, \cL_{X^\xi_t|\cF^0_t}), \q Y^{2, x}_t = U(t, X^{2, x}_t, \cL_{X^\xi_t|\cF^0_t}).
\eea
Moreover, denoting again $\bm_t := \cL_{X^\xi_t|\cF^0_t}$, for the $u_t(x)$ corresponding to $\Phi^\bm$ as in Step 1, we have $u_t(x) = U(t,x, \bm_t)$. Let $\rho_R$ be as in \reff{rhoR2}. Then, for $|x|, \|\mu\|_2<R$, by \reff{rhoR2} we have
\beaa
&&\big|U(0,x,\mu) - U(t, x, \mu)\big| = \Big| \dbE\big[u_0(x) - u_t(x) + U(t,x, \bm_t) - U(t,x, \mu)\big]\Big|\\
&&\le \rho_R(t) + C\dbE\big[W_2(\bm_t, \mu)\big]\le  \rho_R(t) + C\Big(\dbE\big[X^\xi_t - \xi|^2\big]\Big)^{1\over 2} \le \rho_R(t) + C_R \sqrt{t}.
\eeaa
Similarly, we can show that $U_i$ and their related derivatives in terms of $x, \tilde x, \mu$ are locally uniformly continuous in all variables, including the time variable $t$.

Fix $\e>0$ and $\xi, x, \tilde x$,  for notational convenience we denote
\beaa
& \cX^0:= (X^\xi_t, \cL_{X^\xi_t|\cF^0_t}),\q \cX^1_t :=  (X^{1, \tilde x}_t, \cL_{X^\xi_t|\cF^0_t}),\q  \cX^2_t:= (X^{2, x}_t, \cL_{X^\xi_t|\cF^0_t}),\\
&\Xi^0 := \Xi^\xi,\q \Xi^1 := \Xi^{1,\tilde x},\q \Xi^2 := \Xi^{2, x};\\
&\dis\left.\ba{c}
\dis \cY^{i,\e}_t := Y^{i}_t - U_i(\e, \cX^i_0),\q  \cZ^{i,\e,1}_t :=  Z^{i,1}_t- \partial_xU_i(\e, \cX^i_t)\sigma_i^1(t, \Xi^{i}_t),\\
\dis \cZ^{i,\e, 0}_t := Z^{i,0}_t- \partial_xU_i(\e, \cX^i_t)\sigma_i^0(t, \Xi^{i}_t) - \dbE_t\big[\partial_\mu U_i(\e, \cX^i_t,\tilde X^\xi_t) \sigma_i^{0}(t, \tilde \Xi^\xi_t)\big],
\ea\right. i=0,1,2.
\eeaa
Here, in the last two lines above we used notations $U_0 := U_1$, $U_2:= U$, $\Phi_0:= \Phi_1$ for $\Phi = b, \si, f, g$;
and $Z^i = (Z^{i,0}, Z^{i,1})$ refer to the decomposition corresponding to $(dB^0, dB^1)$, as we do for $\si = (\si^0, \si^1)$.
By applying the It\^{o} formula \reff{Ito} on $U_i(\e, \cd, \cd)$ we have
\beaa
\cY^{i,\e}_t &=& U_i(\e, \cX^i_\e) - U_i(\e, \cX^i_0) + \int_t^\e f_i(\Xi^{i}_s)ds - \int_t^\e Z^{i}_s dB_s\\
&=&\int_t^\e \Big[\partial_x U_i(\e, \cX^i_s) \cdot b_i(s, \Xi^{i}_s) + \frac{1}{2} \partial_{xx} U_i(\e, \cX^i_s) : \sigma_i \sigma_i^\top(s,\Xi^{i}_s) + f_i(\Xi^{i}_s)\Big]ds\\
&&\displaystyle+ \int_t^\e \dbE_s \Big[\partial_\mu U_i(\e, \cX^i_s,\tilde X^\xi_s) \cd b_1(s, \tilde \Xi^\xi_s) + \frac{1}{2} \partial_{\tilde x\mu} U_i(\e, \cX^i_s,\tilde X^\xi_s): \sigma_1 \sigma_1^\top(\tilde \Xi^\xi_s) \\
&&\displaystyle \qquad\q+\partial_{x\mu} U_i(\e, \cX^i_s,\tilde X^\xi_s): \sigma^0_i(s, \Xi^{i}_s) (\sigma_1^0)^\top(s, \tilde \Xi^\xi_s)\\
&&\displaystyle \qquad\q +\frac{1}{2}\partial_{\mu\mu}U_i(\e, \cX^i_s,\tilde X^\xi_s,\bar X^\xi_s) : \sigma_1^0(s, \tilde \Xi^\xi_s)(\sigma_1^0)^\top(s,\bar \Xi^\xi_s)\Big] ds - \int_t^\e \cZ^{i,\e}_s dB_s.
\eeaa
By standard arguments, one can easily see that, for a constant $C>0$ independent of $\e$,
\beaa
\dbE\Big[\sup_{0\le t\le \e} |\cY^{i,\e}_t|^2 + \int_0^\e |\cZ^{i,\e}_t|^2dt\Big] \le C\e^2.
\eeaa
This implies, for $i=0,1,2$,
\beaa
&\dis{1\over \e} \dbE\Big[\int_0^\e |Y^i_t - U_i(\e, \cX^i_0)|dt\Big] \le C\e,\q {1\over \e} \dbE\Big[\int_0^\e |Z^{i,1}_t - \partial_xU_i(\e, \cX^i_t)\cdot\sigma_i^1(t, \Xi^{i}_t)|dt\Big] \le C\sqrt{\e},\\
&\dis {1\over \e} \dbE\Big[\int_0^\e \big|Z^{i,0}_t - \partial_xU_i(\e, \cX^i_t)\cdot\sigma_i^0(t, \Xi^{i}_t) - \dbE_t\big[\partial_\mu U_i(\e, \cX^i_t,\tilde X^\xi_t) \cd \sigma_i^{0}(t, \tilde \Xi^\xi_t)\big]\big|dt\Big] \le C\sqrt{\e}
\eeaa
 Now by the desired regularity of $U_i$ and the coefficients, especially their temporal regularity, we can easily show that, by abusing the notation $x$ and $\tilde x$, for $i=1,2$ and $x\in \dbR^{d_x^i}$,
\beaa
\lim_{\e\to 0} {1\over \e}\big[U_i(0, x,\mu) - U_i(\e, x,\mu)\big] = \lim_{\e\to 0} {1\over \e} \cY^{i,\e}_0 = \lim_{\e\to 0} {1\over \e} \dbE[\cY^{i,\e}_0] = \dbL^i U_i(0, x, \mu),
\eeaa
where, for $i=1,2$,
\bea
\label{LU1}
&&\dbL^i U_i(t, x, \mu) :=  \Big[\partial_x U_i \cdot b^{U_i}_i + \frac{1}{2} \partial_{xx} U_i : \sigma^{U_i}_i (\sigma^{U_i}_i)^\top + f^{U_i}_i\Big](t,x,\mu)\nonumber\\
&&\displaystyle+ \int_{\dbR^{d_x^1}} \Big[\partial_\mu U_i(t, x, \mu, \tilde x) \cd b^{U_1}_1(t, \tilde x, \mu) + \frac{1}{2} \partial_{\tilde x\mu} U_i(t,x,\mu, \tilde x): \sigma^{U_1}_1 (\sigma^{U_1}_1)^\top(t, \tilde x, \mu) \nonumber\\
&&\displaystyle \qquad\qq+\partial_{x\mu} U_i(t,x,\mu, \tilde x): \sigma^{U_i,0}_i(t,x,\mu) (\sigma_1^{U_1,0})^\top(t,\tilde x, \mu)\Big]\mu(d\tilde x)\\
&&\displaystyle+\frac{1}{2}\int_{\dbR^{d_x^1}\times \dbR^{d_x^1}}\partial_{\mu\mu}U_i(t,x,\mu,\tilde x, \bar x) : \sigma_1^{U_1,0}(t, \tilde x, \mu)(\sigma_1^{U_1,0})^\top(t,\bar x, \mu)\mu(d\tilde x)\mu(d\bar x),\nonumber
\eea
and for $\Phi = b, f$, and $i=1,2$,
\bea
\label{LU2}
&&\si^{U_i}_i(t,x,\mu) := \si_i\Big(t,x, U_i(t,x,\mu),~ \mu\Big),\nonumber\\
&& \Phi^{U_i}_i(t,x,\mu):=   \Phi_i\Big(t,x, U_i(t,x,\mu), \cV^{U_i}_i(t,x,\mu),~ \mu\Big),\nonumber\\
&&\cV_i^{U_i,1} (t,x,\mu) := \partial_xU_i(t,x,\mu)\cdot\sigma_i^{U_i,1}(t, x,\mu),\\
&&\cV_i^{U_i,0} (t,x,\mu) := \partial_xU_i(t,x,\mu)\cdot\sigma_i^{U_i,0}(t, x,\mu) + \int_{\dbR^{d^1_x}}\partial_\mu U_i(t,x,\mu, \tilde x) \cd \sigma_1^{U_1, 0}(t, \tilde x, \mu) \mu(d\tilde x).\nonumber
\eea
This means that $U_i$ has right time derivative at $(0,x,\mu)$:
$
\pa_{t+}U_i(0, x,\mu) = - \dbL^i U_i(0, x, \mu).
$
Similarly, for any $t<T$,  $U_i$ has right time derivative $\pa_{t+}U_i(t, x,\mu) = - \dbL^i U_i(t, x, \mu)$. By Step 1, Step 2,  and the arguments in the beginning of this step,  $\dbL^i U_i(t, x, \mu)$ is continuous in $(t,x,\mu)$. Then $U_i$ is continuously differentiable in $t$:
\bea
\label{patU}
\pa_{t}U_i(t, x,\mu) = - \dbL^i U_i(t, x, \mu).
\eea

Clearly, when the coefficients are sufficiently smooth, by Step 1 and Step 2, the right side above has higher order derivatives in $(x, \mu)$. This implies that $\pa_t U_i$ has higher order derivatives in $(x, \mu)$. Moreover, note that differential operators can commute, provided that the derivatives are continuous. Then, when the coefficients are also continuously differentiable in $t$, we see that $\dbL^i U_i(t, x, \mu)$ is continuous differentiable in $t$. This implies immediately that $U_i$ is twice differentiable in $t$: $\pa_{tt}U_i(t, x,\mu) = - \pa_t \dbL^i U_i(t, x, \mu)$.
Repeat the arguments we see that $U_i$ has higher order derivatives in $t$ as well, provided the coefficients are sufficiently smooth.

Finally, since $U = U_2$, we obtain the desired regularity of $U$. We shall note that its regularity relies on the regularity of $U_1$, since $\dbL^2 U_2$ involves $U_1$ as well.
\qed

We now summarize our results.
\begin{thm}
\label{thm-Ck}
Let Assumption \ref{assum-1} hold and  $T\le \e_0$ for the $\e_0$ specified in Theorem \ref{thm-pamuUrep}. Assume further that \reff{assum-Ckb} holds for some $k\ge 2$. Then $U_1, U_2=U$ are in $\cC^k_b$ and satisfy \reff{patU}.
\end{thm}

\section{An alternative representation  via master equation}
\label{sect-alterrep}
\setcounter{equation}{0}

In this section we provide an alternative representation formula for $\pa_{\mu}U=\pa_{\mu}U_2$, as well as $\pa_x U$ which is a lot easier, by using the master equations \eqref{patU}. The idea is, instead of differentiating FBSDEs \reff{Xxi}-\reff{X2x} with respect to $\xi$ as we did in \reff{pamuUrep1}, we differentiate the master equation \reff{patU} with respect to $\mu$ and use the new equation for $\pa_\mu U_2$ to obtain a representation formula.  The latter equation has been used by \cite{CDLL, BLL}.\footnote{More precisely, \cite{CDLL} considered the equation for the linear functional derivative ${d\over d\mu} U_2$.}  For this purpose, we assume $U_1, U_2\in \cC_b^3$, and to focus on the main idea, we assume for simplicity that 
\bea
\label{simplicityD}
d_x^1= d^2_x = d_y^1= d^2_y =1.
\eea
We emphasize that this is just for notational simplicity. All the results in this section can be easily extended to the multiple dimensional setting without any significant difficulty (other than the complicate notations), and we leave details to interested readers.

We start with the representation for $\pa_x U$, which is easy but illustrates the idea well. Denote
\bea
\label{VD}
V_i(t,x,\mu) := \pa_x U_i(t,x,\mu),\q \td_x Y^{i,x}_t := V_i(t, X^{i,x}_t, \cL_{X^\xi_t|\cF^0_t}),\q i=1,2.
\eea
Differentiate both sides of \reff{patU} in $x$ and omit the variables $(t,\mu)$, we have, 
\bea
\label{LVD}
&& \pa_t V_i(x) + \big[\partial_x V_i  b^{U_i}_i + \frac{1}{2} \partial_{xx} V_i |\sigma^{U_i}_i|^2\big](x)   \nonumber\\
&&\q +\int_{\dbR} \Big[\partial_\mu V_i(x, \bar x)  b^{U_1}_1(\bar x) + \frac{1}{2} \partial_{\bar x\mu} V_i(x,\bar x) |\sigma^{U_1}_1 (\bar x)|^2\partial_{x\mu} V_i(x,\bar x) \sigma^{U_i,0}_i(x) \sigma_1^{U_1,0}(\bar x) \Big]\mu(d\bar x)\nonumber\\
&& \q+ \frac{1}{2}\int_{\dbR^2}\partial_{\mu\mu}V_i(x,\bar x, \hat x) \sigma_1^{U_1,0}(\bar x) \sigma_1^{U_1,0}(\hat x)\mu(d\bar x)\mu(d\hat x) + \Psi^V_i(x) =0, \q\mbox{where}\\
&& \Psi^V_i(x):= \big[V_i \pa_x( b^{U_i}_i) + \partial_{x} V_i  \sigma^{U_i}_i\cd \pa_x(\sigma^{U_i}_i ) +   \pa_x(f^{U_i}_i )\big] (x)\nonumber\\
&&\qq\qq +  \int_\dbR \partial_{\mu} V_i(x,\bar x) \pa_x (\sigma^{U_i,0}_i)(x)\sigma_1^{U_1,0}(\bar x)\mu(d\bar x).\nonumber
\eea
On the other hand, applying the It\^{o} formula and utilizing \reff{LVD}, we obtain, omitting $(t, \cL_{X^\xi_t|\cF^0_t})$, 
\bea
\label{ItoDx}
\left.\ba{lll}
\dis d \td_x Y^{i,x}_t = \td_x Z^{i,x}_t dB_t - \Psi^V_i (X^{i,x}_t) dt, \q \td_x Z^{i,x,j}_t := \phi_x^{i,j}(X^{i,x}_t), \q j=0,1,\\
\dis \phi^{i,1}_x(x) := \pa_x V_i \si_i^{U_i,1} (x),\q \phi^{i,0}_x(x) := \pa_x V_i \si_i^{U_i,0} (x) + \int_\dbR \pa_\mu V_i(x, \bar x) \si_1^{U_1,0} (\bar x)\mu(dx).
\ea\right.
\eea

Note that, for $\Phi=b, f$,
\beaa
&&\pa_x(\sigma^{U_i}_i ) = (\pa_x \si_i)^{U_i}  + (\pa_y \si_i)^{U_i}V_i,\nonumber\\
&&\pa_x (\Phi^{U_i}_i ) = (\pa_x \Phi_i)^{U_i}  + (\pa_y \Phi_i)^{U_i} V_i  +  (\pa_z \Phi_i)^{U_i}  \cd \big[\phi^i_x  + V_i \pa_x(\sigma^{U_i}_i ) \big],\\
&&\Psi^V_i :=  V_i \pa_x( b^{U_i}_i) + \phi^i_x \cd \pa_x(\sigma^{U_i}_i ) +   \pa_x(f^{U_i}_i ). \nonumber
\eeaa
Recall \reff{YU} and $\Pi = (X, Y, Z)$, and denote $\td_x \Pi := (\td_x Y, \td_x Z)$. Note that, for any $\f(x,y)$ and $\psi(x,y, z)$ (omitting $(t,\mu)$ again),
\beaa
\f^{U_i}(X^{i,x}_t) = \f(X^{i,x}_t, Y^{i,x}_t),\q \psi^{U_i}(X^{i,x}_t) = \psi(\Pi^{i,x}_t).
\eeaa
Introduce the notations $\G_x$, for $\Phi=b, f$,
\bea
\label{GammaDx}
&&\pa_x(\sigma^{U_i}_i ) (X^{i,x}_t) = \pa_x \si_i(X^{i,x}_t, Y^{i,x}_t) + \pa_y \si_i(X^{i,x}_t, Y^{i,x}_t) \td_x Y^{i,x}_t =: \G^{\si_i}_x(\Pi^{i,x}_t, \td_x \Pi^{i,x}_t),\nonumber\\
&&\pa_x (\Phi^{U_i}_i ) (X^{i,x}_t) =\pa_x \Phi_i(\Pi^{i,x}_t) + \pa_y \Phi_i(\Pi^{i,x}_t) \td_x Y^{i,x}_t \nonumber\\
&&\q +  \pa_z \Phi_i(\Pi^{i,x}_t) \cd \big[\td_x Z^{i,x}_t + \td_x Y^{i,x}_t\G^{\si_i}_x(\Pi^{i,x}_t, \td_x \Pi^{i,x}_t)\big] =: \G^{\Phi_i}_x(\Pi^{i,x}_t, \td_x \Pi^{i,x}_t),\nonumber\\
&&\Psi^V_i(X^{i,x}_t) = \Big[ \td_x Y^{i,x}_t  \G^{b_i}_x + \td_x Z^{i,x}_t \cd \G^{\si_i}_x +  \G^{f_i}_x\Big](\Pi^{i,x}_t, \td_x \Pi^{i,x}_t) =: \G^i_x(\Pi^{i,x}_t, \td_x \Pi^{i,x}_t).\nonumber
\eea
Putting together we obtain the following BSDE for representation of $\pa_x U(0, x, \mu) = \td_x Y^{2,x}_0$, and for later purpose we consider two additional BSDEs, in addition to \reff{Xxi}-\reff{X2x}-\reff{X1x}  and again omitting $(t, \cL_{X^\xi_t|\cF^0_t})$:
\bea
\label{BSDEDx}
&&d\td_x Y^{1,\tilde x}_t = \td_x Z^{1, \tilde x}_t  dB_t  - \G^1_x\big(\Pi^{1, \tilde x}_t, \td_x \Pi^{1, \tilde x}_t\big) dt,\q \td_x Y^{1,\tilde x}_T = \pa_x g_1(X^{1,\tilde x}_T);\nonumber\\
&&d\td_x Y^{1,\xi}_t = \td_x Z^{1,\xi}_t  dB_t  - \G^1_x\big(\Pi^\xi_t, \td_x \Pi^{1,\xi}_t\big)dt,\q\td_x Y^{1,\xi}_T = \pa_x g_1(X^\xi_T);\\
&&d\td_x Y^{2,x}_t = \td_x Z^{2,x}_t  dB_t  - \G^2_x\big(\Pi^{2,x}_t, \td_x \Pi^{2,x}_t\big) dt,\q \td_x Y^{2,x}_T = \pa_x g_2(X^{2,x}_T).\nonumber
\eea

We now turn to $\pa_\mu U$, which is much more involved, but the main idea is similar. Denote
\bea
\label{WD}
W_i(t,x,\mu, \tilde x) := \pa_\mu U_i(t,x,\mu, \tilde x),\q \td_\mu Y^{i,x, \tilde x}_t := W_i(t, X^{i,x}_t, \cL_{X^\xi_t|\cF^0_t}, \tilde X^{1,\tilde x}_t),\q i=1,2.
\eea
Differentiate both sides of \reff{patU} in $\mu$ and again omit the variables $(t,\mu)$, we have, 
\bea
\label{LWD}
&&\pa_t W_i(x, \tilde x) + \partial_x W_i (x, \tilde x) b^{U_i}_i(x) + \frac{1}{2} \partial_{xx} W_i(x, \tilde x) |\sigma^{U_i}_i(x)|^2 + \partial_{\tilde x} W_i (x, \tilde x) b^{U_1}_1(\tilde x)\nonumber\\
&&\q  + \frac{1}{2} \partial_{\tilde x \tilde x} W_i(x, \tilde x) |\sigma^{U_1}_1(\tilde x)|^2 + \pa_{x\tilde x}W_i(x, \tilde x) \si_i^{U_i,0}(x) \si^{U_1,0}_1(\tilde x) \nonumber\\
&&\q +  \int_{\dbR} \Big[\partial_{\tilde x\mu}W_i(x,\tilde x, \bar x) \sigma^{U_1,0}_1(\tilde x)  \sigma_1^{U_1,0}(\bar x)+\partial_\mu W_i(x, \tilde x, \bar x)  b^{U_1}_1(\bar x) \nonumber \\
&&\q + \frac{1}{2} \partial_{\bar x\mu} W_i(x, \tilde x, \bar x) |\sigma^{U_1}_1 (\bar x)|^2+\partial_{x\mu}W_i(x,\tilde x, \bar x) \sigma^{U_i,0}_i(x)  \sigma_1^{U_1,0}(\bar x)\Big]\mu(d\bar x)\\
&&\q +\frac{1}{2}  \int_{\dbR^2}  \partial_{\mu\mu}W_i(x,\tilde x, \bar x, \hat x)  \sigma_1^{U_1,0}(\bar x)\sigma_1^{U_1,0}(\hat x)\mu(d\bar x)\mu(d\hat x)+ \Psi^W_i(x, \tilde x)=0,\q\mbox{where}\nonumber\\
&&\Psi^W_i(x, \tilde x):=  V_i(x) \pa_\mu( b^{U_i}_i)(x, \tilde x) + \partial_{x} V_i (x)\sigma^{U_i}_i(x) \cd \pa_\mu(\sigma^{U_i}_i )(x, \tilde x) +   \pa_\mu(f^{U_i}_i )(x, \tilde x)\nonumber\\
&&\q + W_i(x,\tilde x) \pa_x(b_1^{U_1})(\tilde x) + \pa_{\tilde x} W_i(x, \tilde x)\si_1^{U_1}(\tilde x) \cd \pa_x (\si_1^{U_1})(\tilde x)\nonumber \\
&&\q + \pa_{ x} W_i(x, \tilde x)\si_i^{U_i,0}(x)  \pa_x (\si_1^{U_1,0})(\tilde x) + \int_\dbR \Big[\pa_\mu W_i(x, \tilde x, \bar x)\si_1^{U_1,0}(\bar x)  \pa_x (\si_1^{U_1,0})(\tilde x) \nonumber \\
&&\q + W_i(x, \bar x) \pa_\mu (b^{U_1}_1)(\bar x, \tilde x) + \partial_{\bar x} W_i(x,\bar x) \sigma^{U_1}_1 (\bar x)\cd \pa_\mu(\si^{U_1}_1)(\bar x, \tilde x) \nonumber\\
&&\q +\pa_x W_i(x, \bar x) \big[\pa_\mu (\si_i^{U_i,0})(x, \tilde x) \si_1^{U_1,0}(\bar x) + \si_i^{U_i,0}(x) \pa_\mu (\si^{U_1,0}_1)(\bar x, \tilde x)\Big]\mu(d\bar x)\nonumber\\
&&\q + \int_{\dbR^2}  \pa_\mu W_i(x, \bar x, \hat x) \si_1^{U_1,0}(\hat x) \pa_\mu(\si_1^{U_1,0})(\bar x, \tilde x)\mu(d\bar x)\mu(d\hat x). \nonumber
\eea
On the other hand, apply the It\^{o} formula and utilize \reff{LWD}, we obtain, omitting $(t, \cL_{X^\xi_t|\cF^0_t})$, 
\bea
\label{ItoDmu}
&&\!\!\!\!\!\! d\td_\mu Y^{i,x, \tilde x}_t = \td_\mu Z^{i,x, \tilde x,1}_t  dB^1_t + \td_\mu Z^{'i,x, \tilde x,1}_td\tilde B^1_t  + \td_\mu Z^{i,x, \tilde x,0}_t  dB^0_t - \Psi^W_i (X^{i,x}_t,  \tilde X^{1,\tilde x}_t) dt,\nonumber\\
&&\!\!\!\!\!\! \td_\mu Z^{i,x, \tilde x,1}_t := \phi^{i,1}_\mu(X^{i,x}_t, \tilde X^{1,\tilde x}_t),~ \td_\mu Z^{'i,x, \tilde x,1}_t := \phi^{'i,1}_\mu(X^{i,x}_t, \tilde X^{1,\tilde x}_t),~ \td_\mu Z^{i,x, \tilde x,0}_t := \phi^{i,0}_\mu(X^{i,x}_t, \tilde X^{1,\tilde x}_t),\nonumber\\
&&\!\!\!\!\!\! \phi^{i,1}_\mu(x, \tilde x):= \pa_x W_i(x, \tilde x)\si_i^{U_i,1} (x),\q \phi^{'i,1}_\mu(x, \tilde x) := \pa_{\tilde x} W_i(x, \tilde x)\si_1^{U_1,1} (\tilde x),\\
&&\!\!\!\!\!\! \phi^{i,0}_\mu(x, \tilde x) := \pa_x W_i(x, \tilde x)\si_i^{U_i,0} (x) + \pa_{\tilde x} W_i(x, \tilde x)\si_1^{U_1,0} (\tilde x) + \int_\dbR \pa_\mu W_i(x, \tilde x, \bar x) \si_1^{U_1,0}(\bar x)\mu(d\bar x).\nonumber
\eea
Note that, for $\Phi=b, f$,
\bea
\label{GammaDmu}
&&\!\!\!\!\!\!\! \pa_\mu(\sigma^{U_i}_i )(x, \tilde x)=  (\pa_\mu \si_i)^{U_i}(x, \tilde x)  + (\pa_y \si_i)^{U_i}(x) W_i(x, \tilde x),\nonumber\\
&&\!\!\!\!\!\!\! \pa_\mu(\Phi^{U_i}_i )(x, \tilde x) =  (\pa_\mu \Phi_i)^{U_i}(x,\tilde x)  + (\pa_y \Phi_i)^{U_i} (x) W_i(x, \tilde x) \nonumber\\
&&\!\!\!\!\!\!\! \q + (\pa_z \Phi_i)^{U_i}(x) \cd \big[\phi^i_\mu(x, \tilde x)  + V_i (x)\pa_\mu (\sigma^{U_i}_i )(x, \tilde x) \big]  \nonumber\\
&&\!\!\!\!\!\!\! \q +(\pa_{z_0} \Phi_i)^{U_i}(x) \Big[W_i(x, \tilde x) \pa_x(\si_1^{U_1,0})(\tilde x) + \int_\dbR W_i(x, \bar x) \pa_\mu(\si_1^{U_1,0})(\bar x, \tilde x) \mu(d\bar x)\Big],\nonumber\\
&&\!\!\!\!\!\!\! \Psi^W_i(x, \tilde x):=  V_i(x) \pa_\mu( b^{U_i}_i)(x, \tilde x) + \phi^i_x(x) \cd \pa_\mu(\sigma^{U_i}_i )(x, \tilde x) +   \pa_\mu(f^{U_i}_i )(x, \tilde x)\\
&& + W_i(x,\tilde x) \pa_x(b_1^{U_1})(\tilde x) + \phi^{'i,1}_\mu(x, \tilde x)  \pa_x (\si_1^{U_1,1})(\tilde x) + \phi^{i,0}_\mu(x, \tilde x)  \pa_x (\si_1^{U_1,0})(\tilde x) \nonumber\\
&&+ \int_\dbR \Big[ W_i(x, \bar x) \pa_\mu (b^{U_1}_1)(\bar x, \tilde x) + \phi^{'i,1}_\mu(x, \bar x)  \pa_\mu (\si_1^{U_1,1})(\bar x, \tilde x) + \phi^{i,0}_\mu(x, \bar x)  \pa_\mu (\si_1^{U_1,0})(\bar x,\tilde x) \Big]\mu(d\bar x). \nonumber
\eea
Denote $\td_\mu \Pi := (\td_\mu Y, \td_\mu Z^1, \td_\mu Z^{'1}, \td_\mu Z^0)$. We introduce the notations $\G_\mu$, for $\Phi=b, f$,
\bea
\label{GammaDmu2}
&&\pa_\mu(\sigma^{U_i}_i )(X^{i,x}_t, \tilde X^{1,\tilde x}_t) =  \pa_\mu \si_i (X^{i,x}_t, Y^{i,x}_t, \tilde X^{1,\tilde x}_t)+ \pa_y \si_i(X^{i,x}_t, Y^{i,x}_t) \td_\mu Y^{i,x, \tilde x}_t\nonumber\\
&&\qq =: \G^{\si_i}_\mu(\Pi^{i,x}_t, \td_x \Pi^{i,x}_t, \td_\mu \Pi^{i,x, \tilde x}_t, \tilde X^{1,\tilde x}_t), \nonumber\\
&&\pa_\mu(\Phi^{U_i}_i )(X^{i,x}_t, \tilde X^{1,\tilde x}_t) =  \pa_\mu \Phi_i (\Pi^{i,x}_t, \tilde X^{1,\tilde x}_t) + \pa_y \Phi_i(\Pi^{i,x}_t) \td_\mu Y^{i,x, \tilde x}_t\\
&&\q + \pa_{z} \Phi_i(\Pi^{i,x}_t)\cd \big[\td_\mu Z^{i,x, \tilde x}_t +\td_x Y^{i,x}_t\G^{\si_i}_\mu(\Pi^{i,x}_t, \td_x \Pi^{i,x}_t, \td_\mu \Pi^{i,x, \tilde x}_t, \tilde X^{1,\tilde x}_t)\big] \nonumber\\
&&\q +\pa_{z_0} \Phi_i(\Pi^{i,x}_t) \Big[\td_\mu Y^{i,x,\tilde x}_t \G^{\si_1,0}_x(\tilde \Pi^{1, \tilde x}_t, \widetilde {\td_x  \Pi}^{1,\tilde x}_t) + \dbE_{\cF_t\times \tilde \cF^1_t}\big[ ~\overline{\td_\mu Y}^{i,x, \xi}_t \G_\mu^{\si_1,0}(\bar \Pi^{\xi}_t, \tilde X^{1,\tilde x}_t)\big]\Big] \nonumber\\
&&\q =:\G^{\Phi_i}_\mu(\Pi^{i,x}_t,  \td_x \Pi^{i,x}_t, \td_\mu \Pi^{i,x, \tilde x}_t, \tilde \Pi^{1,\tilde x}_t, \widetilde {\td_x  \Pi}^{1,\tilde x}_t; \Pi^\xi, \td_\mu Y^{i,x,\xi}).\nonumber
\eea
Here $\Pi^\xi := (X^\xi, Y^\xi, Z^\xi)$ as in \reff{Xxi} (and noting that $\Pi^\xi = \Pi^{1,\xi}$), $\bar\Pi^\xi$ is a conditionally independent copy of $\Pi^\xi$, conditional on $\dbF^0$;  $\td_\mu Y^{i,x,\xi} = W_i(X^{i,x}_t, \tilde X^\xi_t)$, and $\overline{\td_\mu Y}^{i,x,\xi} = W_i(X^{i,x}_t, \bar X^\xi_t)$ is a conditionally independent copy of $\td_\mu Y^{i,x,\xi} $, conditional on $\dbF$. We note that $\G^{\Phi_i}_\mu$ depends on the last two variables $(\Pi^\xi, \td_\mu Y^{i,x,\xi})$ through their conditional law, while on the first five variables in the pointwise sense. Similarly, by \reff{GammaDmu} we have
\bea
\label{GammaDmu3}
&& \Psi^W_i(X^{i,x}_t, \tilde X^{1, \tilde x}_t) =   \Big[ \td_\mu Y^{i,x, \tilde x}_t  \G^{b_1}_x + \td_\mu Z^{'i,x, \tilde x,1}_t \G^{\si_1,1}_x + \td_\mu Z^{i,x, \tilde x,0}_t \G^{\si_1,0}_x\Big](\tilde \Pi^{1,\tilde x}_t, \td_x \tilde \Pi^{1,\tilde x}_t)\nonumber\\
&& + \Big[ \td_x Y^{i,x}_t  \G^{b_i}_\mu + \td_x Z^{i,x}_t \cd \G^{\si_i}_\mu +  \G^{f_i}_\mu\Big](\Pi^{i,x}_t,  \td_x \Pi^{i,x}_t, \td_\mu \Pi^{i,x, \tilde x}_t, \tilde \Pi^{1,\tilde x}_t, \widetilde {\td_x  \Pi}^{1,\tilde x}_t; \Pi^\xi, \td_\mu Y^{i,x,\xi})\nonumber\\
&& + \dbE_{\cF_t\times \tilde\cF^1_t}\Big[ \big[\overline{\td_\mu Y}^{i,x, \xi}_t  \G^{b_1}_\mu + \overline{\td_\mu Z}^{'i,x, \xi,1}_t \G^{\si_1,1}_\mu + \overline{\td_\mu Z}^{i,x, \xi, 0}_t \G^{\si_1,0}_\mu\big]\\
&&\qq\qq\qq (\bar \Pi^\xi_t,   \overline{\td_x  \Pi}^{1,\xi}_t, \overline{\td_\mu \Pi}^{1,\xi, \tilde x}_t, \tilde \Pi^{1,\tilde x}_t, \widetilde {\td_x  \Pi}^{1,\tilde x}_t; \Pi^\xi, \td_\mu Y^{1, \xi,\xi})\Big]\nonumber\\
&& =: \G^i_\mu\Big(\Pi^{i,x}_t,  \td_x \Pi^{i,x}_t, \td_\mu \Pi^{i,x, \tilde x}_t, \tilde \Pi^{1,\tilde x}_t, \widetilde {\td_x  \Pi}^{1,\tilde x}_t; \Pi^\xi, \td_x \Pi^{1,\xi}, \td_\mu \Pi^{i,x,\xi}, \td_\mu \Pi^{1,\xi, \tilde x}, \td_\mu Y^{1,\xi,\xi}\Big). \nonumber
\eea
Putting together we obtain the following McKean-Vlasov BSDE system for representation of $\pa_\mu U(0, x, \mu, \tilde x) = \td_\mu Y^{2,x,\tilde x}_0$, in addition to \reff{Xxi}-\reff{X2x}-\reff{X1x}-\reff{BSDEDx}  and again omitting $(t, \cL_{X^\xi_t|\cF^0_t})$:
\bea
\label{BSDEDmu}
&&d\td_\mu Y^{2,x, \tilde x}_t = \td_\mu Z^{2,x, \tilde x,1}_t  dB^1_t + \td_\mu Z^{'2,x, \tilde x,1}_td\tilde B^1_t + \td_\mu Z^{2,x, \tilde x,0}_t  dB^0_t\nonumber\\
&&\q  -\G^2_\mu\Big(\Pi^{2,x}_t,  \td_x \Pi^{2,x}_t, \td_\mu \Pi^{2,x, \tilde x}_t, \tilde \Pi^{1,\tilde x}_t, \widetilde {\td_x  \Pi}^{1,\tilde x}_t; \Pi^\xi, \td_x \Pi^{1,\xi}, \td_\mu \Pi^{2,x,\xi}, \td_\mu \Pi^{1,\xi, \tilde x}, \td_\mu Y^{1,\xi,\xi}\Big)dt;\nonumber\\
&&d\td_\mu Y^{2,x, \xi}_t = \td_\mu Z^{2,x, \xi,1}_t  dB^1_t + \td_\mu Z^{'2,x, \xi,1}_td\tilde B^1_t + \td_\mu Z^{2,x, \tilde \xi,0}_t  dB^0_t\nonumber\\
&&\q  -\G^2_\mu\Big(\Pi^{2,x}_t,  \td_x \Pi^{2,x}_t, \td_\mu \Pi^{2,x, \xi}_t, \tilde \Pi^{\xi}_t, \widetilde {\td_x  \Pi}^{1,\xi}_t; \Pi^\xi, \td_x \Pi^{1,\xi}, \td_\mu \Pi^{2,x,\xi}, \td_\mu \Pi^{1,\xi, \xi}, \td_\mu Y^{1,\xi,\xi}\Big)dt;\nonumber\\
&&d\td_\mu Y^{1, \xi, \tilde x}_t = \td_\mu Z^{1,\xi, \tilde x,1}_t  dB^1_t + \td_\mu Z^{'1,\xi, \tilde x,1}_td\tilde B^1_t + \td_\mu Z^{1,\xi, \tilde x,0}_t  dB^0_t\\
&&\q  -\G^1_\mu\Big(\Pi^\xi_t,  \td_x \Pi^{1,\xi}_t, \td_\mu \Pi^{1,\xi, \tilde x}_t, \tilde \Pi^{1,\tilde x}_t, \widetilde {\td_x  \Pi}^{1,\tilde x}_t; \Pi^\xi, \td_x \Pi^{1,\xi}, \td_\mu \Pi^{1,\xi,\xi}, \td_\mu \Pi^{1,\xi, \tilde x}, \td_\mu Y^{1,\xi,\xi}\Big)dt;\nonumber\\
&&d\td_\mu Y^{1, \xi, \xi}_t = \td_\mu Z^{1,\xi, \xi,1}_t  dB^1_t + \td_\mu Z^{'1,\xi, \xi,1}_td\tilde B^1_t + \td_\mu Z^{1,\xi, \xi,0}_t  dB^0_t\nonumber\\
&&\q  -\G^1_\mu\Big(\Pi^\xi_t,  \td_x \Pi^{1,\xi}_t, \td_\mu \Pi^{1,\xi, \xi}_t, \tilde \Pi^\xi_t, \widetilde {\td_x  \Pi}^{1,\xi}_t; \Pi^\xi, \td_x \Pi^{1,\xi}, \td_\mu \Pi^{1,\xi,\xi}, \td_\mu \Pi^{1,\xi, \xi}, \td_\mu Y^{1,\xi,\xi}\Big)dt;\nonumber\\
&&\td_\mu Y^{2,x, \tilde x}_T = \pa_\mu g_2(X^{2,x}_T, \tilde X^{1,x}_T),\q \td_\mu Y^{2,x, \xi}_T = \pa_\mu g_2(X^{2,x}_T, \tilde X^\xi_T), \nonumber\\
&& \td_\mu Y^{1,\xi, \tilde x}_T  = \pa_\mu g_1(X^\xi_T, \tilde X^{1,\tilde x}_T),~ \td_\mu Y^{1,\xi, \xi}_T = \pa_\mu g_1(X^\xi_T, \tilde X^\xi_T).\nonumber
\eea

Now the following result is immediate.
\begin{thm}
\label{thm-alterrep}
Let all the conditions in Theorem \ref{thm-Ck} hold with $k=3$ and assume \reff{simplicityD}. Then, the BSDE systems \reff{BSDEDx} and \reff{BSDEDmu} are well-posed, and  we have the following alternative representation formulae for the first order derivatives of $U=U_2$: for any $x, \tilde x\in \dbR$, $\mu\in \cP_2(\dbR)$, and $\xi\in \dbL^2(\cF_0, \mu)$,
\bea\label{eq:alterrep}
\pa_x U(0,x,\mu) = \td_x Y^{2,x}_0,\q \pa_\mu U(0,x,\mu, \tilde x) = \td_\mu Y^{2,x,\tilde x}_0.
\eea
\end{thm}

\begin{rem}
\label{rem-alterrep}
(i) In the above theorem we are using the regularity $U_1, U_2\in \cC^3_b$, obtained from Theorem \ref{thm-Ck}, to prove the representation formula \reff{eq:alterrep}. One may prove \reff{eq:alterrep} under weaker conditions, without using the master equation, as we did in Theorem \ref{thm-pamuUrep} for the main representation formula \reff{pamuUrep}. From our experience, it is somewhat easier to verify \reff{pamuUrep}, since it is more adapted to the perturbation of the representation \reff{Uxmu} of $U$, see \reff{pamuUrep1}.

(ii) In  \reff{pamuUrep} we used  systems \eqref{tdxFBSDE}-\eqref{tdmuFBSDE}-\eqref{tdmuFBSDE2} which are forward-backward but are linear. Here we use only backward systems \reff{BSDEDx} and \reff{BSDEDmu}. However, we emphasize that the systems \reff{BSDEDx} are \reff{BSDEDmu} are nonlinear, and thus their well-posedness is in fact harder than the linear systems \eqref{tdxFBSDE}-\eqref{tdmuFBSDE}-\eqref{tdmuFBSDE2}, and in particular they require higher regularity than Assumption \ref{assum-1} on the coefficients.
\end{rem}

\section{Extended mean field games}
\label{sect-EMFG}
\setcounter{equation}{0}

We first introduce the extended mean field game (EMFG). Consider the probabilistic setting in the beginning of Section \ref{sect-setting} and let the control set $A$ be a domain in a Euclidian space. The EMFG involves the following data: for certain dimension $d$ and for $\sigma=(\sigma^0, \si^1)$,
\beaa
\left.\ba{c}
\dis b, \b: [0, T]\times \dbR^d\times A \times\cP_2(\dbR^d)\mapsto \dbR^{d_1},\q \sigma^i: [0, T]\times \dbR^d\times\cP_2(\dbR^d)\mapsto \dbR^{d\times d_i},~ i=0,1,\\
\dis f:[0, T]\times \dbR^d\times A \times\cP_2(\dbR^d)\mapsto \dbR,\q g: \dbR^d\times\cP_2(\dbR^d)\mapsto \dbR.
\ea\right.
\eeaa
Again, due to our notational convention, we often omit $t$ inside $b, \b, \si, f$.
Let $\cA$ denote the set of closed-loop controls $\a: [0,T]\times\dbR^d\times \O\mapsto A$ such that $\a$ is $\dbF^0$-progressively measurable, and $\cM$  the set of $\dbF^0$-progressively measurable $\cP_2(\dbR^d)$-valued process $\bm$ with $\int_0^T \|\bm_t\|_2^2dt<\infty$.

Given $\bm\in \cM$,  $\a\in \cA$,  and $\xi\in \dbL^2(\cF_0, \dbR^d)$,  $x\in \dbR^d$, recall the notational convention in the end of Section \ref{sect-setting} and consider the following dynamics, corresponding to the population and the representative  player, respectively:
\bea
\label{EMFG1}
\left.\ba{c}
\dis X^{\mathbf{m}; \xi, \a}_t =  \xi+\int_{0}^{t}\sigma^1 b\big(X^{\mathbf{m}; \xi, \a}_s,\alpha_s(X^{\mathbf{m}; \xi, \a}_s), \bm_s\big)ds +\int_0^t\sigma(X^{\mathbf{m}; \xi, \a}_s, \bm_s)dB_s,\\
\dis \cX^{\mathbf{m}; x, \a}_t =  x+\int_{0}^{t}\sigma^1\b\big( \cX^{\mathbf{m}; x, \a}_s,\alpha_s(\cX^{\mathbf{m}; x, \a}_s), \bm_s\big)ds +\int_0^t\sigma(\cX^{\mathbf{m}; x, \a}_s, \bm_s)dB_s.
\ea\right.
\eea
The representative  player aims to minimize the following cost:
\bea
\label{EMFG2}
\left.\ba{c}
\dis u^\bm_0(x) := \inf_{\a\in \cA} J( \mathbf{m};  x, \a),\q\mbox{where}\\
\dis J( \mathbf{m};  x, \a) := \dbE\Big[g(\cX^{\mathbf{m};  x, \a}_T, \bm_T) + \int_{0}^T f\big(\cX^{\mathbf{m}; x, \a}_s,\alpha_s(\cX^{\mathbf{m}; x, \a}_s), \bm_s\big)ds\Big].
\ea\right.
\eea

\begin{defn}
\label{defn-EMFGMFE}
Given $\cL_\xi = \mu\in \cP_2(\dbR^d)$, we say $(\a^*, \bm^*)\in \cA \times\cM$ is a mean field equilibrium (MFE) of EMFG \reff{EMFG1}-\reff{EMFG2} at $(0,\mu)$\footnote{ Here $0$ refers to time $t=0$, so as the $0$ in \reff{EMFG-V1}.} if

(i) Given $\bm^*$, $\a^*$ is optimal for the representative player:
\bea
\label{a*optimal}
J(\bm^*;  x, \a^*) =u^{\bm^*}_0(x),\q \mbox{for $\mu$-a.e. $x\in \dbR^d$}.
\eea

(ii) Given $\a^*$, $\bm^*$ is the conditional law of the population:
\bea
\label{m*optimal}
\cL_{X^{\bm^*; \xi, \a^*}_t|\cF^0_t} = \bm^*_t,~ 0\le t\le T,~\mbox{ a.s.}
\eea
\end{defn}
Moreover, when MFE $(\a^*, \bm^*)$ is unique, we define the value  of the EMFG at $(0,\mu)$:
\bea
\label{EMFG-V1}
V(0, x, \mu) := J(\bm^*; x, \a^*).
\eea

We next derive heuristically FBSDE systems in the form of  \reff{Xxi}-\reff{X2x} to characterize the MFE.  We emphasize that here we assume all the involved functions have sufficient regularity and we just argue formally. We shall establish the theory rigorously in the end of the section.

First, given $\bm$,  \reff{EMFG2} is a standard stochastic control problem. By the standard literature, $(u^{\bm}, v^{\bm})$ satisfies the following backward SPDE:
\bea
\label{EMFG-BSPDE}
&&u^{\bm}_t(x)= g(x,\bm_T) - \int_t^T v^{\bm}_s(x)dB_s^0+ \int_t^T \Big[H\big(x,\pa_xu^{\bm}_s(x)\sigma^1(x,\bm_s),\bm_s\big)\\
&& \qq\qq +\frac{1}{2}\pa_{xx} u^{\bm}_s(x) : \sigma\si^\top(x,\bm_s)+\sigma^0(x,\bm_s):\pa_x v^{\bm}(s,x)\Big]ds, \nonumber
\eea
where $H$ is the Hamiltonian:\footnote{\label{z1row}Note that $z^1=\pa_x u^\bm \si^1$ is supposed to be $1 \times d_1$-dimensional, namely it's a row vector. We shall abuse the notation and view it as a column vector, while inside the Hamiltonian we still use $\pa_x u^\bm \si^1$ which is more convenient when $u^\bm$ is multi-dimensional. More rigorously, we should consider $z^1 =  (\pa_x u^\bm \si^1)^\top$.}
\bea
\label{EMFG-H}
H(x,z^1,\mu):=\inf_{a\in A}h(x,z^1,\mu,a),~~ h(x,z^1,\mu,a):=\b (x,a,\mu)\cd  z^1+f(x,a,\mu),~ z^1\in \dbR^{d_1}.
\eea
Moreover, assume the Hamiltonian $H$ has a unique optimizer $a^* = \phi(x, z^1, \mu)$, namely:
\bea
\label{EMFG-a*}
H(x,z^1,\mu) = h\big(x,z^1,\mu, \phi(x, z^1, \mu)\big).
\eea
Then  \reff{EMFG2} has the optimal control $\a^\bm \in \cA$:
\bea
\label{EMFG-alpha*}
\a^\bm_t(x) = \phi\big(x, \pa_xu^{\bm}_s(x)\sigma^1(x,\bm_s), \bm_t\big).
\eea

Now assume $(\a^*, \bm^*)$ is an MFE at $(0, \mu)$. We must have $\a^*=\a^{\bm^*}$. Introduce SDEs:
\beaa
&&X^{\xi}_t =  \xi+\int_{0}^{t}\sigma^1(X^\xi_s,  \bm^*_s) b\big(X^\xi_s, \a^{\bm^*}_s(X^\xi_s), \bm^*_s\big)ds +\int_0^t\sigma(X^\xi_s, \bm^*_s)dB_s;\\
&&\cX^{x}_t =  x+\int_{0}^{t}\sigma^1(\cX^{x}_s,  \bm^*_s) \b \big(\cX^{x}_s, \a^{\bm^*}_s(\cX^{x}_s), \bm^*_s\big)ds +\int_0^t\sigma(\cX^{x}_s, \bm^*_s)dB_s.
\eeaa
Denote $Y^\xi_t := u^{\bm^*}_t(X^\xi_t)$, $\cY^{x}_t := u^{\bm^*}_t(\cX^{x}_t)$.
Then, by applying the It\^{o}-Ventzel formula we obtain:
\beaa
d Y^{\xi}_t &=&  -\big[f^\phi + (\b ^\phi - b^\phi) \cd Z^{\xi,1}_t\big]\big(X^{\xi}_t, Z^{\xi,1}_t, \bm^*_t\big)dt + Z^{\xi}_tdB_t,\\
d \cY^{x}_t &=&  -f^\phi\big(\cX^{x}_t, \cZ^{x,1}_t, \bm^*_t\big)dt + \cZ^{x}_tdB_t,
\eeaa
where
\bea
\label{EMFG-Phiphi}
&&  \Phi^\phi(x, z^1, \mu) := \Phi(x, \phi(x, z^1, \mu), \mu),\q \Phi = b, \b , f; \nonumber\\
 &&  Z^{\xi,1}_t = \pa_xu^{\bm^*}_t(X^\xi_t)\sigma^1(X^\xi_t,\bm^*_t), \q Z^{\xi,0}_t = \pa_xu^{\bm^*}_t(X^\xi_t)\sigma^0(X^\xi_t,\bm^*_t) + v^{\bm^*}_t(X^\xi_t);\\
&& \cZ^{x,1}_t =\pa_xu^{\bm^*}_t(\cX^{x}_t) \sigma^1(\cX^{x}_t,\bm^*_t), \q \cZ^{x,0}_t = \pa_xu^{\bm^*}_t(\cX^{x}_t) \sigma^0(\cX^{x}_t,\bm^*_t)+ v^{\bm^*}_t(\cX^{x}_t).\nonumber
\eea
By \reff{m*optimal}, we obtain the following FBSDE systems:
\bea
\label{EMFG-FBSDE1}
&&\!\!\!\!\!\!\!\!\!\!\!\! \dis\left\{\ba{lll}
\dis X^{\xi}_t =  \xi+\int_{0}^{t}\sigma^1(X^{\xi}_s,  \cL_{X^\xi_s|\cF^0_s}) b^\phi\big(X^{\xi}_s, Z^{\xi,1}_s, \cL_{X^\xi_s|\cF^0_s}\big)ds +\int_0^t\sigma(X^{\xi}_s,  \cL_{X^\xi_s|\cF^0_s})dB_s,\\
\dis Y^{\xi}_t =  g(X^\xi_T,\cL_{X^\xi_T|\cF^0_T}) +\int_t^T \big[f^\phi + (\b ^\phi - b^\phi) \cd Z^{\xi,1}_s\big]\big(X^{\xi}_s, Z^{\xi,1}_s, \cL_{X^\xi_s|\cF^0_s}\big)ds -\int_t^T Z^{\xi}_sdB_s;
\ea\right.\ms\\
\label{EMFG-FBSDE2}
&&\!\!\!\!\!\!\!\!\!\!\!\! \left\{\ba{lll}
\dis \cX^{x}_t =  x+\int_{0}^{t}\sigma^1(\cX^{x}_s,  \cL_{X^\xi_s|\cF^0_s}) \b^\phi\big(\cX^{x}_s, \cZ^{x,1}_s, \cL_{X^\xi_s|\cF^0_s}\big)ds +\int_0^t\sigma(\cX^{x}_s,  \cL_{X^\xi_s|\cF^0_s})dB_s,\\
\dis \cY^{x}_t =  g(\cX^{x}_T,\cL_{X^\xi_T|\cF^0_T}) +\int_t^T f^\phi\big(\cX^{x}_s, \cZ^{x,1}_s, \cL_{X^\xi_s|\cF^0_s}\big)ds -\int_t^T \cZ^{x}_sdB_s.
\ea\right.
\eea
This is the system \reff{Xxi}-\reff{X2x} with: for $\pi = (x, y, z)$,
\bea
\label{EMFG-coefficients}
&d_x^i= d, \q d_y^i = 1,\q \si_i(\pi, \mu) = \si(x, \mu), \q g_i(x, \mu) = g(x,\mu),\q i=1,2,\nonumber\\
&b_1(\pi, \mu) = \si^1(x, \mu) b^\phi(x, z^1, \mu),\q b_2(\pi, \mu) = \si^1(x, \mu) \b ^\phi(x, z^1, \mu);\\
&f_1(\pi, \mu)= f^\phi(x, z^1, \mu) + z^1\cd  [\b^\phi-b^\phi](x, z^1, \mu),\q f_2(\pi,\mu)  = f^\phi(x, z^1, \mu).\nonumber
\eea

\begin{rem}
\label{rem-EMFG}
(i) One can easily see that
\bea
\label{EMFG-paH}
\pa_{z^1} H = \b ^\phi,\q H = f^\phi + z^1\cd \b ^\phi.
\eea
In particular, we may rewrite the generators of the backward equations in a more symmetric way:
\beaa
f_1(\pi, \mu) = H(x, z^1,\mu) - z^1\cd \b ^\phi(x, z^1,\mu),\q f_2(\pi, \mu) = H(x, z^1,\mu) - z^1\cd b^\phi(x, z^1,\mu).
\eeaa

(ii) When $\b =b$, EMFG \reff{EMFG1}-\reff{EMFG2} reduces to a standard MFG. Moreover, in the standard literature, typically $d=d_1$ and $\si^1$ is nondegenerate, then one may write the drift as $b$ directly, instead of $\si^1 b$.

(iii) We assume the population and the representative player share the same volatility $\si$. In the general case that the representative player has a different volatility, we are not able to characterize the MFE via FBSDEs. We shall leave this interesting case to future research.
\end{rem}

Finally we establish rigorously the well-posedness of the above FBSDEs as well as the corresponding master equation:
\bea
\label{EMFG-master}
\left.\ba{lll}
\dis \pa_t V + \frac{1}{2} \pa_{xx} V : \sigma\sigma^\top +  H(x,\partial_x V \sigma^1,\mu) + \dbM V =0, \q V(T,x,\mu) = g(x,\mu),\q \mbox{where}\\
\dis  \dbM V(t,x,\mu)  :=\int_{\dbR^d}\int_{\dbR^d} \Big[\frac{1}{2} \pa_{\tilde x\mu}  V(t,x, \mu, \tilde x):\sigma\si^\top(\tilde x,\mu) \\
\dis\qq +\pa_{x\mu}V(t,x,\mu,\tilde x) : \sigma^0(x,\mu)\sigma^0(\tilde x,\mu)^\top+\frac{1}{2}\pa_{\mu\mu}V(t,x,\mu,\tilde x,\bar x): \sigma^0(\tilde x,\mu)\sigma^0(\bar x,\mu)^\top \\
\dis\qq + \pa_\mu V(t, x, \mu, \tilde x) \cd \sigma^1(\tilde x, \mu) b^\phi \big(\tilde x,~\sigma^1(\tilde x,\mu)^\top\pa_x V(t, \tilde x, \mu),~\mu\big)\Big]\mu(d\tilde x)\mu(d\bar x).
\ea\right.
\eea
 We first specify the technical conditions.

\begin{assum}
\label{assum-EMFG} Fix some integer $k\geq 2$.

\ms
\no(i) The Hamiltonian $H$ admits a unique maximizer $\phi$ in the sense of \reff{EMFG-a*}.

\ms

\no(ii) For $\Phi = \si, g:$ $\Phi\in \cC^{k}_b$ and the first order derivatives $|\pa_x\Phi|, |\pa_\mu\Phi|\leq L_\Phi$.

\ms
\no(iii) For $\Phi = b^\phi,  \b ^\phi,  f^\phi:$  $\Phi\in \cC^{k}$, and for any $R>0$, there exists $L(R)>0$ such that, whenever $|z^1|\le R$,   $|\pa_x\Phi|, |\pa_{z^1}\Phi|, |\pa_\mu\Phi|\leq L(R)$ and  all the involved higher order derivatives are bounded.

\ms
\no (iv) $|\sigma^1(x,\mu)|, |\b ^\phi(x, 0, \mu)|,  |b^\phi(x, 0, \mu)| \le L_0$.
\end{assum}
We remark that the local boundedness in (iii) with respect to $z^1$, especially that for $f^\phi$, allows us to deal with linear quadratic models.

\begin{rem}
\label{rem-EMFG-nondegenerate}
 If we assume $\si^1$ is uniformly non-degenerate, then we can rewrite the drifts  in \reff{EMFG1} with $b' = \si^1b$ and $\b'=\si^1\b$ directly. In this case, by imposing the same conditions on $b', \b'$, we do not need to assume $\si^1$ to be bounded, as we will see in the next section.
\end{rem}

\begin{thm}\label{thm-EMFG}
Let Assumption \ref{assum-EMFG} hold and denote $R_0:= L_g + 1$, $R_1 := R_0 L_0$, $R_2:= R_1+2$. Then there exists a constant $\e_0>0$, depending only on $d$, $d_{01}$, $L_g$, $L_{\sigma}$, and $L(R_2)$, such that whenever $0<T\leq\e_0$, the following hold:

\no(i) The FBSDEs \reff{EMFG-FBSDE1}-\reff{EMFG-FBSDE2} are well-posed and the EMFG \reff{EMFG1}-\reff{EMFG2}  has a unique MFE $(\a^*, \bm^*)\in \cA\times \cM$ at $(0,\mu)$.

\no(ii) The master equation \eqref{EMFG-master}  has a unique classical solution $V\in\cC^{k}_b$, and it holds that:
\bea
\label{EMFG-FK}
\left.\ba{c}
\bm^*_t = \cL_{X^\xi_t|\cF^0_t},\q \a^*_t(x) = \phi\big(t, x, \pa_x V\si^1(t,x,\bm^*_t), \bm^*_t\big);\ms\\
Y^\xi_t = V(t, X^\xi_t, \bm^*_t),\q Z^{\xi}_t = \cI^V_t(X^\xi_t),\q \cY^{x}_t = V(t, \cX^{x}_t, \bm^*_t),\q  \cZ^{x}_t =  \cI^V_t(\cX^x_t), \ms\\
\dis\mbox{where}\q \cI^{V,1}_t(x)  :=  \pa_x V\si^1(t, x, \bm^*_t),\ms \\
\dis \cI^{V,0}_t(x) :=  \pa_x V\si^0(t, x, \bm^*_t) +\int_{\dbR^d} \pa_\mu V(t, x, \bm^*_t, \tilde x) \si^0(t,\tilde x, \bm^*_t)\bm^*_t(d\tilde x).
\ea\right.
\eea
\end{thm}
\proof We first note that, under Assumption \ref{assum-EMFG}, the coefficients $b_i, f_i$ in \reff{EMFG-coefficients} do not satisfy the requirements in Assumption \ref{assum-1}.\footnote{In fact, even if we strengthen  Assumption \ref{assum-EMFG} to $|\pa_x\Phi|, |\pa_{z^1}\Phi|, |\pa_\mu\Phi| \le L_\Phi$ for $\Phi = b^\phi, b^\phi_0, f^\phi$, the $f_1$ in  \reff{EMFG-coefficients} is still not uniformly Lipschitz continuous in $(x, z^1, \mu)$, due to the term $z^1 \cd (b^\phi_0-b^\phi)$.} In order to apply Theorem \ref{thm-Ck}, we introduce a truncation function for $z^1$. For any $R>2$, let $I_R: \dbR\to \dbR$ be in $\cC^k_b$ satisfying, for $x\in \dbR$,
\bea
\label{truncation}
I_{R}(x)=x,~ |x|\leq R-2;~ |\pa_{x}I_{R}(x)|=0,~ |x|\geq R;~ |\pa_{x}I_{R}(x)|\leq 1,~R-2<|x|< R.
\eea
For $\Phi=b_i, f_i$ in  \reff{EMFG-coefficients}, set $\Phi^R(t,\pi, \mu) := \Phi(t, x, y, z^0, I_R(z^1), \mu)$,  where $I_R(z^1)$ is truncated component wise. Note that $|b_i^R|\le L_0 + RL(R)$, then it is clear that $(b^{R}_i, \si_i, f^{R}_i, g_i)$, $i=1,2$, satisfy Assumption \ref{assum-1} with $L_1$ depending on $L_g$, $L_\si$, $L_0$, $R$, and $L(R)$. Now set $R= R_2$ and consider FBSDEs \reff{Xxi}-\reff{X2x} with coefficients $(b^{R_2}_i, \si_i, f^{R_2}_i, g_i)$. By Theorem \ref{thm-Ck}, there exists $\e_0>0$, which depends on $d$, $d_{01}$, $L_g$, $L_\si$, $L_0$, $R_2$, and $L(R_2)$, such that these FBSDEs are well-posed with solutions $\Pi^\xi, \Pi^{2,x}$, and the corresponding functions $U_1, U_2 \in \cC^k_b$ satisfy \reff{patU}, with the $b_i, f_i$ in \reff{LU1}-\reff{LU2} replaced with $b^{R_2}_i$, $f^{R_2}_i$.

Next, by Theorem \ref{thm-pax}, we see that $|\pa_x U_1|\le R_0$, $|\pa_x U_2|\le R_0$. Then, in \reff{LU2}, $|\cV_i^{U_i,1}|\le R_0 L_0 = R_1=R_2-2$, and thus, by \reff{truncation},  $(\Phi^{R_2})_i^{U_i} = \Phi_i^{U_i}$ for $\Phi = b, f$. This implies that $U_1, U_2$ satisfy \reff{patU}, with the original coefficients $(b_i, \si_i, f_i, g_i)$ specified in \reff{EMFG-coefficients}. Moreover, by \reff{LU1}-\reff{LU2} and \reff{EMFG-coefficients}, one can verify straightforwardly that, for $i=1,2$,
\beaa
\dbL^i U = \frac{1}{2} \pa_{xx} U : \sigma\sigma^\top +  H(x,\partial_x U \sigma^1,\mu) + \dbM U,
\eeaa
where $\dbM U$ is defined by \reff{EMFG-master}. Since $U_1(T,x,\mu) = U_2(T,x,\mu)=g(x, \mu)$, then $U_1 = U_2$ and they satisfy the master equation \reff{EMFG-master}.

Moreover, denote $\bm^*_t := \cL_{X^\xi_t|\cF^0_t}$. By \reff{randomFK} we see that
 \beaa
Z^{\xi,1}_t =   \pa_x U_1 \si^1_t(t, X^\xi_t, \bm^*_t),\q Z^{2,x,1}_t =  \pa_x U_2 \si^1_t(t, X^{2,x}_t, \bm^*_t).
\eeaa
Then $|Z^{\xi,1}_t|, |Z^{2,x,1}_t|\le R_0 L_0 = R_1 = R_2-2$. By \reff{truncation}, this implies that
\beaa
\Phi^R(t, \Pi^\xi, \bm^*_t) =  \Phi(t, \Pi^\xi, \bm^*_t),\q \Phi^R(t, \Pi^{2,x}, \bm^*_t) =  \Phi(t, \Pi^{2,x}, \bm^*_t),\q \Phi = b_i, f_i.
\eeaa
That is, $\Pi^\xi, \Pi^{2,x}$ satisfy  FBSDEs \reff{Xxi}-\reff{X2x} with the original coefficients $(b_i, \si_i, f_i, g_i)$ as specified in \reff{EMFG-coefficients}. Consequently, they satisfy FBSDEs \reff{EMFG-FBSDE1}-\reff{EMFG-FBSDE2}.

Now given $\bm^*$, consider the FBSDE \reff{EMFG-FBSDE2} and the optimization problem \reff{EMFG2}. By standard stochastic control theory, we see that $u^{\bm^*}_0(x) = \cY^{x}_0$, with optimal control $\a^{\bm^*}_t(x) := \phi\big(t, x, \pa_x U_2\si^1(t,x,\bm^*_t), \bm^*_t\big)$. Since $U_1 = U_2$,  $\bm^*$ satisfies \reff{m*optimal}, then $(\a^*, \bm^*)$ is an MFE of the EMFG, in particular, $V(0,x,\mu) = J(\bm^*, \a^*) = \cY^{x}_0 = U_2(0,x,\mu)$. By extending $V$ to $(t,x,\mu)$ naturally, we have $V = U_2$ and satisfies the master equation \reff{EMFG-master}. Since $V=U_2=U_1$, \reff{EMFG-FK} is obvious, in particular, the expression of $Z$ follows from  It\^{o}'s formula \reff{Ito}.

Finally, the uniqueness of classical solutions in $\cC^2_b$ to master equation \reff{EMFG-master} follows from standard arguments in the mean field game literature, and given the classical solution $V=U_1=U_2$, which serves as the common decoupling field of the FBSDEs \reff{EMFG-FBSDE1} and \reff{EMFG-FBSDE2}, the uniqueness of these FBSDEs as well as the uniqueness of MFE are rather standard.
\qed

\begin{rem}
\label{rem-Ito}
The It\^{o} formula \reff{Ito} and the It\^{o}-Ventzel formula are consistent in the following sense. Given $\bm^*$,  by the It\^{o} formula \reff{Ito}, one can easily see that the solution $V$ to the master equation \reff{EMFG-master} induces naturally the solution to the BSPDE \reff{EMFG-BSPDE}:
$$
u^{\bm^*}_t(x) := V(t,x, \bm^*_t),\q v^{\bm^*}_t(x) := \int_{\dbR^d} \pa_\mu V(t, x, \bm^*_t, \tilde x)\si^0(\tilde x, \bm^*_t)\bm^*_t(d\tilde x).
$$
Then, the representation \reff{EMFG-Phiphi} for $Z$ obtained through the It\^{o}-Ventzel formula coincides with the representation \reff{EMFG-FK}  obtained through the It\^{o} formula.
\end{rem}

\section{Mean field games with volatility control}
\label{sect-MFGVC}
\setcounter{equation}{0}

In this section we study mean field games with volatility control (MFGVC).
Again consider the probabilistic setting  in the beginning of Section \ref{sect-setting}, but with $d_0=d_1=d$ for simplicity, and let $A, \cA$ be as in the previous section.  The MFGVC involves the following data with $\si = (\si^0, \si^1)$:
\beaa
\left.\ba{c}
\dis (b, \si^1, f): [0, T]\times \dbR^d\times A \times\cP_2(\dbR^d)\mapsto \dbR^{d}\times  \dbR^{d\times d} \times \dbR,\\\sigma^0: [0, T]\times \dbR^d\times \cP_2(\dbR^d)\mapsto  \dbR^{d\times d}, \q g: \dbR^d\times\cP_2(\dbR^d)\mapsto \dbR.
\ea\right.
\eeaa
For the purpose of the FBSDE characterization,  we assume $\si^1$ is positive definite and  $\si^0$ does not depend on the control.  Given  $\bm$,  $\a$,  $\xi$,  $x$ as in the previous section, we revise \reff{EMFG1} as:
\bea
\label{MFGVC1}
\left.\ba{c}
\dis X^{\mathbf{m}; \xi, \a}_t =  \xi+\int_{0}^{t} b\big(X^{\mathbf{m}; \xi, \a}_s,\alpha_s(X^{\mathbf{m}; \xi, \a}_s), \bm_s\big)ds +\int_0^t\sigma(X^{\mathbf{m}; \xi, \a}_s, \alpha_s(X^{\mathbf{m}; \xi, \a}_s), \bm_s)dB_s,\\
\dis X^{\mathbf{m}; x, \a}_t =  x+\int_{0}^{t} b\big( X^{\mathbf{m}; x, \a}_s,\alpha_s(X^{\mathbf{m}; x, \a}_s), \bm_s\big)ds +\int_0^t\sigma(X^{\mathbf{m}; x, \a}_s, \alpha_s(X^{\mathbf{m}; x, \a}_s), \bm_s)dB_s,
\ea\right.
\eea
and \reff{EMFG2} is exactly the same. Here, since $X^{\bm;\xi, \a} = \cX^{\bm; x, \a}$ when $\xi \equiv x$,  in the second equation above we use the notation $X$ instead of $\cX$, and again by our national convention we omit the variable $t$ in $b$, $\si$, $f$. We note that, besides the volatility control, we replace the $\si^1 b$ in \reff{EMFG1} with $b$ here. We then define  MFE $(\a^*, \bm^*)$ of MFGVC \reff{MFGVC1}-\reff{EMFG2} exactly by Definition \ref{defn-EMFGMFE}, namely by \reff{a*optimal} and \reff{m*optimal}, and when MFE $(\a^*, \bm^*)$ is unique, we again define the value  of the MFGVC at $(0,\mu)$ by \reff{EMFG-V1}:  $V(0, x, \mu) := J(\bm^*; x, \a^*)$.

 Given $\bm$,  \reff{EMFG2} is a standard stochastic control problem, but with volatility control. In this case, the backward SPDE \reff{EMFG-BSPDE} becomes fully nonlinear:
\bea
\label{MFGVC-BSPDE}
&&u^{\bm}_t(x)= g(x,\bm_T) - \int_t^T v^{\bm}_s(x)dB_s^0+ \int_t^T \Big[H\big(x, \pa_xu^{\bm}_s(x), \pa_{xx} u^\bm_s(x), \bm_s\big)\nonumber\\
&& \qq\qq +\frac{1}{2}\pa_{xx} u^{\bm}_s(x) : \sigma^0(\si^0)^\top(x,\bm_s)+\sigma^0(x, m_s):\pa_x v^{\bm}_s(x)\Big]ds,
\eea
where  the Hamiltonian $H$ becomes, for $z^1\in \dbR^d$ and $\g\in \dbR^{d\times d}$:
\bea
\label{MFGVC-H}
\left.\ba{c}
\dis H(x,z^1, \g, \mu):=\inf_{a\in A}h(x,z^1, \g, \mu,a),\\
\dis h(x,z^1, \g, \mu,a):={1\over 2}\g: \si^1 (\si^1)^\top(x,a,\mu) + b(x,a,\mu)\cd  z^1+f(x,a,\mu).
\ea\right.
\eea
Again, we assume the Hamiltonian $H$ has an optimizer $a^* = \phi(x, z^1, \g, \mu)$.
Then  \reff{EMFG2} has an optimal control $\a^\bm \in \cA$:
\bea
\label{MFGVC-alpha*}
\a^\bm_t(x) = \phi\Big(t, x, \pa_xu^{\bm}_t(x), \pa_{xx}u^{\bm}_t(x), \bm_t\Big).
\eea
Denote, for a function $\Phi$ on $[0,T]\times \dbR^d \times A\times \cP_2(\dbR^d)$ (omitting $t$ again),
\bea
\label{MFGVC-Phiphi}
\Phi^\phi(x, z^1, \g, \mu) := \Phi(x, \phi(x, z^1, \g, \mu), \mu).
\eea
Then, similarly to \reff{EMFG-paH} we have
\bea
\label{MFGVC-paH}
 b^\phi = \pa_{z^1} H,\qq   {1\over 2} \si^{\phi,1} (\si^{\phi,1})^\top=\pa_\g H,\q f^\phi = H - \pa_\g H : \g - \pa_{z^1} H \cd z^1.
\eea
In particular, since $\si^1$ is positive definite and $\pa_\g H$ is symmetric, we have $\si^{\phi,1}=\big(2\pa_\g H\big)^{1\over 2}$.

We next derive heuristically the FBSDEs, which is more involved in this case. Assume $(\a^*, \bm^*)$ is an MFE at $(0, \mu)$. By the optimality we have $\a^*=\a^{\bm^*}$. Introduce SDE:
\bea
\label{MFGVC-Xxi}
X^{\xi}_t =  \xi+\int_{0}^{t} b\big(X^\xi_s, \a^{\bm^*}_s(X^\xi_s), \bm^*_s\big)ds +\int_0^t\sigma(X^\xi_s, \a^{\bm^*}_s( X^\xi_s), \bm^*_s)dB_s.
\eea
To derive the FBSDEs, it is natural to introduce (recalling Footnote \ref{z1row}):
\beaa
&Y^{0,\xi}_t := u^{\bm^*}_t(X^\xi_t),\q Z^{0,\xi,1}_t := \pa_xu^{\bm^*}_t\sigma^{*1}_t(X^\xi_t),\q Z^{0,\xi,0}_t := \pa_xu^{\bm^*}_t\sigma^{*0}_t(X^\xi_t) + v^{\bm^*}_t(X^\xi_t),\\
&\mbox{where}\q \si^{*1}_t(x) := \si^1(t,x, \a^{\bm^*}_t(x), \bm^*_t),\q \si^{*0}_t(x) := \si^0(t,x, \bm^*_t).
\eeaa
However, by \reff{MFGVC-alpha*} the optimal control $\a^*$ involves $\pa_{xx} u^{\bm^*}$, which cannot be recovered from $Z^{0,\xi}_t$. For this purpose, we introduce, for $i=1,\cds, d$,
\beaa
Y^{1,\xi}_{i,t} := \pa_{x_i} u^{\bm^*}_t(X^\xi_t),~ Z^{1,\xi,1}_{i,t} := \pa_{x_i x} u^{\bm^*}_t\sigma^{*1}_t(X^\xi_t),~ Z^{1,\xi,0}_{i,t} := \pa_{x_ix}u^{\bm^*}_t\sigma^{*0}_t(X^\xi_t) + \pa_{x_i} v^{\bm^*}_t(X^\xi_t),
\eeaa
and $\Phi^{1,\xi}_t := (\Phi^{1,\xi}_{1,t}, \cds,  \Phi^{1,\xi}_{d, t})$ for $\Phi = Y, Z$.
Then we may express $(\pa_x u^{\bm^*}, \pa_{xx} u^{\bm^*})$, and hence $\a^{\bm^*}$, through $(Z^{0,\xi,1}, Z^{1,\xi,1})$. This will lead to FBSDE systems in the form of \reff{Xxi}-\reff{X2x}. However, in this case the diffusion coefficient $\sigma^1$ of these FBSDEs will depend on $(Z^{0, \xi,1}, Z^{1,\xi,1})$, which makes the well-posedness of these FBSDEs hard, even when $T$ is small.  To get around of this difficulty, we introduce further that $\Phi^{2,\xi}_t = (\Phi^{2,\xi}_{i,j,t})_{1\le i, j\le d}$ for $\Phi = Y, Z$, where, for $i, j=1,\cds, d$,
\beaa
&Y^{2,\xi}_{i,j,t} := \pa_{x_ix_j}u^{\bm^*}_t(X^\xi_t),\q Z^{2,\xi,1}_{i,j,t} := \pa_{x_ix_jx}u^{\bm^*}_t\sigma^{*1}_t(X^\xi_t),\\
&Z^{2,\xi,0}_{i,j,t} := \pa_{x_ix_jx}u^{\bm^*}_t\sigma^{*0}_t(X^\xi_t) + \pa_{x_ix_j} v^{\bm^*}_t(X^\xi_t).
\eeaa
Then we may express $\a^{\bm^*}$ through $(Y^{1,\xi}, Y^{2,\xi})$, instead of $(Z^{0,\xi,1}, Z^{1,\xi,1})$.

Recall \reff{MFGVC-BSPDE} and differentiate it in $x$, we obtain the equations for $\pa_x u^\bm$ and $\pa_{xx} u^\bm$:
\beaa
&&\!\!\!\!\!\!\!\!\!\!\!\!\!\!  \pa_{x_i} u^{\bm}_t(x)= \pa_{x_i} g - \int_t^T \pa_{x_i} v^{\bm}dB_s^0 + \int_t^T \Big[ \cH_i^1(x, \pa_x u^\bm, \pa_{xx} u^\bm,\bm_s, \pa_x v^\bm) \\
&& + \pa_{z^1} H \cd \pa_{x_ix} u^\bm + \pa_\g H : \pa_{x_ixx} u^\bm +\frac{1}{2}\pa_{x_ixx} u^{\bm} : \sigma^0(\si^0)^\top+\sigma^0:\pa_{x_ix} v^{\bm} \Big]ds, \\
&&\!\!\!\!\!\!\!\!\!\!\!\!\!\!   \pa_{x_ix_j} u^{\bm}_t(x)= \pa_{x_ix_j} g- \int_t^T\!\! \pa_{x_ix_j} v^{\bm}dB_s^0 + \int_t^T \!\!\Big[\cH^2_{ij}(x, \pa_x u^\bm, \pa_{xx} u^\bm,  \bm_s, \pa_{xxx} u^\bm, \pa_x v^\bm, \pa_{xx} v^\bm)   \\
&&  + \pa_{z^1} H \cd \pa_{x_ix_jx} u^\bm + \pa_\g H : \pa_{x_ix_jxx} u^\bm +\frac{1}{2}\pa_{x_ix_jxx} u^{\bm} : \sigma^0(\si^0)^\top+\sigma^0:\pa_{x_ix_jx} v^{\bm} \Big]ds,
\eeaa
for $i, j=1,\cds, d$, where, with  $\cH^0$ corresponding to  the $f^\phi$ in \reff{MFGVC-paH},
\bea
\label{MFGVC-cH}
&&\cH^0(x, \pa_x u^\bm, \pa_{xx} u^\bm,\bm_s)  := H - \pa_\g H : \pa_{xx} u^\bm - \pa_{z^1} H \cd \pa_x u^\bm;\nonumber\\
&&\cH^1_i(x, \pa_x u^\bm, \pa_{xx} u^\bm,\bm_s, \pa_x v^\bm)  := \pa_{x_i} H + \pa_{xx} u^{\bm} : \pa_{x_i}\sigma^0( \si^0)^\top + \pa_{x_i} \si^0 : \pa_x v^\bm;\nonumber\\
&& \cH^2_{ij}(x, \pa_x u^\bm, \pa_{xx} u^\bm, \bm_s, \pa_{xxx} u^\bm, \pa_x v^\bm, \pa_{xx} v^\bm)  := \pa_{x_ix_j} H + \pa_{x_iz^1} H \cd \pa_{x_jx} u^\bm  \\
&&\qq + \pa_{x_jz^1} H \cd \pa_{x_ix} u^\bm + \pa_{x_i\g}  H : \pa_{x_jxx} u^\bm + \pa_{x_j\g} H : \pa_{x_ixx} u^\bm+\pa_{x_ixx} u^{\bm} : \pa_{x_j}\sigma^0( \si^0)^\top \nonumber\\
&&\qq +\pa_{x_jxx} u^{\bm} : \pa_{x_i}\sigma^0( \si^0)^\top + \pa_{xx} u^{\bm} : \pa_{x_i} \sigma^0(\pa_{x_j} \si^0)^\top + \pa_{xx} u^{\bm} : \pa_{x_ix_j} \sigma^0(\si^0)^\top\nonumber\\
&&\qq + \pa_{x_ix_j} \si^0 : \pa_{x} v^\bm+\pa_{x_i} \sigma^0:\pa_{x_jx} v^{\bm} +\pa_{x_j} \sigma^0:\pa_{x_ix} v^{\bm} + \pa_{z^1z^1} H : \pa_{x_ix}u^\bm (\pa_{x_jx} u^\bm)^\top \nonumber \\
&&\qq +\sum_{k=1}^d \pa_{z^{1,k}\g} H : \big[\pa_{x_ix_k} u^\bm \pa_{x_j x x}u^\bm  + \pa_{x_jx_k} u^\bm \pa_{x_i x x}u^\bm]  + \sum_{k,l=1}^d\pa_{\g^{kj}\g} H : (\pa_{jkl} u^\bm  \pa_{ixx} u^\bm).\nonumber
\eea
Recall that $\bm^* = \cL_{X^\xi|\dbF^0}$ and denote
\beaa
&Y^\xi := (Y^{0,\xi},  Y^{1,\xi}, Y^{2,\xi})\in \dbR\times \dbR^d\times \dbR^{d\times d}, \\
& Z^\xi := (Z^{0,\xi}, Z^{1,\xi}, Z^{2,\xi}) \in \dbR^{2d} \times \dbR^{d\times 2d} \times \dbR^{(d\times d)\times 2d},\\
&  \Xi^{\xi} := (X^\xi, Y^\xi, Z^\xi, \bm^*),\q \Xi^{'\xi} := (X^\xi, Y^{1,\xi}, Y^{2,\xi}, \bm^*),\\
& \cH^1 := (\cH^1_i)_{1\le i\le d},\q \cH^2 := (\cH^2_{ij})_{1\le i, j\le d}.
\eeaa
Recall \reff{MFGVC-paH} and note that
\bea
\label{paxv}
\left.\ba{c}
\dis  \pa_x u^{\bm^*}_t(X^\xi_t) = Y^{1,\xi}_t,\q \pa_{xx} u^{\bm^*}_t(X^\xi_t) = Y^{2,\xi}_t,\ms \\
\dis b\big(X^\xi_t, \a^{\bm^*}_t(X^\xi_t), \bm^*_t\big) = \pa_{z^1} H(\Xi^{'\xi}_t),\q \si^{*1}_t = \big(2\pa_\g H(\Xi^{'\xi}_t)\big)^{1\over 2},\ms\\
\dis \pa_{xxx} u^\bm(t, X^\xi_t)  = \cI_0(\Xi^\xi_t), \q \pa_x v^{\bm^*}(t, X^\xi_t) = \cI_1(\Xi^\xi_t),\q \pa_{xx} v^{\bm^*}(t, X^\xi_t) = \cI_2(\Xi^\xi_t),\ms\\
\dis \cI_0(\zeta):=  z^{2,1} \big(2\pa_\g H(\zeta')\big)^{-{1\over 2}},~  \cI_1(\zeta) := z^{1,0} - y^2\si^0(x, \mu),~\cI_2(\zeta):= z^{2,0} -  \cI_0(\zeta)\si^0(x, \mu),
\ea\right.
\eea
where  $y = (y^0, y^1, y^2)$, $z = (z^0, z^1, z^2)$, $z^i = (z^{i,0}, z^{i,1})$ for $i=0,1,2$, and $\zeta = (x,y,z,\mu)$, $\zeta'=(x,y^1, y^2, \mu)$. Applying the It\^{o}-Ventzel formula we obtain the following FBSDE:
\bea
\label{MFGVC-FBSDE1}
\left\{\ba{lll}
\dis X^{\xi}_t =  \xi+\int_{0}^{t}\pa_{z^1} H(\Xi^{'\xi}_s)ds +\int_0^t\big(2\pa_{\g} H(\Xi^{'\xi}_s)\big)^{1\over 2}dB^1_s + \int_0^t \si^0(X^\xi_s, \cL_{X^\xi_s|\cF^0_s})dB^0_s,\\
\dis Y^{0,\xi}_t =   g(X^\xi_T,\cL_{X^\xi_T|\cF^0_T}) +\int_t^T\cH^0(\Xi^{'\xi}_s) ds -\int_t^T Z^{0,\xi}_sdB_s,\\
\dis Y^{1,\xi}_t =  \pa_x g(X^\xi_T,\cL_{X^\xi_T|\cF^0_T}) +\int_t^T\cH^1\big(\Xi^{'\xi}_s, \cI_1(\Xi^\xi_s)\big) ds -\int_t^T Z^{1,\xi}_sdB_s,\\
\dis Y^{2,\xi}_t =  \pa_{xx} g(X^\xi_T,\cL_{X^\xi_T|\cF^0_T}) +\int_t^T\cH^2\Big(\Xi^{'\xi}_s, ~\cI_0(\Xi^\xi_s), \cI_1(\Xi^\xi_s), \cI_2(\Xi^\xi_s)\Big)ds -\int_t^T Z^{2,\xi}_sdB_s .
\ea\right.
\eea
We remark that, the $(z^1, \g)$ in \reff{MFGVC-H} correspond to $(y^1, y^2)$ here, and thus $\pa_{z^1} H, \pa_\g H$ refer to the derivatives with respect to $y^1$ and $y^2$, respectively. Moreover, the second equation above for $(Y^{0,\xi}, Z^{0,\xi})$ is decoupled from the other three equations, so in \reff{MFGVC-FBSDE1} one may first solve the three coupled equations: the first, third, and fourth equations, and then solve the second equation.

Next, let $X^x := X^{\bm^*; x, \a^{\bm^*}}$ as in the second equation in \reff{MFGVC1}, equivalently it is the solution to \reff{MFGVC-Xxi} with initial value $X_0=x$, but still with $\bm^* = \cL_{X^\xi|\cF^0}$. Denote
\beaa
Y^{0,x}_t := u^{\bm^*}_t(X^{x}_t),~ Y^{1,x}_t :=\pa_x u^{\bm^*}_t(X^{x}_t),~ Y^{2,x}_t := \pa_{xx}u^{\bm^*}_t(X^{0,x}_t), ~ Y^x_t := (Y^{0,x}, Y^{1,x}, Y^{2,x}_t),
\eeaa
and introduce $Z^{x}$ similarly.
Then,  for $\Xi^{x}=(X^{x}, Y^{x}, Z^{x}, \bm^*)$ and $\Xi^{'x}=(X^{x}, Y^{1,x}, Y^{2,x},  \bm^*)$,
\bea
\label{MFGVC-FBSDE2}
\left\{\ba{lll}
\dis X^{x}_t =  x+\int_{0}^{t}\pa_{z^1} H(\Xi^{'x}_s)ds +\int_0^t\big(2\pa_{\g} H(\Xi^{'x}_s)\big)^{1\over 2}dB^1_s + \int_0^t \si^0(X^x_s, \bm^*_s)dB^0_s,\\
\dis Y^{0,x}_t =   g(X^{x}_T,\bm^*_T) +\int_t^T\cH^0(\Xi^{'x}_s) ds -\int_t^T Z^{0,x}_sdB_s,\\
\dis Y^{1,x}_t =  \pa_x g(X^{x}_T,\bm^*_T) +\int_t^T\cH^1\big(\Xi^{'x}_s, \cI_1(\Xi^x_s)\big) ds -\int_t^T Z^{1,x}_sdB_s,\\
\dis Y^{2,x}_t =  \pa_{xx} g(X^{x}_T,\bm^*)  +\int_t^T\cH^2\Big(\Xi^{'x}_s, ~\cI_0(\Xi^x_s), \cI_1(\Xi^x_s), \cI_2(\Xi^x_s)\Big) ds -\int_t^T Z^{2,x}_sdB_s.
\ea\right.
\eea
Clearly \reff{MFGVC-FBSDE1}-\reff{MFGVC-FBSDE2} is the system \reff{Xxi}-\reff{X2x} with $\Phi_1=\Phi_2$ for $\Phi=b, \si, f, g$. To be precise, for $i=1,2$, and denoting $\zeta'=(x,y^1, y^2, \mu)$ for $\zeta=(x, y, z, \mu)$, we have
\bea
\label{MFGVC-coefficients}
&d_x^i= d, \q d_y^i = 1+ d+d^2,\q b_i(\zeta) = \pa_{z^1} H(\zeta'),\q \si^0_i(\zeta) = \si^0(x, \mu), \q \si_i^1(\zeta) = \big(2\pa_{\g} H(\zeta')\big)^{1\over 2}\nonumber\\
&g_i^0(\zeta) = g(x, \mu),\q g_i^1(\zeta) = \pa_x g(x, \mu),\q g_i^2(\zeta) = \pa_{xx} g(x,\mu),
\\
&f_i^0(\zeta)= \cH^0(\zeta'), \q  f_i^1(\zeta)= \cH^1\big(\zeta', ~\cI_1(\zeta)\big), \q f_i^2(\zeta)= \cH^2\big(\zeta', ~ \cI_0(\zeta), \cI_1(\zeta), \cI_2(\zeta)\big). \nonumber
\eea

We now establish rigorously the well-posedness of the FBSDEs (\ref{MFGVC-FBSDE1})-(\ref{MFGVC-FBSDE2}) as well as the corresponding master equation:
\bea
\label{MFGVC-master}
\left.\ba{lll}
\dis \pa_t V  +  H(x,\partial_x V, \pa_{xx} V, \mu) + \frac{1}{2} \pa_{xx} V : \sigma^0(\sigma^0)^\top + \dbM V =0, \q V(T,x,\mu) = g(x,\mu),\q \mbox{where}\\
\dis \dbM V(t,x,\mu)  :=\int_{\dbR^d}\int_{\dbR^d}\Big[ \pa_\mu V(t, x, \mu, \tilde x)\cd \pa_{z^1} H\big(\tilde x,\pa_x V(t, \tilde x, \mu),\pa_{xx} V(t, \tilde x, \mu),\mu\big) \\
\dis \qq\qq +\pa_{\tilde x\mu}  V(t,x, \mu, \tilde x) : \pa_{\gamma}H\big(\tilde x,\pa_xV(t,\tilde x,\mu),\pa_{xx}V(t,\tilde x,\mu),\mu\big)\\
\dis \qq\qq+\frac{1}{2}\pa_{\tilde x\mu}  V(t,x, \mu, \tilde x) : \sigma^0(\sigma^0)^\top(\tilde x, \mu) +\pa_{x\mu}V(t,x,\mu,\tilde x):  \sigma^{0}(\sigma^{0})^\top(\tilde x, \mu)\\
\dis\qq\qq +\frac{1}{2}\pa_{\mu\mu}V(t,x,\mu,\tilde x,\bar x)\sigma^0(\tilde x, \mu)(\sigma^0)^\top(\bar x, \mu) \Big] \mu(d\tilde x)\mu(d\bar x).
\ea\right.
\eea
 We first specify the technical conditions. For $k\ge 0$ and a function $\Phi$ on $[0, T]\times \dbR^n \times \cP_2(\dbR^d)$,  with appropriate $n$-dimensional state variable $\bx$, let $\cC^{k,2}$ denote the set of those $\Phi$ such that $\Phi, \pa_\bx \Phi, \pa_{\bx\bx}\Phi \in \cC^k$, and $\cC^{k,2}_b$ denote the subset such that all the involved derivatives are bounded. Moreover, let $\|\Phi\|$ denote the uniform norm of $\Phi$ and $\|\Phi\|_2:= \|\Phi\| + \|\pa_\bx \Phi\|+ \|\pa_{\bx\bx}\Phi\|$.

\begin{assum}
\label{assum-MFGVC} Fix some integer $k\geq 2$.

\ms
\no(i) The Hamiltonian $H$ in \reff{MFGVC-H} admits a  maximizer $\phi$.

\ms

\no(ii) For $\Phi = g, \si^0$:  $\Phi\in \cC^{k,2}_b$ with $\bx = x$, and $\|\pa_{x}\Phi\|_2, \|\pa_{\mu}\Phi\|_2\leq L_\Phi$.

\ms
\no(iii) $H\in \cC^{k,2}$ with $\bx = (x, z^1, \g)$, and for any $R>0$, there exists $L(R)>0$ such that $\|\pa_{\bx}H\|_2$, $\|\pa_{\mu}H\|_2\leq L(R)$ and  all the involved higher order derivatives are bounded,  whenever $|z^1|, |\g|\le R$.

\ms
\no (iv) $\sigma^1$ is positive definite, and $|\si^0|, |(\si^1)^{-1}| \le L_0$.
\end{assum}

\begin{thm}\label{thm-MFGVC}
Let Assumption \ref{assum-MFGVC} hold and denote $R_0:= L_g + 1$, $R_1 := 2R_0 L_0$, $R_2:= R_0\vee R_1+2$. There exists a constant $\e_0>0$, depending only on $d$,  $L_g$, $L_{\sigma^0}$, $L_0$, and $L(R_2)$, such that the following hold whenever $0<T\leq\e_0$:

\no(i) The FBSDEs \reff{MFGVC-FBSDE1}-\reff{MFGVC-FBSDE2} are well-posed and the MFGVC admits an  MFE $(\a^*, \bm^*)$ at $(0,\mu)$. Moreover, $\bm^*$ is unique.

\no(ii) The master equation \eqref{MFGVC-master}  has a unique classical solution $V\in\cC^{k,2}_b$, and it holds that:
\bea
\label{MFGVC-FK}
\left.\ba{c}
\bm^*_t = \cL_{X^\xi_t|\cF^0_t},\q \a^*_t(x) = \phi\big(t, x, \pa_x V(t,x,\bm^*_t), \pa_{xx} V(t,x,\bm^*_t), \bm^*_t\big);\\
Y^{i,\xi}_t = \pa^{(i)}_x V(t, X^\xi_t, \bm^*_t),\q Z^{i, \xi}_t = \cI^{\pa_x^{(i)} V}_t(\Xi^{'\xi}_t),  \\
Y^{i,x}_t = \pa^{(i)}_x V(t, X^x_t, \bm^*_t),\q Z^{i, x}_t = \cI^{\pa_x^{(i)} V}_t(\Xi^{'x}_t), \\
\mbox{where}\q \cI^{\f,1}_t(\zeta') := \pa_x \f(t, x, \bm^*_t)\big(2\pa_{\g} H(\zeta')\big)^{1\over 2} ,\\
\dis \cI^{\f,0}_t(\zeta') :=  \pa_x \f(t, x, \bm^*_t) \si^0(x, \bm^*_t)+  \int_{\dbR^d} \pa_\mu \f(t, x, \bm^*_t, \tilde x)\si^0(\tilde x, \bm^*_t)\bm^*_t(d\tilde x).
 \ea\right.
\eea
Here $\pa^{(i)}_x V$ denotes the $i$-th derivative of $V$ with respect to $x$, $i=0,1,2$.
\end{thm}

\proof Recall \reff{paxv} and the truncation function $I_R$ in \reff{truncation}.  For the coefficients $\Phi$ in \reff{MFGVC-coefficients}, noting that, besides $(t,x,\mu)$, they involve only $y^1$, $y^2$, $z^{1,0}$, $z^{2,0}$, and $z^{2,1}$, in particular, the involvement of $z^{2,1}$ is always through $\cI_0(\zeta)$. We define $\Phi^R(t,\zeta)$ by replacing $y^1$, $y^2$, $z^{1,0}$, $z^{2,0}$, with their truncations: $I_R(y^1)$, $I_R(y^2)$, $I_R(z^{1,0})$, $I_R(z^{2,0})$, and replacing $z^{2,1}$ with $I_R(\cI_0(\zeta))~\! \big(2\pa_\g H(\zeta')\big)^{1\over 2}$.\footnote{This extra effort is to allow $\si^1$ to be unbounded. When $\si^1$ is bounded, which implies $\big(2\pa_\g H(\zeta')\big)^{1\over 2}$ is bounded, we can simply truncate $z^{2,1}$ directly.} Similarly let $\cI^R_i(\zeta)$ denote the truncated version of $\cI_i(\zeta)$:
\beaa
&\cI^R_0(\zeta) = I_R\Big( z^{2,1} \big(2\pa_\g H(\zeta')\big)^{-{1\over 2}}\Big),\q \cI^R_1(\zeta) := I_R(z^{1,0}) - I_R(y^2)\si^0(x, \mu),\\
&\cI^R_2(\zeta):= I_R(z^{2,0}) -  \cI^R_0(\zeta)\si^0(x, \mu).
\eeaa
  Under Assumption \ref{assum-MFGVC}, especially noting that Assumption \ref{assum-MFGVC} (iv) implies $|\big(2\pa_\g H(\zeta')\big)^{-{1\over 2}}|\le L_0$, we can easily verify that $\cI^R_i\in\cC^k_b$, $i=0,1,2$, and hence, by \reff{MFGVC-cH},  for all the functions $\Phi$ in \reff{MFGVC-coefficients}, we have $\Phi^R\in \cC^k_b$, with the bounds of the first order derivatives depending on $R$. 

 Now consider FBSDEs \reff{MFGVC-FBSDE1}-\reff{MFGVC-FBSDE2} with coefficients $\Phi^{R_2}$. Applying Theorem \ref{thm-Ck}, there exists a desired $\e_0>0$ such that these FBSDEs are well-posed whenever $T\le \e_0$, with the solutions  denoted as $(\Xi^{\xi}, \Xi^{x})$, and the corresponding functions $U:= U_1=U_2 \in \cC^k_b$. By Theorem \ref{thm-pax} (ii) and Theorem \ref{thm-pamuUrep} (ii), we have $|\pa_x U|, |\pa_\mu U|\le L_g+1 = R_0$.  Similarly, since $|\pa_x g|, |\pa_{xx} g|\le L_g$, one can obtain $|Y^{i,\xi}_t|\le R_0$, $i=1,2$, for a possibly smaller $\e_0$ depending on the same parameters. Note that, for $i=1,2$,
  \beaa
 &&Z^{i,\xi, 1}_t =\pa_x U^{i}(t, X^{\xi}_t, \bm^{*}_t) \big(2\pa_{\g} H(\Xi^{'\xi})\big)^{1\over 2},\q  \cI_0(\Xi^\xi_t) = \pa_x U^{2}(t, X^{\xi}_t, \bm^{*}_t),\\
 &&Z^{i,\xi, 0}_t =  \pa_x U^{i}(t, X^{\xi}_t, \bm^{*}_t)\si^0(X^{\xi}_t, \bm^{*}_t) +  \int_{\dbR^d} \pa_\mu U^{i}(t, X^{\xi}_t, \bm^{*}_t, \tilde x)\si^0(\tilde x, \bm^{*}_t)\bm^*_t(d\tilde x).
 \eeaa
 Then we have  $| \cI_0(\Xi^\xi_t) |\le R_0$ and $|Z^{i,\xi, 0}_t| \le 2L_0 R_0=R_1$. Similarly,  $| \cI_0(\Xi^x_t) |\le R_0$  and $|Z^{i,x, 0}_t| \le R_1$.
This implies that, for all the coefficients $\Phi$ in \reff{MFGVC-coefficients}, $\Phi^{R_2}(\Xi^\xi) = \Phi(\Xi^\xi)$ and $\Phi^{R_2}(\Xi^x) = \Phi(\Xi^x)$.
That is, $(\Xi^\xi, \Xi^{x})$ satisfy FBSDEs \reff{MFGVC-FBSDE1}-\reff{MFGVC-FBSDE2} with the original coefficients in \reff{MFGVC-coefficients}. Then the rest of the proof follows essentially the same arguments as in Theorem \ref{thm-EMFG}. In particular, in this case we have $U_1 = U_2 = (U^0, U^1, U^2)\in C^k_b$ with  $U^1 = \pa_x U^0$, $U^2=\pa_{xx} U^0$, and $V=U^0$. Then clearly $V\in \cC^{k,2}_b$.
\qed

\begin{rem}
\label{rem-MFGVC1}
(i) We do not assume the uniqueness of the optimizer $\phi$ in Assumption \ref{assum-MFGVC}. Consequently, the MFE $\a^*$ is in general not unique (for example in the trivial case that all the coefficients do not depend on the control). However, by Theorem \ref{thm-MFGVC}, the MFE $\bm^*$ is unique.

\no(ii) The master equation \reff{MFGVC-master} does not depend on $\si^1$ explicitly. For simplicity in this section we assume $\si^1$ is positive definite and thus \reff{MFGVC-paH} uniquely determines $\si^{\phi,1} = (2\pa_\g H)^{1\over 2}$, even when $\phi$ is not unique. Alternatively, we may allow $\si^1$ to be non-symmetric (but still nondegenerate), but $\phi$ is unique and $\si^{\phi,1}$ has desired regularity. In this case all the results in this section remain true, after obvious modifications, in particular the master equation will remain the same. Moreover, in this case $\a^*$ will also be unique.
\end{rem}

\begin{rem}
\label{rem-MFGVC2}
The problem is much harder when $\si^0$ is also controlled. In this case the Hamiltonian $H$ in the BSPDE \reff{MFGVC-BSPDE} will take the form $H\big(x, \pa_x u^{\bm}_t, \pa_{xx} u^{\bm}_t, \pa_x v^{\bm}_t, \bm_t\big)$. In particular, it involves the term $\pa_x v^\bm$. Again denote $Y^{i,\xi}_t = \pa^{(i)}_x u^{\bm^*}_t(X^\xi_t)$. Note that
\beaa
& Z^{1,\xi,0}_t = \pa_{xx}u^{\bm^*}_t(X^\xi_t)\sigma^{*0}_t + \pa_x v^{\bm^*}_t(X^\xi_t), \q\mbox{where}\\
&\sigma^{*0}_t := \si^0(X^\xi_t, \a^{\bm^*}_t(X^\xi_t), \bm^*_t)) = \pa_q H\Big(X^\xi_t, \pa_x u^{\bm^*}_t, \pa_{xx} u^{\bm^*}_t, \pa_x v^{\bm^*}_t, \bm^*_t\Big).
\eeaa
Here $\pa_q H$ refers to the derivative of $H$ with respect to $\pa_x v^{\bm^*}$. Then
\beaa
\pa_x v^{\bm^*}_t(X^\xi_t) = Z^{1,\xi,0}_t - Y^{2,\xi}_t\pa_q H\Big(X^\xi_t, Y^{1,\xi}_t, Y^{2,\xi}_t, \pa_x v^{\bm^*}_t(X^\xi_t), \bm^*_t\Big).
\eeaa
It requires very strong technical conditions to  solve the above equation to obtain $\pa_x v^{\bm^*}_t$. Moreover, $\pa_x v^{\bm^*}_t$ involves $Z^{1,\xi,0}_t$, consequently $\si^*$ will involve $Z^{1,\xi,0}_t$, which makes the resulted FBSDEs much harder to analyze. We also note that, it does not seem possible to circumvent this difficulty  by increasing further the dimension of $Y^\xi$.
\end{rem}

\section{Mean field games with a major player}
\label{sect-MFGM}
\setcounter{equation}{0}

In this section we study mean field games with a major player (MFGM). This game involves one major player with state $X^0$, and a population of minor players with state $X^1$. The dynamics of $X^1$ relies on the major player's state $X^0$, while the dynamics of $X^0$ relies on the aggregate behavior of all minor players, namely the (conditional) law of $X^1$. Note that all minor players rely on $X^0$, then it is natural to consider the conditional law of $X^1$, conditional on $X^0$, even if we do not consider additional common noise. For simplicity in this section we do not consider common noise anymore. Moreover, we consider drift controls only.

Consider the probabilistic setting in the beginning of Section \ref{sect-setting} again. However, here $B^0$, $B^1$ stand for the randomness for the major player and the representative minor player, respectively. Let $A_0, A_1$ denote the control set of the major and minor players, respectively. The MFGM will involve the following data, and again we may omit the time variable $t$: for $i=0,1$ and denoting $d_0':= d_0$, $d_1':= d_0+d_1$,
\beaa
&& \si_i: [0, T]\times \dbR^{d_i'} \times \cP_2(\dbR^{d_1}) \to \dbR^{d_i\times d_i},\q g_i: \dbR^{d_i'} \times \cP_2(\dbR^{d_1})\to \dbR,\\
&&(b_i, f_i): [0, T]\times \dbR^{d_i'} \times A_i \times \cP_2(\dbR^{d_1}) \to (\dbR^{d_i}, \dbR).
\eeaa
Here $(b_0, \si_0, f_0, g_0)$ are for the major player and $(b_1, \si_1, f_1, g_1)$ are for the minor players.

We emphasize that the conditional law of $X^1$ is conditional on $X^0$, we thus consider $\bm\in \cM$ as a mapping as follows. Denote $\dbX^0:= C([0, T]; \dbR^{d_0})$, we consider mappings $\bm: \dbX^0 \to C([0, T]; \cP_2(\dbR^{d_1}))$, which are Lipschitz continuous in $\bx$,  and adapted in the sense that $\bm_t(\bx) = \bm_t(\bx_{[0, t]})$ for any $(t, \bx)\in [0, T]\times \dbX^0$. Let  $\cA_0$ denote the the set of major player's admissible controls  $\a^0: [0, T]\times \dbX^0 \to A_0$, which are progressively measurable and adapted in $t$. Let  $\cA_1$ denote the the set of representative minor player's admissible controls  $\a^1: [0, T]\times \dbX^0 \times \dbR^{d_1}\to A_1$, which are also progressively measurable and adapted in $t$. We note that as before we assume $\a^1$ depends on $X^1$ in a state dependent manner, which is without loss of generality since in the end the MFE will be state dependent. However, since $X^0$ will play the role of common noise, it is more natural to allow $\bm, \a^0, \a$ to depend on the paths of $X^0$.

Now given  $x=(x_0, x_1)\in \dbR^{d_0+d_1}$,  $\a=(\a^0, \a^1)\in \cA:= \cA_0\times \cA_1$, and $\bm\in \cM$, consider the following dynamics: for $X^{\bm, 0} = X^{\bm, 0, x_0, \a^0}$, $X^{\bm,1} = X^{\bm, 1, x,  \a}$, and $X^\bm = (X^{\bm,0}, X^{\bm, 1})$,
\bea
\label{MFGM-X}
\left.\ba{lll}
\dis X^{\bm,0}_t = x_0 +\int_0^t \si_0 b_0\big(X^{\bm, 0}_s, \a^0_s(X^{\bm, 0}), \bm_s(X^{\bm, 0})\big) ds + \int_0^t \si_0\big(X^{\bm, 0}_s, \bm_s(X^{\bm, 0})\big)dB^0_s,\\
\dis X^{\bm, 1}_t = x_1 +\int_0^t \si_1 b_1\big(X^\bm_s, \a^1_s(X^{\bm,0}, X^{\bm,1}_s), \bm_s(X^{\bm,0})\big) ds + \int_0^t \si_1\big(X^\bm_s, \bm_s(X^{\bm, 0})\big)dB^1_s.
\ea\right.
\eea
Here we omit the variable $t$ again. Note that we may consider weak solutions whenever needed. The major player and the representative minor player aim to minimize the following utilities:
\bea
\label{MFGM-J}
\left.\ba{lll}
\dis J_0( \mathbf{m};  x_0, \a^0) := \dbE\Big[g_0(X^{\bm,0}_T, \bm_T(X^{\bm,0})) + \int_{0}^T f_0\big(X^{\bm,0}_s,\alpha^0_s(X^{\bm,0}), \bm_s(X^{\bm,0})\big)ds\Big];\\
\dis J_1( \mathbf{m};  x, \a) := \dbE\Big[g_1(X^\bm_T, \bm_T(X^{\bm,0})) + \int_{0}^T f_1\big(X^\bm_s, \alpha_s(X^{\bm,0}, X^{\bm,1}_s), \bm_s(X^{\bm,0})\big)ds\Big].
\ea\right.
\eea
Moreover, given $\xi_1\in \dbL^2(\cF_0, \dbR^{d_1})$, we introduce: for $X^{\bm, \xi_1} = X^{\bm, 1, x_0, \xi_1, \a}$,
\bea
\label{MFGM-Xxi}
\left.\ba{c}
\dis X^{\bm, \xi_1}_t := \xi_1 +\int_0^t \si_1 b_1\big(X^{\bm,0}_s, X^{\bm,\xi_1}_s, \a^1_s(X^{\bm,0}, X^{\bm,\xi_1}_s), \bm_s(X^{\bm,0})\big) ds \\
\dis+ \int_0^t \si_1\big(X^{\bm,0}_s, X^{\bm,\xi_1}_s,  \bm_s(X^{\bm,0})\big)dB^1_s.
\ea\right.
\eea

\begin{defn}
\label{defn-MFGMMFE}
Given $\cL_{\xi_1} = \mu_1\in \cP_2(\dbR^{d_1})$ and $x_0\in \dbR^{d_0}$, we say $(\a^*, \bm^*)\in \cA\times \cM$ is a mean field equilibrium (MFE) of MFGM \reff{MFGM-X}-\reff{MFGM-J}-\reff{MFGM-Xxi} at $(0,x_0, \mu_1)$ if

(i) Given $\bm^*$, $\a^{0*}$ is optimal for the major player:
\bea
\label{MFGM-a0*optimal}
J_0(\bm^*;  x_0, \a^{0*}) = \inf_{\a^0\in \cA_0} J_0(\bm^*;  x_0, \a^{0}).
\eea

(ii) Given $(\bm^*,\a^{0*})$, $\a^{1*}$ is optimal for the representative minor player:
\bea
\label{MFGM-a1*optimal}
J_1(\bm^*;  x_0, x_1, \a^{0*}, \a^{1*}) = \inf_{\a^1\in \cA_1} J_1(\bm^*;  x_0, x_1, \a^{0*}, \a^1),\q \mbox{for $\mu_1$-a.e. $x_1\in \dbR^{d_1}$}.
\eea

(iii) Given $\a^*$, $\bm^*$ is the conditional law of the minor's population:
\bea
\label{MFGM-m*optimal}
\cL_{X^{\bm^*,1, x_0, \xi_1, \a^*}_t|\cF^{X^{\bm^*,0, x_0, \a^{0*}}_t}} = \bm^*_t(X^{\bm^*, 0, x_0, \a^{0*}}),~ 0\le t\le T,~\mbox{ a.s.}
\eea
\end{defn}
Moreover, when MFE $(\a^*, \bm^*)$ is unique, we define the value  of the MFGM at $(0, x, \mu_1)$:
\bea
\label{MFGM-V1}
V_0(0, x_0, \mu_1) := J_0(\bm^*;  x_0, \a^{0*}),\q V_1(0, x, \mu_1) := J(\bm^*;  x, \a^*).
\eea

\begin{rem}
\label{rem-MFGM}
In this paper we require $\bm$ to be Lipschitz continuous, and we will find MFE such that $\dbF^{X^{\bm^*, 0, x_0, \a^{0*}}}=\dbF^0$. In this sense we call our equilibrium a strong MFE. We may relax the Lipschitz continuity,  and consider weak solutions to the involved SDEs. Then the resulted equilibria are called weak MFE. This will particularly be effective when $\si_0=\si_0(t, x_0)$ does not depend on $\mu_1$ and we can use the weak formulation to formulate the MFG.
\end{rem}

Given $\bm^*$, \reff{MFGM-a0*optimal} and \reff{MFGM-a1*optimal} are standard optimization problems with closed loop controls. Following the standard theory, we can easily obtain:
\bea
\label{MFGM-JY}
\dis J_0(\bm^*;  x_0, \a^{0*}) = Y^{\bm^*,0}_0,\q J_1(\bm^*;  x, \a^{*}) = Y^{\bm^*,1}_0,
\eea
where $\Pi^{\bm^*} = (\Pi^{\bm^*,0}, \Pi^{\bm^*,1})$ solves the following FBSDEs: denoting $\bm^{*0}:= \bm^*(X^{\bm^*,0})$,
\bea
\label{MFGM-BSDE0}
&&\!\!\!\!\!\!\!\!\!\! \dis \left\{\ba{lll}
\dis X^{\bm^*,0}_t = x_0 + \int_0^t \si_0 \pa_{z_0} H_0( X^{\bm^*,0}_s, Z^{\bm^*,0}_s, \bm^{*0}_s) ds + \int_0^t \si_0( X^{\bm^*,0}_s, \bm^{*0}) dB^0_s,\\
\dis Y^{\bm^*,0}_t = g_0(X^{\bm^*,0}_T, \bm^{*0}_T)  + \int_t^T \!\!\!\big[H_0- \pa_{z_0} H_0 \cd Z^{\bm^*,0}_s\big](X^{\bm^*,0}_s, Z^{\bm^*,0}_s, \bm^{*0}_s) ds - \int_t^T \!\!\! Z^{\bm^*,0}_s dB^0_s,
\ea\right.\\
\label{MFGM-BSDE1}
&&\!\!\!\!\!\!\!\!\!\!\dis \left\{\ba{lll}
\dis X^{\bm^*,1}_t = x_1 + \int_0^t \si_1 \pa_{z_1^1} H(X^{\bm^*}_s, Z^{\bm^*,1,1}_s, \bm^{*0}_s) ds + \int_0^t \si_1( X^{\bm^*}_s, \bm^{*0}_s) dB^1_s,\\
\dis Y^{\bm^*,1}_t = g_1(X^{\bm^*}_T, \bm^{*0}_T) + \int_t^T\!\!\! \big[H_1- \pa_{z_1^1} H_1 \cd Z^{\bm^*,1,1}_s\big](X^{\bm^*}_s, Z^{\bm^*,1,1}_s, \bm^{*0}_s) ds - \int_t^T\!\!\! Z^{\bm^*,1}_s dB_s,
\ea\right.
\eea
with the Hamiltonians $H_i$ given by:
\bea
\label{MFGM-H}
\left.\ba{c}
\dis H_0(x_0, z_0, \mu) := \inf_{a_0\in A_0}\big[f_0(x_0, a_0, \mu) + b_0(x_0, a_0, \mu) \cd z_0\big],\\
\dis H_1(x, z^1_1, \mu) := \inf_{a_1\in A_1}\big[f_1(x, a_1, \mu) + b_1(x, a_1, \mu) \cd z_1^1\big].
\ea\right.
\eea
We shall assume $\si_0$ is nondegenerate, and thus in \reff{MFGM-BSDE0} we will have $\dbF^{X^{\bm^*,0}} = \dbF^0$.

The fixed point property \reff{MFGM-m*optimal}, however, does not transform FBSDEs \reff{MFGM-Xxi}-\reff{MFGM-BSDE0}-\reff{MFGM-BSDE1} into the form of \reff{Xxi}-\reff{X2x}, and thus prevents us from applying our approach directly. To overcome this difficulty, we introduce an auxiliary problem which replaces $\bm^*$ with open loop ones, namely we replace it with $\cL_{X^{\bm^*,1, x_0, \xi_1, \a^*}_t|\cF^0_t}$ in \reff{MFGM-m*optimal}. Note also that \reff{MFGM-Xxi} is coupled with the $X^{\bm,0}$ in \reff{MFGM-X}. We thus introduce the following systems of FBSDEs: given $\xi_0\in \dbL^2(\cF_0,\dbR^{d_0})$ and $\xi=(\xi_0, \xi_1)$,
 \bea
\label{MFGM-BSDE'}
\left\{\ba{lll}
\dis X^{\xi,0}_t = \xi_0 + \int_0^t  \si_0\pa_{z_0} H_0(X^{\xi,0}_s, Z^{\xi,0}_s,  \cL_{X^{\xi,1}_s|\cF^{0}_s}) ds + \int_0^t \si_0(X^{\xi,0}_s, \cL_{X^{\xi,1}_s|\cF^{0}_s}) dB^{0}_s,\\
\dis X^{\xi,1}_t = \xi_1  +\int_0^t \si_1\pa_{z^1} H_1(X^{\xi}_s, Z^{\xi,1,1}_s, \cL_{X^{\xi,1}_s|\cF^{0}_s}) ds + \int_0^t \si_1(X^{\xi}_s, \cL_{X^{\xi,1}_s|\cF^{0}_s}) dB^1_s,\\
\dis Y^{\xi,0}_t = g_0(X^{\xi,0}_T,  \cL_{X^{\xi,1}_T|\cF^{0}_T})  - \int_t^T Z^{\xi,0}_s dB^{0}_s\\
\dis \qq + \int_t^T \Big[H_0 - \pa_{z_0}H_0 \cd Z^{\xi,0}_s \Big](X^{\xi,0}_s, Z^{\xi,0}_s,  \cL_{X^{\xi,1}_s|\cF^{0}_s})ds,\\
\dis Y^{\xi,1}_t = g_1(X^{\xi}_T, \cL_{X^{\xi,1}_T|\cF^{0}_T}) - \int_t^T Z^{\xi,1}_s dB_s\\
\dis \qq + \int_t^T \Big[H_1 - \pa_{z^1} H_1 \cd Z^{\xi,1,1}_s\Big](X^{\xi}_s, Z^{\xi,1,1}_s,  \cL_{X^{\xi,1}_s|\cF^{0}_s}) ds,
\ea\right.
\eea
and, denoting $\bm^{*0'}_t:= \cL_{X^{\xi,1}_t|\cF^{0}_t}$,
\bea
\label{MFGM-BSDE'0}
&&\!\!\!\!\!\!\!\!\!\! \dis \left\{\ba{lll}
\dis X^{x,0}_t = x_0 + \int_0^t \si_0 \pa_{z_0} H_0( X^{x,0}_s, Z^{x,0}_s, \bm^{*0'}_s) ds + \int_0^t \si_0( X^{x,0}_s, \bm^{*0'}) dB^0_s,\\
\dis Y^{x,0}_t = g_0(X^{x,0}_T, \bm^{*0'}_T)  + \int_t^T \Big[H_0- \pa_{z_0} H_0 \cd Z^{x,0}_s\Big](X^{x,0}_s, Z^{x,0}_s, \bm^{*0'}_s) ds - \int_t^T Z^{x,0}_s dB^0_s,
\ea\right.
\eea
\bea
\label{MFGM-BSDE'1}
&&\!\!\!\!\!\!\!\!\!\!\dis \left\{\ba{lll}
\dis X^{x,1}_t = x_1 + \int_0^t \si_1 \pa_{z_1^1} H(X^{x}_s, Z^{x,1,1}_s, \bm^{*0'}_s) ds + \int_0^t \si_1( X^{x}_s, \bm^{*0'}_s) dB^1_s,\\
\dis Y^{x,1}_t = g_1(X^{x}_T, \bm^{*0'}_T) + \int_t^T \Big[H_1- \pa_{z_1^1} H_1 \cd Z^{x,1,1}_s\Big](X^{x}_s, Z^{x,1,1}_s, \bm^{*0'}_s) ds - \int_t^T Z^{x,1}_s dB_s.
\ea\right.
\eea
 Now the FBSDEs \reff{MFGM-BSDE'}-\reff{MFGM-BSDE'0}-\reff{MFGM-BSDE'1} are in the form of \reff{Xxi}-\reff{X2x}, with  $\Phi_1=\Phi_2$ for $\Phi=b, \si, f, g$. To be precise, for $i=1,2$, and  $\zeta=(x,y, z, \mu)$, where $\mu = \cL_{\xi}$ and $\mu_1:= \cL_{\xi_1}$  denoting the second marginal of $\mu$,  we have
\bea
\label{MFGM-coefficients}
\left.\ba{c}
 d_x^i= d_0+d_1, \q d_y^i = 2,\\
b^0_i(\zeta) =\si_0(x_0,\mu_1) \pa_{z_0} H_0(x_0, z_0, \mu_1),\q b^1_i(\zeta) = \si_1(x, \mu_1) \pa_{z^1} H_1(x, z_1^{1}, \mu_1), \\
\si^0_i(\zeta) = (\si_0(x_0,\mu_1), 0), ~ \si_i^1(\zeta) = (0, \si(x,\mu_1)),\q g_i^0(\zeta) = g_0(x_0, \mu_1),~ g_i^1(\zeta) = g(x, \mu_1), \\
f_i^0(\zeta)\!=\! H_0(x_0, z_0, \mu_1) - \pa_{z_0} H_0(x_0, z_0, \mu_1) \!\cd\! z_0, ~ f_i^1(\zeta)\! =\! H_1(x, z_1^{1}, \mu_1) - \pa_{z^1_1}H_1(x, z_1^{1}, \mu_1)\!\cd\! z^1_1.
\ea\right.
\eea
Clearly the corresponding value functions are equal: $U_1 = U_2 =: U$, with $U(0, x, \cL_{\xi}) = Y^{x}_0$. Write $U = (U^{0}, U^{1})$. We remark that, despite the possible difference between $\bm^{*0}$ and $\bm^{*0'}=\cL_{X^{\xi,1}|\dbF^{0}}$, the $V_0$ and $V_1$ in \reff{MFGM-V1} correspond to $U^{0}(0, x, \cL_{(x_0, \xi_1)})$ and $U^{1}(0, x, \cL_{(x_0, \xi_1)})$, respectively. In particular, since \reff{MFGM-BSDE'0} and \reff{MFGM-BSDE'1} are decoupled, it is clear that $U^{0}$ does not depend on $x_1$, namely $\pa_{x_1} U^{0}=0$. We thus define
\bea
\label{MFGM-V'}
V'_0(t, x_0, \cL_{\xi_1}) := U^{0}(t, x_0, x_1, \cL_{(x_0, \xi_1)}),\q V'(t, x, \cL_{\xi_1}):= U^{1}(t, x, \cL_{(x_0, \xi_1)}).
\eea

Now given \reff{MFGM-coefficients}, one may derive  from \reff{patU} and \reff{LU1}-\reff{LU2} the master equation system for $U$. We shall remark that, a McKean-Vlasov SDE may lose the flow property when restricting to Dirac measures as in \reff{MFGM-V'}, see e.g. \cite{BLPR} for detailed discussions. However, our FBSDE \reff{MFGM-BSDE'} has a very special structure. That is, given $\xi=(x_0, \xi_1)$, $\Pi^{\xi,0}$ is $\dbF^0$-progressively measurable and the system does not involve the (conditional) distribution of $X^{\xi,0}$. Then actually FBSDE \reff{MFGM-BSDE'}  satisfies the flow property and $V'=(V'_0, V'_1)$ serves as the decoupling field: for $\xi=(x_0, \xi_1)$,
\bea
\label{MFGM-FK'}
Y^{\xi,0}_t = V'_0(t, X^{\xi,0}, \bm^{*0'}),\q Y^{\xi,1}_t = V'_1(t, X^{\xi}_t, \bm^{*0'}).
\eea
We may derive the master equation for $V'$ from that for $U$. However, it is more convenient to derive it directly from \reff{MFGM-FK'} by applying the It\^{o} formula \reff{Ito}. We thus obtain the following coupled system of master equations, written in terms of $V = (V_0, V_1)$:
\begin{equation}
\label{MFGM-master}
\left.\ba{lll}
\dis \!\!\!\!\!\!\!\!\! \pa_t V_0 + \frac{1}{2} \pa_{x_0x_0} V_0 : \sigma_0\sigma_0^\top +  H_{0}(x_0,\partial_{x_0} V_0\sigma_0,\mu_1) + \dbM_0 V =0, ~ V_0(T,x_0,\mu_1) = g_0(x_0,\mu_1),\ms\\
\dis\!\!\!\!\!\!\!\!\! \pa_t V_1 + \frac{1}{2} \pa_{xx} V_1: \sigma_1\sigma_1^\top+{1\over 2}\pa_{x_0x_0} V_1: \sigma_0(\sigma_0)^\top+(\pa_{x_0}V_1)^\top : \sigma_0\pa_{z_0}H_{0}\big(x_0,\pa_{x_0}V_0\sigma_0,\mu_1\big) \\
\dis \q +  H_1(x,\partial_{x_1} V_1\sigma_1,\mu_1) + \dbM_1 V =0, \q V_1(T,x,\mu_1) = g_1(x,\mu_1),
\ea\right.
\end{equation}
where,
\begin{equation*}
\left.\ba{lll}
\dis \dbM_0 V(t, x_0, \mu_1)  :=\int_{\dbR^{d_1}}\Big[\frac{1}{2}\pa_{\tilde x_1\mu_1} V_0(t, x_0, \mu_1, \tilde x_1): \sigma_1\si_1^\top(x_0,\tilde x_1,\mu_1)  \ms\\
\dis+ \pa_{\mu_1} V_0(t, x_0, \mu, \tilde x_1) \cd \sigma_1(x_0,\tilde x_1,\mu_1) \pa_{z^1_1} H_1\big(x_0,\tilde x_1,\pa_{x_1} V_1(t, x_0,\tilde x_1, \mu_1)\sigma_1(x_0,\tilde x_1,\mu_1),\mu_1\big)\Big]\mu_1(d\tilde x_1),\ms\\
\dis \dbM_1 V(t, x,\mu_1)  :=\int_{\dbR^{d_1}} \Big[\frac{1}{2}\pa_{\tilde x_1\mu_1}  V_1(t,x, \mu_1, \tilde x_1):\sigma_1\si_1^\top(x_0,\tilde x,\mu_1)  \ms\\
\dis+\pa_{\mu_1} V_1(t,x, \mu_1, \tilde x_1) \cd \sigma_1(x_0,\tilde x_1,\mu_1)\pa_{z^1_1} H_1\big(x_0,\tilde x_1,\pa_{x_1} V_1(t, x_0,\tilde x_1, \mu_1)\sigma_1(x_0,\tilde x_1,\mu_1),\mu_1\big)\Big] \mu_1(d\tilde x_1).
\ea\right.
\end{equation*}

 We now establish the theory rigorously. We extend all the functions to $(x, z, \mu_1)$.

\begin{assum}
\label{assum-MFGM}
Fix some integer $k\geq 2$.

\ms
\no(i) For $i=0,1$, the Hamiltonian $H_i$ in \reff{MFGM-H} admits a  maximizer $\phi_i$.

\ms

\no(ii) For $\Phi = g, \si$:  $\Phi_i\in \cC^{k}_b$ and $|\pa_{x}\Phi_i|, |\pa_{\mu_1}\Phi_i|\leq L_\Phi$, $i=0,1$.

\ms
\no(iii) For $\Phi= H_0, \pa_{z_0} H_0, H_1, \pa_{z^1_1}H_1$:  $\Phi \in \cC^{k}$, and for any $R>0$, there exists $L(R)>0$ such that $|\pa_{x}\Phi|, |\pa_{z}\Phi|, |\pa_{\mu_1}\Phi|\leq L(R)$ and  all the involved higher order derivatives are bounded,  whenever $|z_0|, |z^1_1|\le R$.

\ms
\no (iv) $\si_0$ is nondegenerate, and $|\si_0|, |\si_1| \le L_0$.
\end{assum}
 Here we need $\si_0$ to be non-degenerate so that $\dbF^{X^0}=\dbF^0$ in \reff{MFGM-X2} below.  As pointed out in Remark \ref{rem-EMFG-nondegenerate}, if $\si_0$ and $\si_1$ are uniformly non-degenerate, we can allow them to be unbounded.

\begin{thm}\label{thm-MFGM}
Let Assumption \ref{assum-MFGM} hold and denote $R_0:= L_g + 1$, $R_1 := R_0 L_0$, $R_2:=  R_1+2$. Then there exists a constant $\e_0>0$, depending only on $d_0$, $d_1$, $L_g$, $L_{\sigma}$, $L_0$, and $L(R_2)$, such that  the following hold whenever $0<T\leq\e_0$:

\no(i) The MFGM  \reff{MFGM-X}-\reff{MFGM-J}-\reff{MFGM-Xxi} admits an  MFE $(\a^*, \bm^*)$ at $(0, x_0, \mu)$, and $\bm^*$ is unique.

\no(ii) The master equation \eqref{MFGM-master}  has a unique classical solution $V = (V_0, V_1)\in\cC^{k}_b$.
\end{thm}

\proof We proceed in three steps.

{\bf Step 1.} Following the proof of Theorem \ref{thm-EMFG} and \ref{thm-MFGVC}, we can easily see that, for appropriate $\e_0>0$ and $T\le \e_0$, FBSDEs \reff{MFGM-BSDE'}-\reff{MFGM-BSDE'0}- \reff{MFGM-BSDE'1} are well-posed, and the corresponding function $U:= U_1=U_2\in \cC^k_b$. Consequently, $V' = (V'_0, V'_1)$ defined by \reff{MFGM-V'} is in $\cC^k_b$ and satisfies the master equation system \reff{MFGM-master}.

{\bf Step 2.} Fix $\xi_1\in \dbL^2(\cF_0, \dbR^{d_1})$. For each $\bx\in \dbX^0$, consider the McKean-Valsov SDE:
\bea
\label{MFGM-X1}
 \left.\ba{c}
 \dis X^{\bx, \xi_1}_t = \xi_1  +\int_0^t \si_1 \pa_{z_1^1} H_1\Big(\bx_s, X^{\bx, \xi_1}_s, \pa_{x_1} V_1'\si_1(s, \bx_s, X^{\bx, \xi_1}_s, \cL_{X^{\bx, \xi_1}_s}),\cL_{X^{\bx, \xi_1}_s}\Big) ds\\
 \dis + \int_0^t \si_1\big(\bx_s, X^{\bx, \xi_1}_s, \cL_{X^{\bx, \xi_1}_s}) dB^1_s.
\ea\right.
\eea
Under our conditions the above SDE is well-posed. Introduce
\bea
\label{Psi}
\Psi_t(\bx) := \cL_{X^{\bx, \xi_1}_t}.
\eea
 Clearly the mapping $\Psi: [0, T]\times \dbX^0\to \cP_2(\dbR^{d_1})$ is adapted and Lipschitz continuous:
\beaa
\sup_t W_2 (\Psi_t(\bx), \Psi_t(\bx')) \le C \|\bx - \bx'\|,\q\mbox{where}\q \|\bx - \bx'\| := \sup_t |\bx_t - \bx'_t|.
\eeaa

Next, fix $x_0\in \dbR^{d_0}$ and  consider SDE
\bea
 \label{MFGM-X2}
 X^{0}_t = x_0 + \int_0^t \!\! \si_0\pa_{z_0} H_0\big(X^{0}_s,  \pa_{x_0}  V_0'  \si_0\big( X^{0},  \Psi_s(X^{0})\big),  \Psi_s(X^{0})\big) ds + \int_0^t\!\! \si_0(X^{0}_s, \Psi_s(X^{0})) dB^{0}_s.
  \eea
  By the desired regularity of $V'$ and the Lipschitz continuity of $\Psi$,  SDE \reff{MFGM-X2} is well-posed. Moreover, since $\si_0$ is nondegenerate, we have $\dbF^{X^{0}} = \dbF^{0}$.
Therefore, $\bm^{*0'}_t:= \Psi_t(X^{0})$ is $\dbF^0$-progressively measurable.

Now consider the following SDE, corresponding to  \reff{MFGM-X1} with $\bx = X^{0}$:
\bea
\label{MFGM-X3}
 X^{1}_t = \xi_1  + \!\!\int_0^t \!\! \si_1 \pa_{z_1^1} H_1\big(X^0_s, X^{1}_s, \pa_{x_1} V_1' \si_1(X^{0}_s, X^{1}_s, \bm^{*0'}_s),\bm^{*0'}_s\big) ds + \!\!\int_0^t\!\! \si_1\big(X^0_s, X^{1}_s, \bm^{*0'}_s\big) dB^1_s.
\eea
 Note that $X^0$ is $\dbF^0$-progressively measurable. Then, by the definition of the mapping $\Psi$, we have
 \bea
 \label{Psi2}
 \bm^{*0'}_t = \cL_{X^1_t|\cF^0_t}.
 \eea
Denote $X=(X^0, X^1)$ and introduce further that
 \beaa
 Y^0_t := V_0'(t, X^0_t,  \bm^{*0'}_t),\q Y^1_t := V_1'(t, X_t, \bm^{*0'}_t), \q Y=(Y^0, Y^1).
 \eeaa
 Since $V_0', V_1'$ satisfy the master equations \reff{MFGM-master}, by applying the It\^o formula \reff{Ito} one may verify straightforwardly that, for appropriate $Z = (Z^0, Z^1)$,  $(X, Y, Z)$ satisfy FBSDE \reff{MFGM-BSDE'}.

Moreover, define the mapping $\bm^*_t(\bx) := \Psi_t(\bx)$. Comparing FBSDEs \reff{MFGM-BSDE0} and \reff{MFGM-BSDE'0}, we see that $(X^{\bm^*,0},  Y^{\bm^*,0}, Z^{\bm^*,0})=(X^0, Y^0, Z^0)$, and $\bm^*_t(X^{\bm^*,0}) = \bm^{*0'}_t$. Then, recalling \reff{MFGM-Xxi} and \reff{MFGM-H}, we see $X^{\bm^*, 1, x_0, \xi_1, \a^*} = X^1$, and thus
\beaa
\cL_{X^{\bm^*,1, x_0, \xi_1, \a^*}_t|\cF^{X^{\bm^*, 0, x_0, \a^{0*}}_t}} = \cL_{X^1_t|\cF^{X^0}_t}= \cL_{X^1_t|\cF^0_t} = \bm^{*0'}_t = \bm^*_t(X^{\bm^*,0}).
\eeaa
This verifies \reff{MFGM-m*optimal}, and therefore, $\bm^*$ is an MFE. Consequently, by \reff{MFGM-V1} we have
\beaa
V(0, x, \mu_1) = Y_0 = V'(0, x, \mu_1).
\eeaa
Then $V = V'$ is in $\cC^k_b$ and satisfies master equation \reff{MFGM-master}.

{\bf Step 3.} Finally we prove the uniqueness. First, the uniqueness of classical solution $V$ to master equation \reff{MFGM-master} follows from standard arguments, for example  by the uniqueness of FBSDEs \reff{MFGM-BSDE'} whose solution can be induced by $V$. Next, for arbitrary MFE $\bm^*$, since it is Lipschitz continuous and $\si_0$ is nondegenerate, by \reff{MFGM-BSDE0} we see that $\dbF^{X^{\bm^*,0}} = \dbF^0$. Then the fixed point property \reff{MFGM-m*optimal} exactly implies that $\bm^* = \Psi$ for the $\Psi$ in \reff{Psi}, and thus $\bm^*$ is unique.
\qed

\section{Proof of Theorem \ref{thm-pax}}
\label{sect-appendix}
\setcounter{equation}{0}

We prove the theorem in six steps. Many FBSDE estimates are standard, cf. \cite[Section 8.2]{Zhang}, and thus we may skip some details. Recall $R_0:= L_g+1$.

{\bf Step 1.}  By standard FBSDE arguments (cf. \cite[Section 8.2]{Zhang}), one can easily see that there exists $\e_0>0$, depending only on $d_{01}, d_x, d_y$, and $L_\Phi$, $\Phi=b,\si, f, g$, such that, whenever $T\le \e_0$,  the FBSDEs \reff{FBSDEx} and \reff{tdFBSDEx} are well-posed.

The well-posedness implies the existence of the decoupling field. That is, there exist $\dbF^0$-progressively measurable random field $u$ and $\dbR^{d_y\times d_x}$-valued process $\Si_t$ such that
\bea
\label{decouple}
Y^x_t = u_t(X^x_t),\q \td Y_t = \Si_t \td X_t.
\eea
In particular, the existence of $\Si_t$ is obviously due to the linear structure of \reff{tdFBSDEx}, in fact, $\Si_t = \hat \td_x \Pi^x$ as in Footnote \ref{matrix}, which depends on $x$ but not on $\D x$.

Moreover, by \reff{tdFBSDEx} one can easily see that $\|\td_x \Pi\|_2 \le C|\D x|$. Since $|\pa_x g|\le L_g$, one can easily see that, for a generic constant $C$ depending on the parameters specified in the lemma,
\beaa
\dbE[\sup_{0\le t\le T} |\td_x X_t|^2] \le (1+CT)|\D x|^2,\q |\td_x Y_0|^2 \le \Big[L_g^2(1+CT) + CT\Big]|\D x|^2.
\eeaa
Then, for a possibly smaller $\e_0$ and for $T\le \e_0$, we have $|\td_x Y_0| \le R_0 |\D x|$. This implies $|\Si_0|\le R_0$. Similarly, $|\Si_t|\le R_0$, and $u$ is uniformly Lipschitz continuous in $x$ with Lipschitz constant $R_0$.

{\bf Step 2.} We next prove \reff{Lpest1},  following the  arguments in \cite[Corollary 8.4]{MWZZ}. Note  that $\si$ does not depend on $z$, thus $\pa_\pi\si \td_x \Pi$ does not involve $\td_x Z$. Fix $p\ge 2$.  By the boundedness of $\pa_\pi \si$, it follows from standard arguments that there exists $\e_p\le \e_0$ such that \reff{Lpest1} holds true whenever $T\le \e_p$ and the terminal condition satisfies $|\td_x Y_T|\le R_0 |\td_x X_T|$. Now for general $T\le \e_0$, let $0=T_0<\cds<T_m = T$ be a partition of $[0, T]$ such that $\D T_i \le \e_p$. Note that $\td_x\Pi$ satisfies FBSDE \reff{tdFBSDEx}  on $[T_0, T_1]$ with initial condition $\D x$ and terminal condition $\td_x Y_{T_1} = \Si_{T_1}\td_x X_{T_1}$. Since $|\Si_{T_1}| \le R_0$, we have $\|\td_x \Pi \1_{[T_0, T_1]}\|_p \le C_p |\D x|$. In particular, $\|\td_x X_{T_1}\|_p \le C_p |\D x|$. Next, consider FBSDE \reff{tdFBSDEx}  on $[T_1, T_2]$ with initial condition $\td_x X_{T_1}$ and terminal condition $\td_x Y_{T_2} = \Si_{T_2} \td_x X_{T_2}$. Since $|\Si_{T_2}| \le R_0$, we have $\|\td_x \Pi \1_{[T_1, T_2]}\|_p \le C_p \|\td_x X_{T_1}\|_p \le C_p^2 |\D x|$. Repeat the arguments, we have $\|\td_x \Pi \1_{[T_{i-1}, T_{i}]}\|_p \le  C_p^i |\D x|$ for $i=1,\cds, m$. This implies  \reff{Lpest1} immediately for a larger $C_p$. We emphasize that $\e_0$ is independent of the $p$ here.

{\bf Step 3.} In this step we show that $\psi_i(x):= \td_x Y^{x, e_i}_0$ is locally uniformly continuous in $x$. Recall the $I_0$ in Assumption \ref{assum-0} (i). Fix $R_1\ge I_0$  and $x^1, x^2\in \dbR^{d_x}$ such that $|x^1|, |x^2|\le R_1$. Denote $\D x:= x^1-x^2$, $\D \Pi := \Pi^{x^1}-\Pi^{x^2}$, $\D^i \td_x \Pi = \td_x \Pi^{x^1,e_i}-\td_x \Pi^{x^2, e_i}$, $\D \Phi := \Phi(\Pi^{x^1}) -  \Phi(\Pi^{x^2})$.  By standard FBSDE estimates, we have
\bea
\label{Piest}
&&\!\!\!\!\!\!\!\!\!\!\!\!  \!\!\!\!\!\!\!\!\!\!\!\!  \|\Pi^{x^j}\|_2 \le C(|x^j|+I_0) \le CR_1,\q \|\td_x \Pi^{x^j,e_i}\|_4 \le C,~ j=1,2;\q \|\D \Pi\|_2 \le C|\D x|;\\
&&\!\!\!\! \!\!\!\!\!\!\!\! \!\!\!\!\!\!\!\!\!\!\!\!  \|\D^i \td_x \Pi\|_2^2 \le C\dbE\Big[ |\D \pa_x g|^2 |\td_x X^{x^1,e_i}_T|^2 + \Big(\sum_{\Phi = b, f}\int_0^T |\D\pa_\pi \Phi_s\td_x \Pi^{x^1, e_i}_s|ds\Big)^2 \nonumber\\
&&\qq\qq\qq  + \int_0^T \big(|\D \pa_x \si_s|^2 |\td_x X^{x^1, e_i}_s|^2 +|\D \pa_y \si_s|^2 |\td_x Y^{x^1, e_i}_s|^2\big) ds\Big].\nonumber
\eea
Note that, for $\Phi = b, \si, f, g$,
\beaa
&&\dbE\Big[|\D \pa_x \Phi_s|^2 |\td_x X^{x, e_i}_s|^2\Big]\le C\Big(\dbE\big[ |\D \pa_x \Phi|^4\big]\Big)^{1\over 2} \Big(\dbE\big[ |\td_x X^{x,e_i}_s|^4\big]\Big)^{1\over 2};\\
&&\dbE\Big[|\D \pa_y \Phi_s|^2 |\td_x Y^{x, e_i}_s|^2\Big]\le C\Big(\dbE\big[ |\D\pa_y \Phi |^4\big]\Big)^{1\over 2} \Big(\dbE\big[ |\td_x Y^{x,e_i}_s|^4\big]\Big)^{1\over 2};\\
&&\dbE\Big[\Big(\int_0^T|\D \pa_z \Phi_s| |\td_x Z^{x, e_i}_s|ds \Big)^2\Big]\le \dbE\Big[\int_0^T|\D \pa_z \Phi_s|^2ds \int_0^T |\td_x Z^{x, e_i}_s|^2ds \Big]\\
&&\q \le C\Big(\dbE\big[ \int_0^T |\D \pa_z \Phi|^4\big]ds\Big)^{1\over 2} \Big(\dbE\Big[ \big(\int_0^T |\td_x Z^{x, e_i}_s|^2ds\big)^2\Big]\Big)^{1\over 2}.
\eeaa
Then, by \reff{Lpest1} and denoting $\pi=(x, y, z)$, we have
\bea
\label{Dpsi}
|\psi_i(x^1) - \psi_i(x^2)|^4 \le \|\D^i \td_x \Pi\|_2^4\le C \dbE\Big[|\D \pa_x g|^4 + \sum_{\Phi = b, \si, f}\int_0^T|\D \pa_\pi \Phi_s|^4ds\Big].
\eea
Let $\rho^0_R$ denote the modulus of continuity function in \reff{rhoR1}. Recall that $\pa_\pi \Phi$ is bounded by $L_\Phi$. Then, for $\Phi = b, \si, f$ and for any $R>0$,
\beaa
|\D \pa_\pi \Phi| &\le& 2L_\Phi \1_{\{|\Pi^{x^1,e_i}|\vee |\Pi^{x^2,e_i}| > R\}} + 2 L_\Phi \1_{\{|\D \Pi|\ge \sqrt{|\D x|}\}} + \rho^0_R(\sqrt{|\D x|})\\
&\le&{C\over \sqrt{R}} \Big[\sqrt{|\Pi^{x^1,e_i}|} + \sqrt{|\Pi^{x^2,e_i}|}\Big] + {C\over |\D x|^{1\over 4}} \sqrt{|\D \Pi|} + \rho^0_R(\sqrt{|\D x|}),
\eeaa
and similar estimate holds for $\D \pa_x g$. Then, by \reff{Piest},
\beaa
|\psi_i(x^1) - \psi_i(x^2)|^4 &\le& C \dbE\Big[{1\over R^2}[ |X^{x^1,e_i}_T|^2 + |X^{x^2,e_i}_T|^2] + {1\over |\D x|} |\D X_T|^2 + |\rho^0_R(\sqrt{|\D x|})|^4  \\
&&+ \int_0^T\big[{1\over R^2}[ |\Pi^{x^1,e_i}_s|^2 + |\Pi^{x^2,e_i}_s|^2] + {1\over |\D x|} |\D \Pi_s|^2 + |\rho^0_R(\sqrt{|\D x|})|^4\big]ds\Big]\\
&\le& C\Big[{R_1^2\over R^2} + |\D x| + |\rho^0_R(\sqrt{|\D x|})|^4\Big] .
\eeaa
Then, for $|x^1|, |x^2|, I_0\le R_1$,
\bea
\label{rhoR3}
|\psi_i(x^1) - \psi_i(x^2)| \le \rho_{R_1}(|\D x|) := \inf_{R>0} C\Big[{\sqrt{R_1}\over \sqrt{R}} + |\D x|^{1\over 4} +\rho^0_R(\sqrt{|\D x|})\Big].
\eea
It is clear that $\lim_{|\D x|\to 0}  \rho_{R_1}(|\D x|)  = 0$. That is, $ \rho_{R_1}$ is a modulus of continuity function.

{\bf Step 4.} In this step we  show that $u$ is differentiable in $x$ with the representation \reff{paxrep}. Indeed, fix $i=1,\cds, d_x$, and for any $\e\in \dbR$, denote $x^{i,\e} := x + \e e_i$. It is clear that $\lim_{\e\to 0} \|\Pi^{x^{i,\e}}-\Pi^x\|_2 = 0$. Denote $\td^{i,\e}_x \Pi := {1\over \e}  (\Pi^{x^{i,\e}}-\Pi^x)$. We can formally write down the FBSDEs for $\td^{i,\e}_x \Pi$ and compare it with \reff{tdFBSDEx}. Following the second inequality in \reff{Dpsi} we can easily show that $\lim_{\e\to 0} \|\td^{i,\e}_x \Pi - \td_x \Pi^{x,e_i}\|_2 =0$. In particular, this implies that
\beaa
\lim_{\e\to 0} {1\over \e} \big[u_0(x^{i,\e}) - u_0(x)\big] = \lim_{\e\to 0} \td^{i,\e}_x Y_0 = \td_x Y^{x, e_i}_0 = \psi_i(x).
\eeaa
This implies $\pa_{x_i} u_0(x) = \psi_i(x)$. Then \reff{paxrep} holds, $|\pa_x u_0|\le R_0$, and $\pa_x u_0$ is locally uniformly continuous in $x$. Similarly, for any $t$, we have $|\pa_x u_t|\le R_0$, and $\pa_x u_t$ is locally uniformly continuous in $x$. Moreover, since
\beaa
\lim_{\e\to 0} \td^{i,\e}_x X_t = \td_x X^{x,e_i}_t,~  \td_x Y^{x,\e_i}_t = \lim_{\e\to 0} \td^{i,\e}_x Y_t  = \lim_{\e\to 0} {u_t(X^{x^{i,\e}}_t) - u_t(X^x_t)\over \e}= \pa_x u_t(X^x_t)  \td_x X^{x,e_i}_t.
\eeaa
This is the second equality in \reff{randomFK}, and we see that $\Si_t =  \pa_x u_t(X^x_t)$.

{\bf Step 5.} We now show that $\pa_x u$ is locally uniformly continuous in $t$ in the sense of \reff{rhoR2}. Fix $R_1\ge I_0$ and $x\in \dbR^{d_x}$ with $|x|\le R_1$, and denote $\td_x \Pi^i := \td_x \Pi^{x,e_i}$. By the above step, we have
\beaa
&&\!\!\!\!\! \big|\pa_{x_i} u_0(x) - \dbE[\pa_{x_i} u_t(x)]\big| = \Big|\dbE\Big[\td_x Y^{i}_0 - \td_x Y^{i}_t + \pa_{x} u_t(X^x_t) \td_x X^{i}_t - \pa_{x_i} u_t(x)\Big]\Big|\\
&&\!\!\!\!\!\le \dbE\Big[\int_0^t |\pa_\pi f_s \td_x \Pi^{i}_s|ds + \big|\pa_{x} u_t(X^x_t)- \pa_{x} u_t(x)\big| |\td_x X^{i}_t| + |\pa_{x} u_t(x)|\big| \td_x X^{i}_t - e_i\big|\Big].
\eeaa
Note that $ \|\Pi^x\|_2 \le CR_1$, $\|\td_x \Pi^i\|_2 \le C$, for the norm in \reff{Lpest1}. Then
\beaa
&&\dbE\Big[\int_0^t |\td_x \Pi^{i}_s|ds\Big] \le \Big( \dbE\Big[ t\int_0^t |\td_x \Pi^{i}_s|^2ds\Big] \Big)^{1\over 2} \le C\sqrt{t};\\
&&\dbE\Big[\Big|\int_0^t \pa_\pi \si_s(\Pi^x_s) \td_x \Pi^{i}_sdB_s\Big|\Big] \le C\Big( \dbE\Big[ \int_0^t [|\td_x X^{i}_s|^2 + |\td_x Y^{i}_s|^2]ds\Big] \Big)^{1\over 2} \le C\sqrt{t};\\
&&\dbE\Big[|X^x_t-x|^2\Big] \le C\dbE\Big[t\int_0^t |b_s(\Pi^x_s)|^2ds + \int_0^t |\si_s(X^x_s, Y^x_s)|^2ds\Big]\le CR_1^2 t.
\eeaa
Thus, for any $R>R_1$ and for $\rho_{R}$ in \reff{rhoR3}, we have
\beaa
&&\big|\pa_{x_i} u_0(x) - \dbE[\pa_{x_i} u_t(x)]\big|^2 \le Ct + C\dbE\Big[\big|\pa_{x_i} u_t(X^x_t)- \pa_{x_i} u_t(x)\big|^2\Big]\\
&&\le Ct +C\dbE\Big[\1_{\{|X^x_t|\ge R\}} + \1_{\{|X^x_t - x| \ge t^{1\over 4}\}} + \1_{\{I_t \ge R\}} + |\rho_{R}(t^{1\over 4})|^2\Big]\\
&&\le Ct +C\dbE\Big[{1\over R^2}|X^x_t|^2 + {1\over \sqrt{t}}|X^x_t - x|^2  + {1\over R^2} I_t^2+ |\rho_R(t^{1\over 4})|^2\Big] \\
&&\le Ct + C\Big[{R_1^2\over R^2} + R_1^2\sqrt{t} +|\rho_R(t^{1\over 4})|^2\Big],
\eeaa
where we used the obvious fact that $\dbE[I_t^2] \le I_0^2 \le R_1^2$. Then, for $|x|, I_0\le R_1$,
\beaa
\big|\pa_{x_i} u_0(x) - \dbE[\pa_{x_i} u_t(x)]\big| \le \bar\rho_{R_1}(t) := \inf_{R>R_1}  C\Big[\sqrt{t}+ {R_1\over R} + R_1t^{1\over 4} +\rho_R(t^{1\over 4})\Big].
\eeaa
Similarly, for any $t_1<t_2$ and $|x|, I_{t_1}\le R_1$,
\beaa
\big|\pa_{x_i} u_{t_1}(x) - \dbE_{t_1}[\pa_{x_i} u_{t_2}(x)]\big| \le \bar\rho_{R_1}(t_2-t_1).
\eeaa
It is clear that $\lim_{|\D t|\to 0}  \bar\rho_{R_1}(|\D t|)  = 0$, and thus $ \bar\rho_{R_1}$ is also modulus of continuity function.

{\bf Step 6.} It remains to prove the last equality in \reff{randomFK}, namely the representation of $Z^{x,1}$. We prove it in three cases.

{\bf Case 1.} We first assume  the coefficients depend on $B^0$ in a state dependent manner:
\bea
\label{bsifg-state}
\Phi_t(\pi) = \Phi'(t, \pi, B^0_t),~ \Phi = b, \si, f,\qq g(x) = g'(x, B^0_T),
\eea
where $\Phi', g'$ are deterministic functions and twice continuously differentiable in $(\pi, x')$  with bounded derivatives, where $x'$ refers to the last variable corresponding to $B^0_t$. Then, by the Markovian structure \reff{bsifg-state}, it follows from the standard FBSDE theory that $Y^x_t = u'(t, X^x_t, B^0_t)$ for a deterministic function $u'$.
It is clear that  $u_t(x, \o) := u'(t, x, B^0_t(\o))$. We claim that $u'\in \cC^2_b$. Then by the standard It\^o formula we see that
\beaa
Z^{x,1}_t = \pa_x u'(t, X^x_t, B^0_t)  \si^1(t, X^x_t, Y^x_t, B^0_t) = \pa_x u_t(X^x_t)  \si^1_t(X^x_t, Y^x_t).
\eeaa

We now show that $u'\in \cC^2_b$, following similar but much simpler arguments as in Section \ref{sect-bootstrap}. First, from the arguments in the previous steps, one can easily see that $\pa_x u', \pa_{x'} u'$ exist and are locally uniformly continuous in $(t, x, x')$. Moreover, by differentiating FBSDE \reff{tdFBSDEx} further, following similar arguments as in the previous steps, one can show that $u'$ is twice differentiable in $(x, x')$ such that the second order derivatives are also bounded and locally uniformly continuous.

Fix $\e>0$ small and denote
\beaa
\cY^\e_t := Y^x_t - u'(\e, x, 0), \q \cZ^{\e,1}_t := Z^{x,1}_t - \pa_x u'(\e, X^x_t, B^0_t) \si^1(t, X^x_t, Y^x_t, B^0_t),\\
 \cZ^{\e,0}_t := Z^{x,0}_t - \pa_x u'(\e, X^x_t, B^0_t) \si^0(t, X^x_t, Y^x_t, B^0_t) - \pa_{x'} u'(\e, X^x_t, B^0_t).
\eeaa
Then, for $t\in [0, \e]$, by applying the standard It\^{o} formula on $u'(\e,\cd,\cd)$ we have
\beaa
\cY^\e_t &=& u'(\e, X^x_\e, B^0_\e) - u'(\e, x, 0) + \int_t^\e f(s, \Pi^x_s, B^0_s) ds - \int_t^\e Z^x_s dB_s\\
& =& \int_t^\e \Big[ {1\over 2}\pa_{xx} u' : \si\si^\top + {1\over 2} \tr(\pa_{x'x'} u') + \pa_{x x'} u' : \si + \pa_x u' \cd b+ f\Big] ds - \int_t^\e \cZ^\e_s dB_s,
\eeaa
where the variables inside the derivatives of $u'$ is $(\e, X^x_s, B^0_s)$, while the variables inside $b, \si, f$ are $(s, \Pi^x_s, B^0_s)$. By the boundedness of the derivatives of $u'$, and the desired regularity of $\Phi'$, one can easily see that, for a constant $C$ independent of $\e$,
\beaa
\dbE\Big[\sup_{0\le t\le \e} |\cY^\e_t|^2 + \int_0^\e |\cZ^\e_t|^2dt\Big] \le C\e^2.
\eeaa
This implies that
\bea
\label{Z1est}
\dbE\Big[\int_0^\e |\cZ^{\e,1}_s|ds\Big] \le\Big( \dbE\Big[\big(\int_0^\e |\cZ^{\e,1}_s|ds\big)^2\Big]\Big)^{1\over 2}\le \Big( \dbE\Big[\e \int_0^\e |\cZ^{\e,1}_s|^2ds\Big]\Big)^{1\over 2}\le C \e^{3\over 2}.
\eea

Now for $n$ large and let $0=t_0<\cds<t_n=T$ be the uniform partition. Then, similar to \reff{Z1est}, we can show that
\beaa
\sum_{i=1}^{n}\dbE\Big[\int_{t_{i-1}}^{t_i} |Z^{x,1}_t - \pa_x u'(t_i, X^x_t, B^0_t) \si^1(t, X^x_t, Y^x_t, B^0_t)|dt\Big]\le C\sum_{i=1}^n ({T\over n})^{3\over 2} \le {C\over \sqrt{n}}.
\eeaa
Thus
\beaa
&&\dbE\Big[\int_0^T |Z^{x,1}_t - \pa_x u'(t, X^x_t, B^0_t) \si^1(t, X^x_t, Y^x_t, B^0_t)|dt\Big]\\
&&\le \sum_{i=1}^{n}\dbE\Big[\int_{t_{i-1}}^{t_i} | \pa_x u'(t_i, X^x_t, B^0_t)  - \pa_x u'(t, X^x_t, B^0_t)|| \si^1(t, X^x_t, Y^x_t, B^0_t)|dt\Big] + {C\over \sqrt{n}}.
\eeaa
Send $n\to \infty$, by the locally uniform continuity of $\pa_x u'$ in $t$ and applying the dominated convergence theorem, we obtain
\beaa
\dbE\Big[\int_0^T |Z^{x,1}_t - \pa_x u'(t, X^x_t, B^0_t) \si^1(t, X^x_t, Y^x_t, B^0_t)|dt\Big] = 0.
\eeaa
This implies the last equality in \reff{randomFK}.

{\bf Case 2.} We next relax \reff{bsifg-state} to the discrete path dependence on $B^0$:
\bea
\label{bsifg-staten}
\Phi_t(\pi) = \Phi'(t, \pi, B^0_{t_1: t_i}),~ t\in (t_i, t_{i+1}],~ \Phi = b, \si, f,\q\mbox{and}\q g(x) = g'(x, B^0_{t_1:t_n}),
\eea
where $0=t_0<\cds<t_n=T$ is a time partition, $B^0_{t_1:t_i} := (B^0_{t_1},\cds, B^0_{t_i})$, and $\Phi'$ are deterministic functions and twice continuously differentiable in all the spatial variables with bounded derivatives. Then, one can easily see that $u_t(x) = u^{'i}(t, x, B^0_{t_1:t_i})$, $t\in [t_i, t_{i+1}]$, for deterministic functions $u^{'i}: [t_i, t_{i+1}] \times \dbR^{d_x} \times (\dbR^{d_{01}})^{i+1}$, $i=0,\cds, n-1$, such that
\beaa
u^{'(n-1)}(t_n, x, x'_{1:n}) = g'(x, x'_{1:n}),\q u^{'i-1}(t_{i}, x, x'_{1:i})= u^{'i}(t_{i}, x, x'_{1:i}, x'_i),~i=n-1,\cds, 1,
\eeaa
where $x'_{1:i} = (x'_1, \cds, x'_i)$.
Then, applying the result in Case 1 backwardly in time, we have
\beaa
 Z^{x,1}_t = \pa_x u^{'i}  \si^1(t, X^x_t, Y^x_t, B^0_{t_1:t_{i}}) = \pa_x u_t(X^x_t) \si^1_t(X^x_t, Y^x_t),\q t\in (t_i, t_{i+1}).
\eeaa
This is the desired result.

\ms
{\bf Case 3.} We  now consider the general path dependence on $B^0$. For $n\ge 1$, let $0=t_0<\cds<t_n=T$ be the uniform partition. By standard arguments, one can easily construct $\Phi^n$ for $\Phi=b,\si, f, g$ such that
\begin{itemize}
\item For each $n$, $\Phi^n$ satisfies the requirements in Case 2.

\item $\Phi^n$ satisfies Assumption \ref{assum-0} uniformly, uniformly in $n$.

\item $\Phi^n$ converges to $\Phi$ in the sense that, for any $t, \pi$,
\bea
\label{PhiB0nconv}
\left.\ba{c}
\dis \lim_{n\to\infty} \dbE\Big[| \Phi^n_t(\pi) - \Phi_t(\pi)|^2+ | \pa_\pi\Phi^n_t(\pi) -\pa_\pi\Phi_t(\pi)|^2\Big]=0,\q \Phi=b,\si, f,\ms\\
\dis \lim_{n\to\infty} \dbE\Big[| g^n(x) - g(x)|^2 + | \pa_x g^n(x) - \pa_x g(x)|^2 \Big]=0.
\ea\right.
\eea
\end{itemize}

Let $\e_0$ be as in Step 1 and $T\le \e_0$. Since $\Phi^n$ satisfies Assumption \ref{assum-0} uniformly, then FBSDEs \reff{FBSDEx} and \reff{tdFBSDEx} with coefficients $\Phi^n$ are also well-posed, with solutions denoted as $(\Pi^{n,x}, \td_x \Pi^n)$ and the decoupling field denoted as $u^n$. By Case 2 we have
\bea
\label{Zn1}
Z^{n,x,1}_t = \pa_x u^n_t(X^{n,x}_t) \si^1_t(X^{n,x}_t, Y^{n,x}_t).
\eea
 By \reff{PhiB0nconv}  we have the convergence:
\bea
\label{FBSDEnconv}
\lim_{n\to\infty} \Big[\|\Pi^{n,x} - \Pi^x\|_2 + \|\td_x\Pi^n - \td_x\Pi\|_2\Big]=0.
\eea
In particular, this implies that $\pa_{x_i} u^n_0(x) = \td Y^{n,x, e_i}_0 \to  \td Y^{x, e_i}_0 = \pa_{x_i} u_0(x)$. Similarly $\pa_x u^n_t \to \pa_x u_t$. Moreover, again since $\Phi^n$ satisfies Assumption \ref{assum-0} uniformly, then $\pa_x u^n$ are uniformly bounded, and are locally uniformly continuous with the modulus of continuity function $\rho_R$ in \reff{rhoR2} independent of $n$. Then, for any $R>0$ and $\e>0$,
\beaa
&&\big|\pa_x u^n_t(X^{n,x}_t) - \pa_x u_t(X^{x}_t)\big|\le \big|\pa_x u^n_t(X^{n,x}_t) - \pa_x u^n_t(X^{x}_t)\big| + \big|\pa_x u^n_t(X^{x}_t) - \pa_x u_t(X^{x}_t)\big|\\
&&\le C\big[\1_{\{|X^{n,x}_t|\ge R\}} + \1_{\{|X^{x}_t|\ge R\}} + \1_{\{|X^{n,x}_t-X^x_t|\ge \e\}}\big] + \rho_R(\e) +  \big|\pa_x u^n_t(X^{x}_t) - \pa_x u_t(X^{x}_t)\big|.
\eeaa
Then
\beaa
&&\!\!\!\!\! \dbE\Big[\int_0^T \big|\pa_x u^n_t(X^{n,x}_t) \si^1_t(X^{n,x}_t, Y^{n,x}_t)-\pa_x u_t(X^{x}_t) \si^1_t(X^{x}_t, Y^{x}_t)\big|dt\Big]\\
&&\!\!\!\!\!\le C\dbE\Big[\int_0^T \big|\si^1_t(X^{n,x}_t, Y^{n,x}_t) -\si^1_t(X^{x}_t, Y^{x}_t)\big|dt\\
&&\!\! +\int_0^T \big[ {|X^{n,x}_t|+|X^x_t|\over R} + {|X^{n,x}_t-X^x_t|\over \e} + \rho_R(\e) + \big|\pa_x u^n_t(X^{x}_t) - \pa_x u_t(X^{x}_t)\big|\big] | \si^1_t(X^{x}_t, Y^{x}_t)|dt\Big].
\eeaa
This implies that
\beaa
&&\limsup_{n\to\infty} \dbE\Big[\int_0^T \big|\pa_x u^n_t(X^{n,x}_t) \si^1_t(X^{n,x}_t, Y^{n,x}_t)-\pa_x u_t(X^{x}_t) \si^1_t(X^{x}_t, Y^{x}_t)\big|dt\Big]\\
&&\le C \dbE\Big[\int_0^T \big[ {|X^x_t|\over R} +  \rho_R(\e)\big]dt\Big] \le C\big[{I_0+|x|\over R} + \rho_R(\e)\big].
\eeaa
By first sending $\e\to 0$ and then $R\to \infty$, we obtain
\beaa
\lim_{n\to\infty} \dbE\Big[\int_0^T \big|\pa_x u^n_t(X^{n,x}_t) \si^1_t(X^{n,x}_t, Y^{n,x}_t)-\pa_x u_t(X^{x}_t) \si^1_t(X^{x}_t, Y^{x}_t)\big|dt\Big]=0.
\eeaa
Moreover, since $\dbE\big[\int_0^T |Z^{n,x,1}_t-Z^{x,1}_t|^2dt\big]\to 0$, then by \reff{Zn1} we obtain the claimed representation:
 $Z^{x,1}_t = \pa_x u_t(X^{x}_t) \si^1_t(X^{x}_t, Y^{x}_t)$.
 \qed

\end{document}